\documentclass[12pt]{amsart}
\usepackage[margin=1in]{geometry}
\usepackage{amsmath, amssymb, amsthm,latexsym,overpic}
\usepackage[all]{xy}
\usepackage{graphicx}
\usepackage{enumerate,mathtools,comment}
\usepackage{enumitem}
\usepackage{color}
\usepackage[pdftex]{hyperref}

\title[Quasi-visual approximations]{Quasi-visual approximations}%
\author{Mario Bonk}
\thanks{M.B.\ was partially supported by NSF grant DMS-2054987. M.H.\ was partially supported by the Marie Sk\l{}odowska-Curie Postdoctoral Fellowship under the EU's Horizon Europe Programme (Grant No.\ 101068362).}

\address{Department of Mathematics, University of California, 
Los Angeles, CA 90095, USA}
\email{mbonk@math.ucla.edu}

\author{Mikhail Hlushchanka}
\address{Korteweg-de Vries Instituut voor Wiskunde, Universiteit van Amsterdam,  1090 GE \newline Amsterdam, The Netherlands}
\email{mikhail.hlushchanka@gmail.com}

\author{Daniel Meyer}
\address{Department of Mathematical Sciences, University of Liverpool,
Mathematical Sciences Building,  Liverpool L69 7ZL,  United Kingdom}
\email{dmeyermail@gmail.com}

\date{\today}


\setenumerate{itemsep=3pt,topsep=3pt}
\setenumerate[1]{label=\upshape(\roman*)}

\SetLabelAlign{parright}{\smash{\parbox[t]{\labelwidth}{\raggedleft#1}}}

\newcommand\C{{\mathbb C}}
\newcommand\D{{\mathbb D}}
\newcommand\CDach{\widehat{{\mathbb C}}}

\newcommand\N{{\mathbb N}}

\newcommand\R{{\mathbb R}}

\newcommand\dist{\operatorname{dist}}
\newcommand\diam{\operatorname{diam}}
\newcommand\hdiam{\operatorname{hdiam}}
\newcommand\id{\operatorname{id}}

\renewcommand\:{\colon}
\newcommand\sub {\subset}
\newcommand\ra {\rightarrow}

\newcommand\Ga{\Gamma}

\newcommand\ga{\gamma}
\newcommand\la{\lambda}
\newcommand\eps{\epsilon}

\providecommand{\abs}[1]{\lvert#1\rvert}

\DeclarePairedDelimiter\ceil{\lceil}{\rceil}

\newcommand{\X} {\mathbf{X}}

\newcommand{\I} {\mathbf{I}}

\newcommand{\V} {\mathbf{V}}

\newcommand{\Jul}{\mathcal{J}}

\newcommand{\comb}{{\tt cv}}

\newcounter{main}

\newtheorem{theorem}{Theorem}[section]

\newtheorem{proposition}[theorem]{Proposition}
\newtheorem{cor}[theorem]{Corollary}

\newtheorem{lemma}[theorem]{Lemma}

\theoremstyle{definition}
\newtheorem{example}[theorem]{Example}
\newtheorem{definition}[theorem]{Definition}

\theoremstyle{remark}
\newtheorem*{claim}{Claim}
\newtheorem{remark}[theorem]{Remark}

\numberwithin{equation}{section}

\begin{document}

\begin{abstract}
We develop the foundations of the  theory of {\em quasi-visual approximations} of  bounded metric spaces. Roughly speaking, these are sequences of covers of a given space for which  the diameters of the sets in the covers  shrink to  zero  and for which relative metric quantities (such as ratios of diameters and distances) are uniformly controlled. This  framework has applications to questions in quasiconformal geometry.  In particular, quasi-visual approximations can be used to detect whether a given homeomorphism between two bounded metric spaces is a quasisymmetry. 

We also explore the 
connection to the theory of Gromov hyperbolic spaces via the {\em tile graph}  associated with  a quasi-visual approximation. 
As an application, we relate these ideas to  the dynamics of semi-hyperbolic rational maps.  More specifically, we show that the Julia set of a rational map admits a \emph{dynamical} quasi-visual approximation if and only if the map is semi-hyperbolic.  
\end{abstract}


\maketitle

\tableofcontents

\section{Introduction}
\label{sec:intro}
In this paper, we are interested in certain aspects of  the geometry of metric spaces. For convenience, we will stipulate that the metric spaces under consideration are bounded, but one could
build up a similar theory as developed here also in greater generality at the cost of some technicalities. 

It is well known  that  the topology and geometry of a (bounded) metric  space $S$   can often be described by 
 a sequence $\{\X^n\}$ of typically finite covers $\X^n$ of $S$ that become finer as $n\to \infty$. In this case, a finite data set associated with each $\X^n$---such as the incidence relations among the sets in $\X^n$ and their metric properties (for example, diameters and mutual distances)---provides a coarse ``snapshot'' of the given space. We can then hope for higher and higher ``resolutions" of these snapshots as $n$ increases. An obvious advantage of this approach is that the relevant information becomes finitary: each cover $\X^n$ is encoded in a finite data set, and the entire sequence $\{\X^n\}$ corresponds to a countable data set.

 This idea is well known and  has proved useful in many contexts. 
 Depending on  the concrete setting, one can hope to find conditions on the incidence relations and the metric properties  of the sets in  the covers $\X^n$ that faithfully capture the relevant 
  features of $S$. 
  
   In the present paper, we focus on the {\em quasiconformal 
  geometry} of bounded metric spaces. Roughly speaking, this term refers to geometric properties that are of a robust scale-invariant nature. These properties are characterized by conditions that do not involve absolute distances, but rather relative distances, that is, ratios of distances. They appear naturally when one investigates the geometry of self-similar spaces, in particular fractals that arise in some dynamical settings such as limits sets of Kleinian groups, Julia sets of rational maps, or attractors of iterated function systems.
  
  Often one wants to describe and  characterize these spaces up to a natural equivalence given by homeomorphisms with good geometric control over relative 
  distances. The relevant  mappings in this context are {\em quasisymmetries}, which provide such a distortion control (for the definition, see Section~\ref{sec:quasisymmetries}; general background on quasisymmetries and related concepts can be found in [He01]). 

So one can say more precisely that the  quasiconformal geometry of a metric space 
is described by those  geometric properties that are preserved under quasisymmetries. In particular, 
we call two metric spaces $S$ and $T$ {\em quasisymmetrically equivalent} if there exists a quasisymmetric homeomorphism from $S$ onto $T$. Obviously, this notion of equivalence is stronger than topological equivalence, but weaker than, say, 
 the requirement that $S$ and $T$ are isometric. 

Now the question arises as to which conditions one should impose on a
 sequence $\{\X^n\}$ of finite covers  of $S$  so that the quasiconformal geometry of $S$  is faithfully represented. We hope to convince the reader that the relevant properties are condensed in the notion of a {\em quasi-visual approximation} $\{\X^n\}$ 
 of a bounded metric space $S$ (see Definition~\ref{def:qv_approx}). 
 
In the following subsections of this introduction, we will present the key 
 topics, concepts, and results of this paper in an informal way, trying to 
 explain the underlying ideas without going into  technical details. 
We will point out where  precise definitions and statements of theorems can be found in the body of the paper.

\subsection{Visual approximations} The notion of a quasi-visual approximation of a  metric space $S$ evolved from a stronger concept, namely that of a  {\em visual approximation}. To describe this in intuitive terms, we first fix some terminology. 

Suppose  $\{\X^n\}$ is a sequence of covers of a (bounded) metric space $S$.
 We call an element $X$ of the cover $\X^n$ a {\em tile of level $n$} or simply an {\em $n$-tile}. A tile $X$ represents a {\em location} in the given metric space 
 $S$ and its diameter $\diam(X)$ can be thought of as the {\em local scale} represented  by $X$.
 
  A natural condition that captures the idea that 
 the covers $\X^n$ become finer as $n\to \infty$ is to impose a uniform exponential decay of the size of  tiles with their level $n$; in other words,  to require 
 \begin{equation}\label{eq:diamabs} \diam(X)\asymp \Lambda^{-n} 
 \end{equation}
 for $X\in \X^n$ with some fixed parameter $\Lambda>1$ 
(for the precise meaning of the symbol $\asymp$ used here and the symbols  $\lesssim$ and $\gtrsim$  used later in this introduction, see Section~\ref{sec:notation}).
 In a similar spirit, one can require that combinatorial separation of tiles $X$ and $Y$ on the same level 
 $n$   translates to the metric separation 
 \begin{equation}\label{eq:combsep}
 \dist(X,Y)\gtrsim \Lambda^{-n}.
 \end{equation}
 
For the purpose of this introduction,  the reader should think of 
combinatorial separation simply as the condition $X\cap Y=\emptyset$. It turns out though that  working  with a stronger condition for combinatorial separation leads to a more satisfactory and complete theory. 
This stronger separation property involves a {\em width} parameter  $w\in \N_0$ (see Definition~\ref{def:visual}). 
It  makes the corresponding requirement \eqref{eq:combsep} weaker and more flexible.

 If we impose the two conditions \eqref{eq:diamabs} and \eqref{eq:combsep} on the sequence $\{\X^n\}$, then we are led to the concept of a {\em visual approximation} of a metric space $S$ (see Definition~\ref{def:visual} for the precise formulation based on the stronger separation condition). A  (bounded) metric $d$ on $S$  for which such 
 a visual approximation $\{\X^n\}$ of $(S,d)$ exists is called a  {\em visual metric} on $S$.
 
 We develop this theory in Section~\ref{sec:visu-appr}. One of our main results here is that {\em  a bounded metric space $S$ has a visual approximation if and only if it is uniformly perfect} (see Proposition~\ref{prop:vis_width1_ex}). The condition on $S$---namely, that it is uniformly perfect  (see Section~\ref{sec:prop-metric-spaces})---is  very natural from the viewpoint of quasiconformal geometry, as this condition is invariant under quasisymmetries and is satisfied in many interesting settings.

\subsection{Quasi-visual approximations}
 
The notion of a visual approximation, together with the associated concept of a visual metric, has proved useful in the study of dynamical systems (for example, in the theory of expanding Thurston maps; see \cite{BM}). However, it is not entirely in the spirit of quasiconformal geometry as the local scales and the metric separation of tiles are controlled by {\em absolute} distances.

To develop a notion of an  approximation of a metric space adapted to its quasiconformal geometry, one should weaken the conditions so that only relative distances are involved. 
This is exactly the idea of a quasi-visual approximation $\{\X^n\}$ 
 of a bounded metric space $S$  as introduced in Definition~\ref{def:qv_approx}.
 
 It consists of four conditions that can  be described as follows.
 Two of these conditions (\ref{item:qv_approx1} and \ref{item:qv_approx3} in 
 Definition~\ref{def:qv_approx}) essentially say that the local 
 scale $\diam(X)$ represented by an $n$-tile $X$ is more or  less well-defined.
Namely, if a $k$-tile $Y$ is in the same location as  $X$ (that is, $X\cap Y\ne \emptyset$) and  has about the same level (that is, $|k-n|\lesssim 1$), then $\diam(X)\asymp \diam(Y)$. 

We also require  uniform  exponential decay of the diameters of $k$-tiles 
$Y$ in the same location as $X$ measured {\em relative} to the local scale, that is,
\[
\diam(Y)\lesssim \rho^{k-n}\diam(X)
\]
for  $k\ge n$ with a uniform parameter $\rho \in (0,1)$ (in \ref{item:qv_approx4} of Def\-ni\-tion~\ref{def:qv_approx} this is formulated in a different, but equivalent way; see Lemma~\ref{lem:sub_shrink}). 
Finally, combinatorial separation of two tiles $X$ and $Y$ on the same level 
translates to metric separation, but again relative to the local scale:
\[
\dist(X,Y)\gtrsim \diam(X)
\]
(condition~\ref{item:qv_approx2} in Definition~\ref{def:qv_approx} expresses a weaker version of this requirement again involving  a width parameter).

One of the main reasons why we think that quasi-visual approximations are the right concept to faithfully capture the quasiconformal geometry of a metric 
space is that these approximations are preserved under quasisymmetries and can actually detect whether a homeomorphism is a quasisymmetry.

To formulate this precisely, let $\varphi\: S\ra T$ be a bijection between 
bounded metric spaces $S$ and $T$,  the sequence $\{\X^n\}$ be a quasi-visual approximation of $S$, and $\{\mathbf{Y}^n\}$ be the image of $\{\X^n\}$ under $\varphi$, that is, \[
\mathbf{Y}^n\coloneqq \{ \varphi(X):X\in \X^n\} \text{ for } n\in \N_0.
\]
Then
{\em $\varphi$ is a quasisymmetry if and only if  $\{\mathbf{Y}^n\}$ is a quasi-visual 
approximation of $T$} (see Corollary~\ref{cor:qv-qs}).

\subsection{From metrics to combinatorics and back} 

Once one has a quasi-visual approximation $\{\X^n\}$ of a metric space 
$(S,d)$, 
 its metric information is entirely encoded in the local scales (that is, the diameters 
of tiles) and the combinatorial incidence relations among tiles. To make this precise, we introduce 
a {\em proximity function $m$} that  quantifies the combinatorial separation 
of two distinct points $x$ and $y$ in $S$. Roughly speaking, $m=m(x,y)$ is the largest 
level such that $x$ and $y$ are in the ``same location", meaning that there are $m$-tiles $X$ and $Y$ containing $x$ and $y$, respectively, that are not yet combinatorially separated ($X\cap Y\neq \emptyset$); see Definition~\ref{def:mxy}, which again involves a width parameter that we ignore here for simplicity. At this critical level $m=m(x,y)$, 
the distance $d(x,y)$ is comparable to the local scale $\diam(X)$ of any $m$-tile $X$ with $x\in X$ (see Lemma~\ref{lem:qv_metric}).  

The proximity function $m$ can be defined for any sequence $\{\X^n\}$ of covers 
of a set $S$. If $S$ is a bounded metric space and $\{\X^n\}$
is a quasi-visual approximation of $S$,  then $m$ has  certain characteristic properties. We turn this around and call a sequence $\{\X^n\}$ of covers of a set $S$ 
for which the proximity function $m$  has these properties a {\em combinatorially visual approximation} of $S$ (see Definition~\ref{def:comb_visual}).  It is now very satisfactory that 
for such an approximation  one can always find a metric $d$ on $S$ such that 
 $\{\X^n\}$ becomes a  quasi-visual approximation of $(S,d)$. Actually, with the right choice of $d$  the sequence  $\{\X^n\}$ is even a visual approximation of
 $(S,d)$ (see Theorem~\ref{thm:_comb_qv} for the precise formulation).

\subsection{Tile graphs and Gromov hyperbolic geometry}\label{subsec:tile graph intro}
 
 It will  not be a  surprise to  experts that the theory we develop here is also related to the theory of Gromov hyperbolic spaces  (for general background on Gromov hyperbolic spaces, see \cite{Gr,GH, BH99, BS}). 
 We explore this in Sections~\ref{sec:tile-graph}--\ref{sec:from-grom-hyperb}. The starting point is the {\em tile graph} $\Gamma=
 \Gamma(\{\X^n\})$ associated with a sequence $\{\X^n\}$ of covers  of a given set $S$. By definition, the vertex set of $\Gamma$ consist of all tiles on all levels, and we connect two vertices as represented by tiles $X\in \X^n$ and $Y\in \X^k$ whenever $X\cap Y\ne \emptyset$ and $|n-k|\le 1$.  
 
 We  show that {\em if $\{\X^n\}$ is  a combinatorially visual approximation of a set $S$, then  the associated tile graph $\Gamma$ is Gromov hyperbolic} (see Theorem~\ref{thm:Gromov_hyp}). The proof is based on the close relation of the {\em Gromov product} on the tile graph (see \eqref{eq:def_Gromov}) to the proximity function $m$ discussed earlier (see \eqref{eq:defmXY} and Proposition~\ref{prop:mXY_XY}) 
 In particular, $\Gamma$ is Gromov hyperbolic when $\{\X^n\}$  is a quasi-visual 
 approximation of a metric space $(S,d)$. In this case, 
 $(S,d)$ can be recovered from the associated tile graph $\Gamma$ under the additional assumption that $(S,d)$ is complete. Indeed,
 we will prove that then  $(S,d)$ is  quasisymmetrically equivalent  to the  {\em boundary at infinity $\partial_\infty \Gamma$} of $\Gamma$  equipped with any  visual metric $d_\infty$  (in the sense of Gromov hyperbolic spaces)  via a natural  identification  map $\Phi\: S\ra \partial_\infty \Gamma$   (see 
 \eqref{eq:defnatid} and Proposition~\ref{prop:qv_visual_Gromov_hyp}).  
 
 Finally, one can ask  about the reverse implication in this context.
 Namely, let  $\Gamma=
 \Gamma(\{\X^n\})$ be the tile graph associated with  a sequence $\{\X^n\}$ of covers of a
 metric space $(S,d)$, and suppose $\Gamma$ is Gromov hyperbolic. Can we then conclude  that $\{\X^n\}$ is a quasi-visual or even a visual approximation of $(S,d)$? 

By itself, the Gromov hyperbolicity of $\Gamma$ is not enough to draw any conclusions here. However, if we assume in addition that the natural identification map
 $\Phi\: (S,d)\ra (\partial_\infty \Gamma, d_\infty)$  is nice enough (namely, that it is a quasisymmetry or a snowflake equivalence), then one obtains a version of the desired conclusion:    
 not for the original sequence $\{\X^n\}$ of covers, but for some associated sequence $\{\V^n\}$
of
covers obtained by taking suitable {\em neighborhood clusters} of the original tiles in $\{\X^n\}$  (see \eqref{eq:cluster} and Theorem~\ref{thm:Gromov_implies_qv}).      

\subsection{Quasi-visual approximations of Julia sets}
Fractal spaces arising in various contexts often  admit natural decompositions on different scales  generating   a sequence of covers. It is then a key question  whether these covers actually provide a (quasi-)visual approximation of the given fractal space. For some recursively constructed standard fractals, such as the Cantor ternary set, the Koch snowflake, or the Sierpi\'nski carpet and gasket, the required conditions can easily be verified. However,  in other settings this  may not be so straightforward. To illustrate this point in the present paper, we chose to focus on a specific dynamical setting:  Julia sets of semihyperbolic rational maps. 

Since the Julia set $\Jul(g)$ of a rational map $g$ is always uniformly perfect (see \cite{Hi} and \cite{MdR}),  the existence 
of a (quasi-)visual approximation $\{\X^n\}$  of $\Jul(g)$ is guaranteed by one of our general results (see Proposition~\ref{prop:vis_width1_ex}). It is clear, however, that in the presence of dynamics one wants to impose stronger conditions that tie the iteration of the map and the sequence $\{\X^n\}$ together. Therefore, it is natural to require that the quasi-visual approximation $\{\X^n\}$  
is {\em dynamical} in the sense that the map $g$  sends tiles to tiles with an appropriate shift of the levels (see 
Section~\ref{subsec:dynqvisual} for the relevant 
definition in a general setting).

We will see that such {\em a
dynamical quasi-visual approximation exists
for the Julia set $\Jul(g)$ of a rational map $g$  if and only if $g$ is  semi-hyperbolic} 
(see Corollary~\ref{cor:charsemi}); it follows from Theorems~\ref{thm:semi_hyp_admits_approx} and~\ref{thm:Adconverse} dealing with the two implications of this equivalence).

Our result is closely related to 
work of Ha\"{i}ssinsky and Pilgrim \cite[Theorem~4.2.3]{HP09}, who gave a characterization of semi-hyperbolic rational maps in their framework of {\em metric CXC systems}. We will comment more on the relation to their work in Remark~\ref{rem:relHP}. 

\subsection{Further discussion} The main point of this paper is to build up the foundations of our subject. Accordingly, we have not included more applications,  but in the future we intend to apply the framework of quasi-visual approximations  to the study of the geometry of fractal spaces, dynamical systems, and other areas. To give the reader a glimpse of what such investigations might entail, we explored quasi-visual approximations in the setting of Julia sets of semihyperbolic rational maps, but we expect quasi-visual approximations to naturally arise in many other contexts such as geometric group theory, self-similar groups, iterated function systems, and spaces generated by graph-substitution rules.

Although we will not pursue this here, we anticipate that the framework of quasi-visual approximations will be particularly  useful in the study of the {\em conformal dimension} of a  metric 
space $(S,d)$ (for the definition of this concept and general background, see \cite{MT10}).
There, one is faced with the problem of constructing 
new metrics $\varrho$ on $S$ that are quasisymmetrically equivalent to 
$d$ and so that $(S,\varrho)$ has close to minimal Hausdorff dimension 
over all such metrics. Under suitable conditions, such a metric can be obtained by appropriately redefining local scales. The notion of a quasi-visual approximation 
then makes the problem of verifying the quasisymmetric equivalence between $d$ and $\varrho$ conceptually straightforward: it suffices to show that a given quasi-visual approximation $\{\X^n\}$ of $(S,d)$ remains a quasi-visual approximation for $(S,\varrho)$.

\subsection{Previous work}  Visual and quasi-visual approximations were introduced by these names in \cite{BM22}, but there a separation condition without a width parameter was used. The definitions in \cite{BM22} correspond to ours with width $w=0$. Introducing the width parameter is a small, but crucial advance that leads to a more well-rounded theory. For example, while it is 
rather easy to prove that every uniformly perfect metric space admits a visual approximation of width $w=1$ (see Proposition~\ref{prop:vis_width1_ex} and Remark~\ref{rem:w1}), a corresponding statement for width $w=0$ is much harder to prove and requires additional conditions  on the space (see Corollary~\ref{cor:visual_ex}). 

Visual metrics (again with $w=0$) were introduced in the setting
of \emph{expanding Thurston maps} in \cite[Chapter~8]{BM}. Their distinctive properties formulated  in \cite[Proposition~8.4]{BM} led  to the definition used in the present paper.  The characterization of quasisymmetries via quasi-visual
approximations appears (for width $w=0$) already in
\cite[Proposition~2.7]{BM22}. This was motivated by the
desire to formalize arguments that had appeared in  previous  work
(see \cite[Theorem~5.1]{Me02},  
\cite[Theorem~18.1~(ii)]{BM}, and \cite[Theorem~1.2]{BM20}).

As we mentioned earlier, the idea of encoding the geometry of a (bounded) metric space $S$ in a sequence $\{\X^n\}$  of its covers is not new and has been employed  many times before.  In this context, it is very natural 
to package the sequence $\{\X^n\}$ into a single object $\Gamma=\Gamma(\{\X^n\})$---the tile graph mentioned in Section~\ref{subsec:tile graph intro}. In the literature, versions of $\Gamma$ are  often referred to as the {\em hyperbolic filling} of $S$. We will briefly highlight some important work on this subject without attempting  an exhaustive overview.

The work \cite{BP03} by Bourdon and Pajot has been very influential in this  area. They showed (using different terminology) that the hyperbolic filling
$\Gamma$ of a (doubling and uniformly perfect) bounded metric space $S$  is always Gromov hyperbolic, and they  
used the interplay between $\Gamma$ and $S$ to prove  deep results 
on the analysis of metric spaces. More specifically, they identified the $\ell^p$-cohomology of $\Gamma$ with a Besov space on $S$ and applied this correspondence 
to obtain information on the conformal dimension of $S$. 

Carrasco \cite{Ca13} applied the concept of hyperbolic fillings and 
the geometry of Gromov hyperbolic spaces to study the conformal dimension 
of (doubling and uniformly perfect) bounded metric spaces. He obtained a characterization of conformal dimension in terms of the asymptotic 
behavior  of discrete moduli of annuli. His construction 
of the relevant hyperbolic filling involves numerical parameters that make 
the verification of his conditions somewhat cumbersome. 

This may have been  the motivation for Kigami to recast the whole theory in his own terms \cite{Ki20}. 
In his work, the basic idea is to  approximate a given metric space 
by a sequence of partitions leading to a tree-like structure. By assigning weights to the elements of these partitions, one can then construct new metrics and prove theorems in the same spirit as in \cite{Ca13}. 

Motivated by the dynamics of Kleinian groups and rational maps (and, in particular, by the conformal elevator principle), Ha\"{i}ssinsky and Pilgrim introduced \emph{coarse expanding
  dynamical systems (CXC systems)} in \cite{HP09}. In this setting, they
associate a graph with a (dynamical) sequence of covers of the 
underlying space in 
essentially the same way as our tile graph is  defined  in
Section~\ref{sec:tile-graph}. They prove that this  graph 
is Gromov hyperbolic and use  visual metrics (in the sense of Gromov hyperbolic spaces) on the respective boundary at infinity  to
study the corresponding dynamical system. As we already mentioned above, our
characterization of semi-hyperbolic rational maps was inspired 
by their analogous  characterization of these  maps in the setting of
CXC systems.

\subsection{Structure of the paper}

The paper is organized as follows. In Section~\ref{sec:background} we set up notation and review
some basic definitions that are used throughout this work. Visual approximations are introduced and studied in Section~\ref{sec:visu-appr},  while quasi-visual approximations make their appearance in Section~\ref{sec:quasi-visu-appr}.
We then introduce the proximity function  in Section~\ref{sec:dist-comb} and  briefly discuss dynamical 
quasi-visual approximations in Section~\ref{subsec:dynqvisual}

The crucial relation between quasisymmetries and quasi-visual approximations 
(see Theorem~\ref{thm:qv-qs}) is the topic of Section~\ref{sec:quasi-visu-appr-1}. In Section~\ref{sec:comb-visu-appr} we further explore the connection  between the properties of the proximity function and quasi-visual approximations: we introduce combinatorially visual approximations and establish Theorem~\ref{thm:_comb_qv} relating  the different approximation concepts studied in this paper. In  Sections~\ref{sec:tile-graph}--\ref{sec:from-grom-hyperb} we investigate the close relation between  quasi-visual approximations and the theory of Gromov hyperbolic spaces.

In Section~\ref{sec:semihyp} we show
that the Julia set of a rational map $g$ admits a dynamical
quasi-visual approximation if and only if $g$ is semi-hyperbolic
(see Corollary~\ref{cor:charsemi} which follows from Theorems~\ref{thm:semi_hyp_admits_approx} and~\ref{thm:Adconverse}). The argument is based on some 
auxiliary results, namely a technical distortion estimate 
(Lemma~\ref{lem:dist1}) and some facts about normal families
(Lemmas~\ref{lem:Zalc} and~\ref{lem:GBd}).
In order not to distract from the main narrative in  
 Section~\ref{sec:semihyp}, we delegated the proofs of these facts to an appendix. 

\smallskip
{\bf Acknowledgments.} This work was started while the authors were in
residence at the Mathematical Sciences
Research Institute in Berkeley, California, during the Spring semester of 2022. The authors would also like to thank Zhiqiang Li for raising a question that led to Theorem~\ref{thm:_comb_qv}.

\section{Background}
\label{sec:background}

\subsection{Notation}
\label{sec:notation}

We denote by $\N=\{1, 2, \dots \}$ the set of natural numbers and by
$\N_0= \N \cup \{0\}$ the set of natural numbers including $0$. We write $\C$ for the set of complex numbers, and use the notation $\D \coloneqq \{z\in \C :|z|< 1\}$ for the open unit disk in $\C$,
and $\CDach\coloneqq \C \cup \{\infty\}$ for the Riemann sphere.

The cardinality of a set $X$ is denoted by $\#X$ and the identity map
on $X$ by $\id_X$. If $f\: X \to Y$ is a map between two sets and $W\subset X$, then $f|_W$ stands for the restriction of $f$ to $W$.

We will often consider quantities that agree up to a
multiplicative constant. For the most part, we are
not interested in their precise values. For this reason, we often
hide such constants in the notation.
More precisely, we call two non-negative quantities $a$ and $b$
\emph{(multiplicatively) comparable} if there is a constant $C\geq
1$ (possibly depending on some ambient parameters $\alpha$, $\beta$,  \dots) such that
\begin{equation*}
  \frac{1}{C} a \leq b \leq C a.
\end{equation*}
In this case, we write
\begin{equation*}
  a \asymp b.
\end{equation*}
The constant $C$ is referred to as the \emph{comparability constant}, and it is denoted by $C(\asymp)$ or by $C(\asymp)=C(\alpha,\, \beta, \dots)$ if we want to emphasize its dependence on some ambient parameters
$\alpha$, $\beta$, \dots\,.  Similarly, we
write
\begin{equation*}
  a\lesssim b, \quad\text{or }\quad b \gtrsim a
\end{equation*}
if there is a
constant $C>0$ such that $a \leq C b$ and refer to $C$ as
$C(\lesssim)$ or $C(\gtrsim)$, respectively.

Let $(S,d)$ be a metric space. For $x\in S$ and $r>0$, we denote by 
\[B_d(x, r) \coloneqq \{y\in S : d(x, y) < r\}  \text{ and }  \overline{B}_d(x, r) \coloneqq \{y\in S : d(x, y) \leq r\}\] the open and closed balls of radius $r$ centered at $x$, respectively. If $X, Y\subset S$, we
write $\diam_d(X)$ for the diameter of $X$, and
\[\dist_d(X,Y) \coloneqq \inf \{d(x,y) : x\in X, \, y\in Y\}\] for the distance between $X$ and $Y$. We suppress the subscript $d$ in our notation for $B_d(x,r)$, etc., if the
metric is clear from the context. 

\subsection{Properties of metric spaces}
\label{sec:prop-metric-spaces}
Let $(S,d)$ be a metric space. We call $(S,d)$ \emph{uniformly perfect}, if
there is a constant $0<\lambda<1$ (called the \emph{uniform perfectness constant} of $(S,d)$) such that the following statement is true: for all $x\in S$ and  $0<r\leq \diam(S)$ we have 
\[
B(x,r)\setminus \overline{B(x,\lambda r)}\ne \emptyset.\]
Note that $\lambda r \leq \diam(B(x,r)) \leq 2r$ in this case. Every connected metric space  is uniformly perfect.

The space $(S,d)$ is called \emph{doubling} if there is a constant $N\in
\N$ (called the \emph{doubling constant} of $(S,d)$) such that
every open ball $B\subset S$ of radius $r>0$ can be covered by  $N$ (or fewer) open balls of radius $r/2$. By induction, this implies that $B$ can be covered by $N^k$ open balls of radius $2^{-k}r$ for all $k\in \N_0$.

We say that a set $V\subset S$ is  \emph{$\delta$-separated} for a $\delta >0$ if
$d(x,y)\geq \delta$ for all distinct $x,y\in V$. Zorn's lemma implies that every metric space contains a \emph{maximal} $\delta$-separated subset (for any fixed 
$\delta >0$).

The doubling
property can be expressed in terms of the size of separated sets,
as the following well-known lemma shows.

\begin{lemma}
  \label{lem:doubling_separated}
  A metric space $(S,d)$ is doubling if and only if for every
  $0<\lambda<1$ there is a constant $K= K(\lambda)$ such that each open
  ball $B\subset S$ of radius $r>0$ contains at most $K$ points that are
  $\lambda r$-separated.
\end{lemma}

This is a  standard fact, but we include a proof for the convenience of the reader.

\begin{proof}
  Assume first that $(S,d)$ is doubling with the doubling constant $N\in \N$. Let $B\subset S$ be an open ball
  of radius $r>0$ and let $V\subset B$ be a $\lambda r$-separated
  set. By the assumption, there is a constant $n=n(N, \lambda)$ such that $B$ can be covered by $n$ open balls $B_1,\dots, B_n$ of diameter $<\lambda
  r$. Since each $B_j$ can contain at most one point from $V$, we conclude that $\# V\leq n$. It follows that $(S,d)$ satisfies the
  condition in the statement with $K=n$.  

  We now prove the converse. Suppose $(S,d)$ satisfies the condition in the
  statement of the lemma. Let $B\subset S$ be an open ball of radius $r>0$ and
  $V\subset B$ be a maximal $r/2$-separated subset of $B$. By
  assumption, we have  $\#V \leq K=K(1/2)$. At the same time, the maximality of $V$ implies that the balls $B(x, r/2)$, $x\in V$,
  cover $B$. Therefore,  $(S,d)$ is doubling with the doubling constant $N=K$. 
\end{proof}

\subsection{Quasi-metrics}
\label{sec:quasi-metrics}
When constructing a metric on a set $S$, it is often easier to first define a distance function with weaker properties, which can then be promoted to a metric.

\begin{definition}
  \label{def:qmetric}
  Let $S$ be a set. A function $q\colon S\times S \to [0,\infty)$ is called
  a \emph{quasi-metric}  on $S$ if there exists  a constant  $K\geq 1$ such that the
  following conditions are satisfied for all $x,y,z\in S$:
  \begin{enumerate}
  \smallskip
  \item
    \label{item:qmetric1}
    $q(x,y) = 0$ if and only if  $x=y$;
  \smallskip
  \item
    \label{item:qmetric2}
    $q(x,y) = q(y,x)$;
  \smallskip  
  \item
    \label{item:qmetric3}
    $q(x,y) \leq K \max\{q(x,z), q(z,y)\}$. 
  \end{enumerate}
 \end{definition}
  In this case, $q$  is called a \emph{$K$-quasi-metric} on $S$,
  and we refer to the pair $(S,q)$ as a \emph{$K$-quasi-metric space} or simply a \emph{quasi-metric space}. 

Clearly, every metric on a set $S$ is a $K$-quasi-metric with $K=2$. Conversely, if $q$ is a $K$-quasi-metric on $S$ with $K\leq 2$, we can find  a metric $d$ on $S$ that is comparable to $q$.  

\begin{theorem}[{\cite[Theorem~1.2]{Sch06}}]
  \label{thm:qmetric_metric}
  Let $q$ be a $K$-quasi-metric on a set $S$ with $1\le K\leq 2$. Then
  there is a metric $d$ on $S$ that satisfies
  \begin{equation*}
    \frac{1}{2K} q \leq d \leq q. 
  \end{equation*}
\end{theorem}

\subsection{Classes of maps between metric spaces}
\label{sec:quasisymmetries}
Let $(S_1,d_1)$ and $(S_2,d_2)$ be two 
metric spaces, and $\phi\colon S_1 \to S_2$ 
be a bijection. We say that $\phi$ is a 
{\em bi-Lipschitz map} if there exists a constant $L\ge 1$ such that 
\[
\frac{1}{L}d_1(x,y) \leq d_2(\phi(x),\phi(y)) \leq 
L d_1(x,y) 
\]
for all $x,y\in S_1$.
We say that $\phi$ is a 
\emph{snowflake equivalence} if there exist $\alpha>0$ and $C\geq 1$ such that
\[
\frac{1}{C}d_1(x,y)^\alpha \leq d_2(\phi(x),\phi(y)) \leq 
C d_1(x,y)^\alpha 
\]
for all $x,y\in S_1$.  We call the map $\phi$ a 
\emph{quasisymmetry}  if there exists a homeomorphism $\eta\colon [0,\infty) \to [0,\infty)$ (which plays the role of a  ``distortion
function")  such that
\[
\frac{d_2(\phi(x), \phi(y))}{d_2(\phi(x), \phi(z))} \leq \eta\left( \frac{d_1(x,y)}{d_1(x,z)}\right)
\]
for all $x,y,z\in S_1$ with $x\neq z$. When we want to emphasize the distortion function $\eta$, we speak of an \emph{$\eta$-quasisymmetry}.

Each of these three classes of maps is preserved under compositions or taking inverse maps. For example, the composition of two bi-Lipschitz maps is bi-Lipschitz and the inverse  of a quasisymmetry is also a quasisymmetry. 

Clearly, every bi-Lipschitz map is a snowflake equivalence, and every snowflake equivalence is a quasisymmetry. It is also easy to see that every quasisymmetry (and hence every snowflake equivalence and every bi-Lipschitz map) is a homeomorphism. 

These types of maps also lead to notions of equivalence for metrics on a given set $S$. For example, we call two metrics $d_1$ and $d_2$ on $S$ {\em quasisymmetrically equivalent} if the identity map $\id_S\: (S, d_1)\ra (S, d_2)$ is a quasisymmetry. Bi-Lipschitz and snowflake equivalence of metrics on $S$ are defined in a similar way. 

We record some  well-known facts about quasisymmetries that we will use in the sequel.

\begin{lemma}[{\cite[Propositions~10.8 and 10.10, and Exercise~11.2]{He}}]\label{lem:qs_properties}
    Let $\phi\colon (S_1,d_1) \to (S_2,d_2)$ be a quasisymmetry. Then the following statements are true: 
    \begin{enumerate}
 \smallskip
        \item If $(S_1,d_1)$ is bounded, then $(S_2,d_2)$ is bounded.
\smallskip
        \item If $(S_1,d_1)$ is complete, then $(S_2,d_2)$ is complete.
\smallskip        
        \item If $(S_1,d_1)$ is uniformly perfect, then $(S_2,d_2)$ is uniformly perfect. 
    \end{enumerate}
\end{lemma}

\begin{lemma}[{\cite[Proposition 10.8]{He}, see also \cite[Lemma~2.8]{BM22}}]
  \label{lem:qs_diam}
  Assume $\id_S\colon (S,d_1)\to (S,d_2)$ is an $\eta$-quasisymmetry. Let
  $X\subset S$ be a subset of $S$, and $x,y\in X$ be points satisfying
  $d_1(x,y) \asymp \diam_1(X)$ with  a constant
  $C_1=C(\asymp)$. Then
  \begin{equation*}
    \diam_2(X) \asymp d_2(x,y),
  \end{equation*}
  where $C(\asymp)=C(C_1,\eta)$. 
\end{lemma}

\begin{cor}
  \label{cor:qmetric_snow}
  Every quasi-metric space $(S,q)$ is snowflake-equivalent to a
  metric space $(S,d)$. 
\end{cor}

\begin{proof}
  Let $(S,q)$ be a $K$-quasi-metric space. Note that for
  $x,y,z\in S$ and $\alpha>0$ we have
  \begin{equation*}
    q^\alpha(x,y) \leq K^\alpha \max\{q^\alpha(x,z), q^\alpha(z,y)\}.
  \end{equation*}
  Choosing $\alpha>0$ sufficiently small such that $K^\alpha\leq
  2$, we see from Theorem~\ref{thm:qmetric_metric} that
  $q^\alpha \asymp d$ for some metric $d$ on $S$. 
\end{proof}

\section{Visual approximations}
\label{sec:visu-appr}

Many classical fractal sets (such as the standard Cantor set, the Koch
snow\-flake, and the
Sierpi\'{n}ski carpet and gasket) have ``standard decompositions''
into sets of size $\Lambda^{-n}$ for $n\in \N_0$, where  $\Lambda>1$ is a suitable parameter. In this section, we abstract the properties  of such decompositions into  the notion of a  \emph{visual approximation}. Along with controlling the sizes of the elements in these approximations, we will also 
impose a metric separation condition (as discussed in the introduction).

Let $S$ be a set and $\{\X^n\}_{n\in\N_0}$ be a sequence of covers of $S$.   For simplicity, we will usually just write $\{\X^n\}$ instead of
$\{\X^n\}_{n\in\N_0}$ with the index set $\N_0$ for $n$ 
understood. The covers $\X^n$, $n\in \N_0$, can be arbitrary and we allow infinite covers.

Each set $X\in \X^n$ for $n\in \N_0$ is called a \emph{tile of level $n$}, or an
\emph{$n$-tile}; we call $X$ a \emph{tile} if its level is
irrelevant.  Note that even when $S$ is endowed with a metric, there are no topological requirements on
tiles, though in applications  tiles will typically be either 
all compact or all open.

Given $n\in \N_0$ and two $n$-tiles $X,Y\in \X^n$, a \emph{chain of $n$-tiles from $X$ to $Y$} is a finite sequence $X_0=X,\,X_1,\,\dots,\,X_k=Y$
of $n$-tiles such that $X_j \cap X_{j+1}\neq \emptyset$ for all
$j=0,\dots,k-1$. We call $k$ the \emph{length} of the chain (so it consists of $k+1$ $n$-tiles).

The \emph{neighborhood of a tile $X\in \X^n$ of width $w\in \N_0$}, denoted by $U_w(X)$, is
defined  as the 
set 
\begin{align}
  \label{eq:def_UwX^n}
  U_{w}(X) \coloneqq \bigcup\{Y\in \X^n: \text{there } &\text{exists a chain of $n$-tiles}\\
  &\text{from $X$ to $Y$ of length $\le w$}\}.\notag
                \end{align}
Clearly, we have
\begin{equation}\label{eq:nbhds-are-nested}
  X=U_0(X) \subset U_1(X) \subset \dots 
\end{equation}
for each $X\in \X^n$. Moreover, if $U_w(X) \cap U_w(Y) \neq \emptyset$ for $X, Y\in \X^n$, then we can find a chain 
\begin{equation}
X_0=X,\, X_1,\,\dots,\, X_w,\, Y_w,\, \dots,\, Y_1,\,Y_0
  =Y
  \end{equation}
of $n$-tiles from $X$ to $Y$ of length $2w+1$ (note that we can make the length equal to $2w+1$ here by listing a tile multiple times).

\begin{definition}[Visual approximations]
  \label{def:visual} 
  Let $(S,d)$ be a bounded metric space, $\Lambda>1$, and $w\in \N_0$.
  A \emph{visual
    approximation of width $w$ with (visual) parameter $\Lambda$ for $(S,d)$} is a
  sequence  
  $\{\X^n\}_{n\in\N_0}$ of covers  of $S$ by some of its subsets, where $\X^0=\{S\}$,
  such that the following conditions are satisfied: for all $n\in \N_0$ and 
   $X,Y\in \X^n$,  we have 
  \begin{enumerate}%[label=(VA \roman*)]  
 \smallskip
  \item 
    \label{item:visual1}
    $\diam(X) \asymp \Lambda^{-n}$, 
    \smallskip
  \item 
    \label{item:visual2}
    $\dist(X, Y) \gtrsim \Lambda^{-n}$
    whenever
    $U_w(X) \cap U_w(Y) = \emptyset$.
  \end{enumerate}  

  \smallskip\noindent
  Here we require  that the implicit constants $C(\asymp)$ and $C(\gtrsim)$ are
  independent of $n$, $X$, and~$Y$.  
\end{definition}

For brevity, we will frequently omit the reference to the width $w$ or the visual parameter $\Lambda$ 
when we mention  visual approximations. 
For the width
$w=0$, the second condition above simplifies to the following:
\begin{enumerate}[label=(\roman*')]
  \setcounter{enumi}{1}

  \smallskip
\item
  \label{item:visual_width0}
  $\dist(X,Y) \gtrsim \Lambda^{-n}$
  whenever  
  $X\cap Y=\emptyset$
  for $X, Y\in \X^n$.
\end{enumerate}

\smallskip\noindent
Since we require $\X^0=\{S\}$,  
 condition~\ref{item:visual1} for  $n=0$ simply says that 
 $(S,d)$ is bounded and has positive diameter (or equivalently, consists of  more than one point), while condition~\ref{item:visual2}
for  $n=0$ is vacuous.
Moreover,  \eqref{eq:nbhds-are-nested} implies that a visual approximation of width
$w\in \N_0$ is automatically a visual approximation of width $w'\in \N_0$ for all  $w'\geq w$.

We can also reverse the point of view and focus on the metric instead 
of the sequence~$\{\X^n\}$. 
\begin{definition}[Visual metrics]
\label{def:vis_metric}
Let $S$ be a set and $\{\X^n\}$ be a sequence of covers of $S$ such that 
$\X^0=\{S\}$. We call a metric $\varrho$ on $S$ a 
\emph{visual metric for $\{\X^n\}$} if there exist parameters   $w\in \N_0$ and $\Lambda>1$ such that $\{\X^n\}$ is a visual approximation of width $w$ with visual parameter~$\Lambda$ for $(S,\varrho)$,  as in 
Definition~\ref{def:visual}. 
\end{definition}
Note that in this case  $(S,\varrho)$ is bounded and $S$ consists of more than one point. We call $w\in \N_0$ the {\em width parameter} and  $\Lambda>1$ the  {\em visual parameter}   of the visual metric $\varrho$.

Condition~\ref{item:visual1} and condition~\ref{item:visual2} with width $w=0$ from Definition~\ref{def:visual} were first introduced 
in \cite{BM} in the context  of \emph{expanding Thurston maps} (see \cite[Proposition~8.4]{BM}). There, the name {\em visual metric}
was also used, although this term was orginally coined for certain metrics on the boundary at infinity of Gromov hyperbolic spaces. As we will explain,
visual metrics  as in Definition~\ref{def:vis_metric}
and visual metrics for Gromov hyperbolic spaces are closely related (see  Proposition~\ref{prop:vis_Gromov_hyp}). This connection motivated our terminology.

\begin{remark} 
 In all of our main results for visual and quasi-visual approximations (to be defined in Section~\ref{sec:quasi-visu-appr}), such as
 Corollary~\ref{cor:visual_ex},
    Theorem~\ref{thm:Gromov_implies_qv}, and
    Theorem~\ref{thm:semi_hyp_admits_approx},  we have $w\in \{0,1\}$ for the width parameter $w$. We will also see momentarily that if a metric space admits a visual approximation of some width $w\in \N_0$, then it also admits a visual approximation of width $1$ (see Remark~\ref{rem:w1}). 
    Therefore,  it is not so clear at the moment  whether 
    it is really necessary  to consider widths $w\geq
    2$. Nevertheless, we chose to present the theory in this greater generality in order to have more flexibility for potential   applications in the future. 
  \end{remark}

\subsection{Existence of visual approximations} We first address the question when a given metric space has a visual
approximation.

\begin{proposition}
  \label{prop:vis_width1_ex}
  Let $(S,d)$ be a bounded metric space.
  Then $(S,d)$ admits a visual approximation   if and only if $(S,d)$ is uniformly perfect.
\end{proposition}

\begin{proof} ``$\Rightarrow$''  Suppose that $(S,d)$ admits a visual approximation $\{\X^n\}$. Let $\Lambda>1$ be the  visual parameter
of $\{\X^n\}$. Then there exists a constant  $C\geq 1$ such  that 
  \begin{equation*}
    \frac{1}{C} \Lambda^{-n} \leq \diam(X^n) \leq C \Lambda^{-n}
  \end{equation*}
  for all $n\in \N_0$ and $X^n \in \X^n$. 

  Now let $x\in S$ and $0<r\leq \diam(S)$ be  arbitrary,   and  $n=n(r)\in \N_0$ be the
  smallest number such that $C\Lambda^{-n}<r$. Then
  $C\Lambda^{-n+1}\geq r$. We can choose $X^n\in\X^n$ with $x\in X^n$. Then 
  we have $\diam(X^n) \leq C \Lambda^{-n} <r$ and so 
  $X^n\subset B(x,r)$. On the other hand, since $\diam(X^n) \geq
  \frac{1}{C} \Lambda^{-n}$, there exists  a point $y\in X^n$
  with
  \begin{equation*}
    d(x,y)
    >
    \frac{1}{3}
    \diam(X^n)
    \ge
    \frac{1}{3C} \Lambda^{-n}
    \geq
    \frac{1}{3C^2\Lambda} r.  
  \end{equation*}
  This shows that $(S,d)$ is uniformly perfect with constant
  $\lambda= \frac{1}{3C^2\Lambda}$.

  \smallskip
  ``$\Leftarrow$'' Suppose that $(S,d)$ is uniformly perfect with constant 
  $\la\in (0,1)$. We fix a parameter $\Lambda>1$. 
  
  For $n\in \N$, let $S^n\subset S$ be a maximal
  $\Lambda^{-n}$-separated set.  
  We set $\X^0 \coloneqq \{S\}$ and 
  \begin{equation*}
    \X^n\coloneqq \{B(x, 2 \Lambda^{-n}) : x\in S^n\} 
  \end{equation*}
  for $n\in \N$. Clearly, each $\X^n$ is a cover of~$S$.  We want to show  that the sequence $\{\X^n\}$ provides  a  visual approximation for $(S,d)$ of width $w=1$ with parameter $\Lambda$. 

 For $n\in \N$ and $X\in \X^n$ we have $X=B(x, 2 \Lambda^{-n})$ for some 
 $x\in S^n$. We conclude that 
  \begin{align*}
    \diam(X) \asymp \Lambda^{-n},
  \end{align*}
 where 
  $C(\asymp)= C(\lambda)$ depends only on the uniform perfectness constant 
  $\lambda$. Since $(S,d)$ is bounded, it follows that 
  condition~\ref{item:visual1} in
  Definition~\ref{def:visual} holds. It remains to show
  condition~\ref{item:visual2}. 

  \smallskip
  \emph{Claim.} For all $n\in \N$ and  $X,Y\in \X^n$
  we have the implication:
  \begin{align*}
  U_1(X) \cap U_1(Y) = \emptyset \ 
    \Rightarrow \
    \dist(X, Y) \geq \tfrac12 \Lambda^{-n}.
  \end{align*}

\smallskip
 We argue by contradiction and assume that $U_1(X) \cap U_1(Y) = \emptyset$, but  there
  are points $x_0\in X$ and $y_0\in Y$ with  $d(x_0,y_0) < \tfrac12 \Lambda^{-n}$.  
  By the choice of $S^n$, the balls $B(s, \Lambda^{-n})$, $s\in S^n$, form a cover of 
  $S$. Thus, we can  find $x'\in S^n$ with $x_0\in B(x',
  \Lambda^{-n})$. Then  $X' \coloneqq B(x', 2 \Lambda^{-n}) \in \X^n$ and $x_0\in 
  X\cap X'$. We conclude that $X' \subset U_1(X)$. 

  On the other hand, we have
  \begin{align*}
    d(x', y_0)
    \leq
    d(x', x_0) + d(x_0,y_0)
    < \Lambda^{-n} +
    \tfrac12 \Lambda^{-n}  
    < 2\Lambda^{-n}. 
  \end{align*}
  It follows that  $y_0 \in X' \cap Y$ and so 
  $X'\subset U_1(Y)$. Therefore,  $$U_1(X) \cap
  U_1(Y) \neq \emptyset,$$ in contradiction to our
  assumption. This proves the claim.

  \smallskip
 It follows that condition~\ref{item:visual2} in
  Definition~\ref{def:visual} is
  satisfied as well. Thus, $\{\X^n\}$ is a visual approximation of width $1$ with parameter $\Lambda$ for $(S,d)$. This completes the proof of ``${\Leftarrow}$". 
\end{proof}

\begin{remark} \label{rem:w1} The previous proof showed that if a metric space $(S,d)$ is bounded and uniformly perfect, then $(S,d)$ admits a visual approximation 
of width $w=1$ for each visual parameter $\Lambda>1$. It follows that if an arbitrary metric space $(S,d)$ has a visual approximation of some width $w\in \N_0$, then it also admits a visual approximation of width $1$.
    \end{remark}

\subsection{Visual approximations of doubling metric spaces}
\label{sec:visu-appr-doubl}

If our given bounded metric space $(S,d)$ is uniformly perfect and doubling, then there always exists
a visual approximation of width $w=0$. 
We record this fact for reference only and will not use it in the remainder of the paper. We will  deduce it from a more technical statement.  

\begin{proposition}[Visual approximations from separated sets]
  \label{prop:vis_max_sep}
  Suppose $(S,d)$ is a 
  bounded, doubling, and uniformly perfect metric space. 
  Let
  $\Lambda>1$  and  $S^n$ be a maximal $\Lambda^{-n}$-separated subset of $S$ for each $n\in \N$.

  Then there are  numbers $r_{s,n}\in [1,2)$ for  $s\in S^n$ and $n\in
  \N$ such that if we define $\X^0\coloneqq \{S\}$
and 
\[
\X^n \coloneqq \{B(s,r_{s,n} \Lambda^{-n}) : s\in S^n\}
\] for  $n\in \N$,
  then the sequence $\{\X^n\}$ is a visual approximation of width $w=0$ with parameter $\Lambda$ for $(S,d)$.
\end{proposition}

Since $(S,d)$ is bounded and doubling,   Lemma~\ref{lem:doubling_separated} implies  that each set $S^n$, $n\in \N$, in the previous
proposition is finite. The following fact is an immediate consequence. 
\begin{cor}
  \label{cor:visual_ex}
  Let $(S,d)$ be a bounded metric space. 
  Suppose $S$ is doubling and uniformly
  perfect. Then for every $\Lambda>1$ there exists a visual
  approximation $\{\X^n\}$ of 
  width $w=0$ with parameter $\Lambda$ for $(S,d)$, in which each cover $\X^n$ is finite.
\end{cor}

The proof of Proposition~\ref{prop:vis_max_sep}  is based on the following fact.

\begin{lemma}
  \label{lem:reldist} 
  Let $\delta>0$ and $V$ be a maximal $\delta$-separated set in a doubling metric
  space $(S,d)$. Then  there exist numbers  $r_x\in [1,2)$ for $x\in V$ 
  such that the following condition is true: 
  
  If we define 
  $B_x\coloneqq B(x, r_x\delta)$ for $x\in V$, then 
  for all $x,y\in V$ we have
  \begin{equation}
    \label{eq:sepimpl}
    \text{either }
    B_x\cap B_y \neq \emptyset
    \;\text{ or }
    \dist (B_x, B_y) \ge C\delta,
  \end{equation} 
  where $C>0$ is a constant depending only on the doubling constant of~$(S,d)$. 
\end{lemma} 
This is essentially \cite[Lemma~3.1]{Ma08}, but we include  a proof  for the convenience of the reader.

\begin{proof}  Without loss of generality, we may assume that $\delta=1$. 

  \smallskip
  
  \emph{Claim.} 
  There exists $N\in \N$, depending only  on the doubling constant
  of $(S,d)$, such that we can split $V$ into $N$ pairwise disjoint 
  (possibly empty)
  sets $A_1, \dots, A_N$ satisfying the following condition:
  \[
  \text{if $x,y\in A_i$ for
  some $i\in\{1,\dots, N\}$ and $x\ne y$, then $d(x,y)\ge 10$.}
  \]
  
On a more intuitive level, this means that we can ``color'' the points in $V$
  by at most $N$ distinct colors such that distinct points of the same
  color are $10$-separated.

To prove the claim, first note that since $V$ is a $1$-separated set,
  we can choose $N\in \N$ only depending 
  on the doubling constant of $(S,d)$ such that each ball $B(v, 10)$ with $v\in V$ contains at most $N$ elements in $V$ (see Lemma~\ref{lem:doubling_separated}).

We now  consecutively choose a maximal $10$-separated set 
$A_1$ in $V$, 
   a maximal $10$-separated set $A_2$ in $V\setminus A_1$, a maximal $10$-separated set $A_3$ in $V\setminus (A_1\cup A_2)$, etc.
Continuing in this way, we obtain pairwise disjoint $10$-separated subsets $A_1$, $A_2$, $A_3$,
\dots of $V$.

Here $A_{N+1}=\emptyset$. Indeed, if $A_{N+1}\ne 
\emptyset$, then we can find an element $v\in A_{N+1}\sub V$. For each 
$k=1, \dots, N$, the element $v$  is not contained  
in the $10$-separated set $A_k$ that is  maximal in 
\[V\setminus (A_1\cup \dots \cup A_{k-1})\supset \{v\}.\] It follows that there exists $a_k\in A_k$ such that 
$d(v, a_k)<10$. Then  $v, a_1, \dots, a_N$ are $N+1$  distinct elements in 
$V$ all contained in the ball $B(v, 10)$. This is impossible by the
choice of $N$. We conclude that $A_{N+1}=\emptyset$ and so 
$V=A_1\cup \dots \cup A_N$. The claim follows.

  \smallskip

  Suppose $A_1,\dots, A_N\subset V$ are chosen according
  to the claim above. We now show that \eqref{eq:sepimpl} holds with
  $C=1/(2N)$ and appropriate numbers $r_x\in [1,2) $ for $x\in V$.

  Set $V_n\coloneqq A_1\cup\dots \cup A_n$ for  
  $n= 1,\dots, N$. Then we have  $A_1=V_1 \subset V_2 \subset \dots \subset V_N=V$. The numbers $r_x\in [1,2)$ for $x\in V$ will now be chosen
  inductively so that \eqref{eq:sepimpl} holds for all
  $x,y\in V_n$.

  We set $r_x \coloneqq 1$ for $x\in V_1 = A_1$. 
  If $x,y\in V_1=A_1$ are distinct, then 
  $d(x,y)\geq 10$, and so
  \begin{equation*}
    \dist(B_x,B_y) \geq 8 > \frac{1}{2N}=C.
  \end{equation*}
  It follows that \eqref{eq:sepimpl} holds for all $x,y\in V_1$, establishing the base case 
  $n=1$.

 Now assume that $n\in \N$, $1\le n\le N$,  and 
  that $r_x\in [1,2)$ has been chosen  for each $x\in
  V_n$ so that \eqref{eq:sepimpl} holds for all
  $x,y\in V_n$ (with $C=1/(2N)$).

 If $n=N$, 
 the proof is complete; so suppose that  $n<N$ and let 
 $x\in A_{n+1}=V_{n+1} \setminus V_n$ be arbitrary (if 
 $A_{n+1}=\emptyset$, and so $V_n=V$, there is nothing left to do).  
 For  $y\in V_n$, we now define
  \begin{equation*}
    d(y) \coloneqq \inf\{r>0 : B(x,r) \cap B_y \neq \emptyset\}.
  \end{equation*}
  Then  $B(x,r)\cap B_y = \emptyset$ for $0<r< d(y)$ and $B(x,r)
  \cap B_y \neq \emptyset$ for  $r> d(y)$ (note that the first case cannot occur when $d(y)=0$).
  
Let $W_x$ be the (possibly empty) set of all $y\in
  V_n$ with $d(y) \in [1,2)$. Then for each
  $y\in W_x$ we have $d(x,y) \leq 2 + r_y < 4$, which implies that 
  $d(y,y') < 8$
  for  $y,y'\in W_x$. 
   
  It follows 
  that the points in $W_x$ must all lie in distinct sets among $A_1, \dots, A_{n}$, and so 
   $M:=\#W_x \leq n \leq N-1$. Let
  us label the points in $W_x$ as $y_1, \dots , y_M$ (possibly this is an empty list) so that
  \begin{equation*}
    1\eqqcolon d_0 \leq d_1 \leq \dots \leq d_M < d_{M+1}\coloneqq 2,
  \end{equation*}
  where $d_j \coloneqq d(y_j)$ for $j=1,\dots,M$. 
  It follows that there is an index $k\in \{0,\dots, M\}$
  such that $d_{k+1} - d_k \geq 1/N$.

  We now define $r_x\coloneqq \frac{1}{2}(d_k + d_{k+1})\in [1,2)$. Note that
  \begin{equation*}
    r_x -d_k
    = d_{k+1}-r_x=
    \frac{1}{2}(d_{k+1}-d_k)
    \geq
    \frac{1}{2N} = C. 
  \end{equation*}
  Moreover, for each $y\in V_n$ we have: \[\text{either } \ d(y) \leq
  d_k< r_x\    \text { or } \   d(y)\geq d_{k+1}>r_x.\]
  Then 
  \begin{equation*}
    B(x,r_x) \cap B_y\neq \emptyset
  \end{equation*}
  in the former case, and 
  \begin{equation*}
    \dist(B(x,r_x), B_y)\geq d_{k+1}-r_x\geq C
  \end{equation*}
  in the latter case. 
  
  With this choice 
 of $r_x\in [1,2)$ for each $x\in A_{n+1}$, 
 condition~\eqref{eq:sepimpl} is satisfied for all $x\in
  A_{n+1}=V_{n+1}\setminus V_n$ and $y\in V_n$.

  By the induction hypothesis, \eqref{eq:sepimpl} also holds for all
  $x,y\in V_n$. It remains to verify \eqref{eq:sepimpl} for distinct points  $x,y\in A_{n+1}$. In this case, $d(x,y)\geq 10$ and so 
  \begin{equation*}
    \dist(B_x, B_y) \geq 10 - r_x - r_y > 6 > \frac{1}{2N}=C.
  \end{equation*}
  It follows that \eqref{eq:sepimpl} holds for all $x,y\in
  V_n\cup A_{n+1}=V_{n+1}$. This  completes the inductive step and the proof of the statement. 
\end{proof} 

The proof of Proposition~\ref{prop:vis_max_sep} is  now routine.

\begin{proof}[Proof of Proposition~\ref{prop:vis_max_sep}]
 In  the given setup,  we may assume 
without loss of generality that $\diam(S)=1$.

  According to Lemma~\ref{lem:reldist}, there are numbers $r_{s,n}\in [1,2)$ for $n\in \N$ and $s\in S^n$ such that the balls $B_s\coloneq B(s, r_{s,n}\Lambda^{-n} )$ satisfy the following condition for $x,y\in S^n$:
    \begin{equation*}
    \text{either }\;
    B_x\cap B_y \neq \emptyset
    \text{ or }
    \dist (B_x, B_y) \ge C\Lambda^{-n}.
  \end{equation*}
Here  $C>0$ is a constant only depending on the doubling constant 
 of $S$ (and not on $n$). 
 
 As in the statement, let 
$\X^0\coloneqq \{S\}$ and 
  \begin{equation*}
    \X^n \coloneqq \{B(s,r_{s,n}\Lambda^{-n}) : s\in S^n\} 
  \end{equation*}
  for $n\in \N$.
  
  Clearly, each $\X^n$ is a covering of $S$. Moreover, each $\X^n$ satisfies condition~\ref{item:visual_width0}
  in Definition~\ref{def:visual} with $C(\gtrsim)=C$. Finally, since $\diam(S)=1$ and $(S,d)$ is uniformly perfect, we have that
  \[\lambda
  \Lambda^{-n} \leq \diam(X) \leq 4 \Lambda^{-n}\]
  for each $n\in \N_0$ and $X\in \X^n$,
  where $\lambda$ is the uniform perfectness constant. It
  follows that property~\ref{item:visual1} in
  Definition~\ref{def:visual} holds as well. Thus $\{\X^n\}$ is a visual approximation of width $0$ with visual parameter $\Lambda$ for $(S,d)$. 
\end{proof}

In Proposition~\ref{prop:vis_max_sep} the balls in each cover  $\X^n$, $n\in \N$,  are open, but we may instead take the corresponding closed balls and 
consider $\overline{\X}^0 \coloneqq \{S\}$ and 
 \[\overline{\X}^n \coloneqq \{\overline{B}(s,r_{s,n} \Lambda^{-n}) :
s\in S^n\}\] for $n\in \N$.  Then the sequence $\{\overline{\X}^n\}$ is also a visual
approximation for $(S,d)$; indeed, the implicit constants in
Definition~\ref{def:visual} for $\{\overline{\X}^n\}$ are the same as
those for $\{\X^n\}$.

We now show by an explicit example that visual approximations of width $w=0$ 
may exist for metric spaces that are not doubling.  In other words,  being
doubling is not a necessary condition for the existence of visual
approximations. 

\begin{example}
  \label{ex:non_doubl_unif_perf}
  Our space $S$ will be the ``boundary at infinity"  of  an infinite rooted tree where every vertex on the $n$-th level (with $n\in\N_0$)
  has exactly $n+2$ ``children".
  To make this precise, we define $S$ to be  the set of all infinite sequences $\{x_n\}_{n\in \N_0}$
  with 
  \begin{equation*}
    x_n \in \{0,1,\dots, n\}
  \end{equation*}
  for $n\in \N_0$.

  Given two such sequences  $x=\{x_n\}_{n\in \N_0}$ and $y=\{y_n\}_{n\in \N_0}$, we set
  \begin{equation*}
    m(x,y) \coloneqq \left\{\begin{array}{cl}\min\{n\in \N : x_n \neq y_n\}& \text{if $x\ne y$,}\\
    \infty & \text{if $x=y$.}
    \end{array} \right.
  \end{equation*}
   We then define
  \begin{equation*}
    d(x,y) \coloneqq 2^{-m(x,y)},
  \end{equation*}
  where it is understood that $d(x,x) = 2^{-\infty} = 0$. It is immediate that $d$
  is a metric on $S$; in fact, it is an \emph{ultrametric}, meaning that \[d(x,y)\leq \max\{d(x,z), d(z,y)\}\] for all $x,y,z\in S$.

  The \emph{cylinder sets of level $n\in \N_0$} are the subsets of $S$ of the form
  \begin{equation*}
    [x_0x_1\dots x_n]
    \coloneqq
    \big\{y=\{y_k\}_{k\in \N_0} \in S : y_0=x_0, \, y_1 = x_1, \,\dots, \,y_n =x_n\big\} 
  \end{equation*}
  with $x_i\in \{0, 1,\dots, i\}$ for $i=0, \dots, n$. For each $n\in\N_0$, we let
  $\X^n$ be the set of all such cylinder sets of  level
  $n$; in particular, we have $\X^0=\{S\}$. It is straightforward to check that $\{\X^n\}$ forms a
  visual approximation of width $w=0$ with parameter $\Lambda= 2$ for $(S,d)$.

  However, the space $(S,d)$ is not doubling. Indeed, each
  cylinder set $[x_0x_1\dots x_n]$ of level $n\in \N_0$  is also an open ball of
  radius $2^{-n}$ in $S$. On the other hand, for each 
  $k=0, \dots ,n+1$, this cylinder set
   contains the  point $y_k$ given by the sequence
  \[x_0,\, x_1,\, \dots,\, x_n,\, k,\, 0, \, 0,\, \dots \,. \]
  Since the $n+2$ points $y_0, \dots, y_{n+1}\in [x_0x_1\dots x_n]$ are  $2^{-n-1}$-separated, the space $(S,d)$ cannot be doubling according to Lemma~\ref{lem:doubling_separated}. 
\end{example}

\section{Quasi-visual approximations}
\label{sec:quasi-visu-appr}

In general, visual approximations are not preserved under quasisymmetries, because this concept of an approximation is too rigid.   To remedy this, we introduce the more 
flexible notion of a {\em quasi-visual approximation}. Recall the definition of the neighborhood $U_w(X)$ of a tile $X$ of width $w\in \N_0$ from
\eqref{eq:def_UwX^n}. 

\begin{definition}[Quasi-visual approximations]
  \label{def:qv_approx}
  Let $(S,d)$ be a bounded metric space containing  more than one point. 
  A \emph{quasi-visual
  approximation of width $w\in \N_0$} for $(S,d)$ is a sequence
  $\{\X^n\}_{n\in\N_0}$ of covers $\X^n$ of $S$ by some of its subsets, where $\X^0=\{S\}$ and  the following conditions hold  for all $n\in \N_0$ with implicit constants independent of~$n$ and independent of the tiles $X$ and $Y$ that appear in these conditions: 
    \begin{enumerate}
\smallskip
  \item 
    \label{item:qv_approx1}
       $\diam(X) \asymp \diam(Y)$
  for all  $X,Y\in \X^n$ with $X\cap Y\ne \emptyset$.
  
  \smallskip
  \item 
    \label{item:qv_approx2}
      $\dist(X,Y) \gtrsim \diam(X)$
      for all  $X,Y\in \X^n$ with 
      $U_w(X)\cap U_w(Y)=\emptyset$.
      
      \smallskip
     \item 
    \label{item:qv_approx3}
     $\diam(X)\asymp \diam(Y)$
for all $X\in \X^n$, $Y\in \X^{n+1}$ with 
      $X\cap Y\ne \emptyset$.

      \smallskip
  \item 
    \label{item:qv_approx4}
    For some constants $k_0\in \N$ and $\lambda \in (0,1)$ independent of $n$, we have $ \diam(Y) \leq \lambda \diam(X)$
    for all $X\in \X^n$ and $Y\in \X^{n+k_0}$ 
     with $X\cap Y\neq \emptyset$.   \end{enumerate}
\end{definition}

For the width $w=0$, the second condition above reduces to the
following:
\begin{enumerate}[label=(\roman*')]
  \setcounter{enumi}{1}

  \smallskip
\item
  \label{item:qvisual_width0}
  $\dist(X,Y) \gtrsim \diam(X)$
  for all   $X, Y\in \X^n$ with 
  $X\cap Y=\emptyset$.
\end{enumerate}

\smallskip\noindent 
In this form, quasi-visual approximations were  introduced in \cite[Definition~2.1]{BM22}; note though that in contrast to 
this previous work, we do not assume that the covers $\X^n$ are finite.

Similarly to visual approximations, we
sometimes change our point of view and focus on the metric.

\begin{definition}[Quasi-visual metrics]
\label{def:qvis_metric}
Let $\{\X^n\}$ be a sequence of covers of a set $S$. A metric $d$ on $S$ is 
called \emph{quasi-visual of width $w\in \N_0$ for $\{\X^n\}$} if $0<\diam_d(S)<\infty$ and  $\{\X^n\}$ is 
a quasi-visual approximation of width $w$ for $(S,d)$, that is, the conditions 
in Definition~\ref{def:qv_approx} are satisfied for the sequence $\{\X^n\}$ with respect to 
$d$. 
\end{definition}

In the following, whenever we talk about a metric space equipped with a quasi-visual approximation, we automatically assume that $S$ is bounded and has more than one point. Moreover, for brevity the reference to the width $w$ will often be omitted when discussing quasi-visual approximations or metrics.

The following three statements  
(Lemmas~\ref{lem:sub_shrink} and~\ref{lem:visual-is-also-quasi}  and  Corollary~\ref{cor:qv_tiles_shrink})  have very close analogs in  \cite{BM22}. We will still provide full
 details for the sake of completeness. 
 
First, we show that condition~\ref{item:qv_approx4} in Definition~\ref{def:qv_approx}  can be expressed in a slightly
different form.

\begin{lemma}
  \label{lem:sub_shrink}
  Let $\{\X^n\}$ be a sequence of covers  of a bounded metric space  $S$ satisfying condition~\ref{item:qv_approx3} in Definition~\ref{def:qv_approx}. Then
  $\{\X^n\}$ satisfies \ref{item:qv_approx4} if and only 
  if there
  exist  constants $C\ge 1$ and $\rho \in (0,1)$ such that
  \begin{gather}
    \label{eq:qv_approx4p}
    \tag{iv'} 
  \diam(Y) \leq C \rho^k \diam(X)
  \end{gather}
  for all $n\in \N_0$, $k\in \N$,  $X\in \X^n$, and $Y\in\X^{n+k}$ with
  $X\cap Y\neq \emptyset$. 
  \end{lemma}

\begin{proof}
  Assume first that condition~\ref{item:qv_approx4} in Definition~\ref{def:qv_approx} holds with  
   $k_0\in \N$ and $\lambda\in (0,1)$. Let  $n\in \N_0$, 
   $k\in \N$, 
    $X\in \X^n$, and  $Y\in\X^{n+k}$ with
  $X\cap Y\neq \emptyset$ be arbitrary. 
  We pick a point $x\in X\cap Y$.  For each $i=n, \dots, n+k$, we then choose a tile 
  $Y^i\in \X^i$  with $x\in Y^i$ so that $Y^n=X$ and $Y^{n+k}=Y$.  We also write $k$  in the form $k=jk_0 + \ell$ with  $j\in \N_0$ and
  $\ell\in \{0,\dots,k_0-1\}$. 
   Finally,  let $C_1=C(\asymp)\ge 1$ be  the constant in condition~\ref{item:qv_approx3} of
  Definition~\ref{def:qv_approx}. Then by 
 repeated applications of  \ref{item:qv_approx3} and  \ref{item:qv_approx4} 
   we  see 
  that
     \begin{align*}
    \diam(Y)
    &=\diam(Y^{n+k})=
    \diam(Y^{n+jk_0 +\ell}) 
    \leq 
    C_1^{\ell}\diam(Y^{n+jk_0})
    \\
    &\leq
    C_1^{\ell}\lambda^j\diam(Y^{n})     
    \le  
    C_1^{k_0}(\lambda^{1/k_0})^{jk_0}\diam(X). 
   \end{align*}  
   If we define $\rho \coloneqq\lambda^{1/k_0}\in (0,1)$ 
      and $C\coloneqq C_1^{k_0}/\lambda\ge 1$, then 
\[C_1^{k_0}= C\lambda =C\rho^{k_0}\le C\rho^{\ell}, \] 
and so
       \begin{align*}
    \diam(Y)
  &\le  C_1^{k_0} \rho^{\,jk_0}\diam(X)
    \le C\rho^{\,jk_0 + \ell} \diam(X)
    \\
    &= C\rho^k \diam(X),
   \end{align*}  
  as desired. 
  
  To show the reverse implication, assume that
  \eqref{eq:qv_approx4p} holds for constants $C>0$ and 
  $\rho\in (0,1)$. We can then choose $k_0\in \N$ sufficiently large
  such that $\lambda\coloneqq C\rho^{k_0}<1$. With
  these choices,  condition~\ref{item:qv_approx4} 
  in Definition~\ref{def:qv_approx} is clearly satisfied.
\end{proof}

A very similar argument shows that  condition~\ref{item:qv_approx3} in Definition~\ref{def:qv_approx} implies an inequality opposite to   \eqref{eq:qv_approx4p}. Namely, there exists a constant $\tau\in (0,1) $
such that 
\begin{equation}
  \label{eq:qv_approx8} \diam(Y)\ge \tau^k \diam(X)
\end{equation} for all $n\in \N_0$, $k\in \N$, $X\in \X^n$,
and $Y\in\X^{n+k}$ with $X\cap Y\neq \emptyset$. Note that
\eqref{eq:qv_approx4p} and \eqref{eq:qv_approx8} imply that
\begin{equation}
  \label{eq:taurhoin} 0<\tau\leq \rho<1.
\end{equation}
When combining the estimates in \eqref{eq:qv_approx4p} and
\eqref{eq:qv_approx8}, it is useful to consider \[\nu\coloneqq
\frac{\log(1/\rho)}{\log(1/\tau)}=\frac{\log(\rho)}{\log(\tau)}.\] We have $\nu\in (0,1]$ by
\eqref{eq:taurhoin} and
\begin{equation}
  \label{eq:def_nu}
  \rho = \tau^\nu. 
\end{equation}

\begin{remark}\label{rem:X^0auto} 
Let $\{\X^n\}_{n\in \N}$ be  a sequence of covers of a bounded metric space $(S,d)$ with positive diameter. Suppose  $\{\X^n\}_{n\in \N}$
satisfies  conditions 
\ref{item:qv_approx1}--\ref{item:qv_approx4} in Definition~\ref{def:qv_approx} for all $n\in \N$ with suitable constants. If $\X^1$ is a finite cover and consists of sets of positive diameters, then we can 
add $\X^0=\{S\}$ to the sequence and $\{\X^n\}_{n\in \N_0}$ becomes a 
quasi-visual approximation of $(S,d)$ (with suitably adjusted constants). 

Indeed, in Definition~\ref{def:qv_approx}
condition~\ref{item:qv_approx1} is then also (trivially) true for $n=0$,  while condition~\ref{item:qv_approx2} is vacuous for $n=0$. Our assumptions imply that $\diam(X)\asymp\diam(S)$
for all $X\in \X^1$. This yields condition~\ref{item:qv_approx3}
for all $n\in \N_0$ with a suitably adjusted constant $C(\asymp)$. 

Finally, to see that 
condition~\ref{item:qv_approx4} is true, one verifies  the equivalent 
condition~\eqref{eq:qv_approx4p}. It holds with suitable constants $C$ and $\rho$ for all $n\in \N$ by our assumptions on the original sequence  $\{\X^n\}_{n\in \N}$. Condition~\eqref{eq:qv_approx4p} is then also valid for all $n\in \N_0$ if we replace the original constant $C$ there with $C/\rho$. 
\end{remark}

Lemma~\ref{lem:sub_shrink} implies that in a quasi-visual approximation
the diameters of tiles tend to $0$ uniformly with their level.

\begin{cor}
  \label{cor:qv_tiles_shrink}
  Let $\{\X^n\}$ be a quasi-visual approximation of a bounded
  metric space $S$. Then
  \begin{equation*}
    \sup\{\diam(X) : X\in \X^n\} \to 0
    \text{ as } n\to \infty. 
  \end{equation*}
\end{cor}

\begin{proof}
  Indeed, by Lemma~\ref{lem:sub_shrink} we know that there exist
  constants $C>0$ and $\rho\in (0,1)$ such that 
  \begin{equation*}
    \diam(X^n) \leq C \rho^n \diam(X^0)
  \end{equation*}
  for all  $n\in \N$ and $X^n\in \X^n$. Here $X^0=S$ is the
  only $0$-tile. The statement follows. 
\end{proof}

We now show  that quasi-visual approximations generalize visual approximations. 
 
\begin{lemma}[Visual approximations are quasi-visual]
  \label{lem:visual-is-also-quasi} Let $(S,d)$ be a 
metric space.
 Then every visual
  approximation $\{\X^n\}$  of width $w\in \N_0$ for  $(S,d)$   is also a quasi-visual approximation  of width $w$
 for $(S,d)$. 
\end{lemma}

\begin{proof}
  Let $\{\X^n\}$ be a visual approximation of width $w\in \N_0$ and with  visual parameter $\Lambda>1$ for the metric
  space $(S,d)$, that is, $\{\X^n\}$ satisfies the conditions in
  Definition~\ref{def:visual}. 
  In particular, the space $(S,d)$ is bounded and has more than one point. 
  We have to verify conditions 
\ref{item:qv_approx1}--\ref{item:qv_approx4} in
Definition~\ref{def:qv_approx}.  In the following, all implicit
constants are independent of the tiles under consideration and their levels. 

Note first that each $\X^n$ is a cover of $S$ and that $\X^0= \{S\}$. 
Fix $n\in \N_0$,
and consider arbitrary tiles $X,Y\in \X^n$. Then
  \begin{equation*}
    \diam(X)
    \asymp
    \Lambda^{-n}
    \asymp
    \diam(Y),
  \end{equation*}
  by Definition~\ref{def:visual}~\ref{item:visual1}. Thus
  condition~\ref{item:qv_approx1} in
  Definition~\ref{def:qv_approx} is satisfied.

  If $U_w(X)\cap U_w(Y)=\emptyset$, then by 
  Definition~\ref{def:visual}~\ref{item:visual2} we have
  \begin{equation*}
    \dist (X,Y)
    \gtrsim
    \Lambda^{-n}
    \asymp
    \diam (X), 
  \end{equation*} 
 and so condition~\ref{item:qv_approx2} in
 Definition~\ref{def:qv_approx} is satisfied.

  Let $Z\in \X^{n+1}$ be
  arbitrary. Then
  \begin{equation*}
    \diam(Z)
    \asymp
    \Lambda^{-n-1}
    \asymp
    \Lambda^{-1} \diam(X),
  \end{equation*}
  by Definition~\ref{def:visual}~\ref{item:visual1}. Condition~\ref{item:qv_approx3}
  in 
  Definition~\ref{def:qv_approx} immediately follows (with a constant depending on $\Lambda$ in addition to the other ambient parameters).

  Similarly, if $Z\in \X^{n+k}$ for some $k\in \N$, then
  \begin{equation*}
    \diam(Z)
    \asymp
    \Lambda^{-(n+k)}
    \asymp
    \Lambda^{-k} \diam(X), 
  \end{equation*}
  again by Definition~\ref{def:visual}~\ref{item:visual1}.
 This implies that condition~\eqref{eq:qv_approx4p} in Lemma~\ref{lem:sub_shrink} is true, and thus condition~\ref{item:qv_approx4} in 
  Definition~\ref{def:qv_approx} holds as well.  

The statement follows.
\end{proof}

 The following lemma compares tile neighborhoods with metric balls.
\begin{lemma}
  \label{lem:U2w1_qround}
  Let $\{\X^n\}$ be a quasi-visual approximation of width $w\in
  \N_0$ for  a bounded metric space $(S,d)$. Then there are
  constants $r_0, R_0>0$ such that for each $X\in \X^n$ with 
  $n\in \N_0$ and each $x\in X$ we have 
  \begin{equation*}
    B(x, r_0 \diam(X))
    \subset
    U_{2w+1}(X)
    \subset
    B(x, R_0 \diam(X)),
  \end{equation*}
\end{lemma}
Here $U_{2w+1}(X)$ is defined as in \eqref{eq:def_UwX^n}. Note
that 
\begin{equation*}
  U_{2w+1}(X)
  =
  \bigcup\{Y\in \X^n : U_w(X) \cap U_w(Y) \neq
  \emptyset\}.   
\end{equation*}
The lemma asserts that the sets $U_{2w+1}(X)$ are {\em quasi-balls}
with uniform constants (meaning that each $U_{2w+1}(X)$ lies between two concentric  balls with radii whose ratio is  uniformly bounded). 

\begin{proof} In the given setup, suppose  
 $Y\in \X^n$ is  a tile that contributes to the union $U_{2w+1}(X)$. Then there is a chain
  of $n$-tiles from $X$ to $Y$ of length $\le 2w+1$. All $n$-tiles in
  this chain have diameter $\asymp \diam(X)$ as follows from 
  condition \ref{item:qv_approx1} in  Definition~\ref{def:qv_approx}. This implies that 
$U_{2w+1}(X) \subset B(x,R_0 \diam(X))$ for a
 uniform  constant $R_0>0$.

 Now consider a point $z\in S \setminus U_{2w+1}(X)$, and choose  $Z\in
  \X^n$ with $z\in Z$. It follows that $U_w(X) \cap U_w(Z) =
  \emptyset$, and so condition \ref{item:qv_approx2} in 
  Definition~\ref{def:qv_approx} implies that
  \begin{equation*}
    d(x,z)
    \geq
    \dist(X,Z)
    \gtrsim
    \diam(X).
  \end{equation*}
  This means that there is a uniform  constant $r_0>0$ such that $d(x,z)
  \geq r_0 \diam(X)$. We conclude   that $B(x,r_0\diam(X)) \subset
  U_{2w+1}(X)$, and the statement follows. 
  \end{proof}

We close this section with an auxiliary result that will be used later in Section~\ref{subsec:dynqvisual}.

\begin{lemma}
  \label{lem:metricball}
  Let $\{\X^n\}$ be a quasi-visual approximation of a bounded
  metric space $(S, d)$. Let $R>0$, $n\in \N_0$, $X,Y\in \X^n$,
  and $x\in X$ be arbitrary, and suppose that
  \begin{equation*}
    Y\cap B(x, R\diam(X))\ne \emptyset.
  \end{equation*}
  Then $\diam (X)\asymp \diam(Y)$ with a constant
  $C(\asymp)=C(R)$ of comparability that depends on~$R$, but is
  independent of $n$, $X$, $Y$, and $x$.
\end{lemma}

\begin{proof}
  Let $w\in \N_0$ be the width of the quasi-visual approximation
  $\{\X^n\}$ and $r_0>0$ be the constant from
  Lemma~\ref{lem:U2w1_qround} such that 
  \begin{equation}\label{eq: ball_nbhd}
   B(x, r_0 \diam(Z))
    \subset
    U_{2w+1}(Z)
  \end{equation}
  for all $n\in \N_0$, $Z\in \X^n$, and $x\in Z$. By condition~\eqref{eq:qv_approx4p} in
  Lemma~\ref{lem:sub_shrink} we can find $n_0=n_0(R)\in \N$
  depending on $R>0$ such that for all $n\in\N_0$, $Z^n\in \X^n$, and $Z^{n+n_0}\in
  \X^{n+n_0}$ with $Z^n\cap Z^{n+n_0}\neq \emptyset$ we have
  \begin{equation}\label{eq: n0_tiles}
    R\diam(Z^{n+n_0})<r_0\diam (Z^n).
  \end{equation}

After this preparation, let $n\in \N_0$,  $X,Y\in \X^n$, and  $x\in X$  be  as in the
statement. Suppose first that $n< n_0$. Then by
condition~\ref{item:qv_approx3} of Definition~\ref{def:qv_approx}
we have  
\[
\diam(X)\asymp \diam(S) \asymp \diam (Y) 
\]
with comparability constants that depend on $n_0$ (and hence on
$R$), but not on $X$ and $Y$.  

So we may assume that $n \geq n_0$. In this case, choosing a tile 
$X'\in \X^{n-n_0}$ with $x\in X'$, we obtain from
\eqref{eq: ball_nbhd} and \eqref{eq: n0_tiles} that
\begin{equation*}
  B(x,R \diam(X))
  \subset
  B(x, r_0\diam(X'))
  \subset
  U_{2w+1}(X').
\end{equation*}
Since, by our assumptions, $Y$ meets the ball $B(x, R\diam(X))$,  we conclude that there exists a point $y\in Y\cap
U_{2w+1}(X')$.

Choosing now a tile $Y'\in \X^{n-n_0}$ with $y\in
Y'$, we have $Y'\cap U_{2w+1}(X')\ne \emptyset$. By
condition~\ref{item:qv_approx1} in
Definition~\ref{def:qv_approx}, this implies
\begin{equation*}
  \diam(X')\asymp \diam (Y')
\end{equation*}
with a uniform constant $C(\asymp)$.
At the same time, using condition~\ref{item:qv_approx3} of
Definition~\ref{def:qv_approx}, we obtain
\begin{equation*}
  \diam(X)\asymp\diam(X')
  \ \text{ and } \
  \diam (Y') \asymp\diam(Y)
\end{equation*}
with comparability constants that depend on $n_0$ (and
hence on $R$), but not on $X$ and~$Y$.  
The statement follows.
\end{proof}

\section{Distances and combinatorics}
\label{sec:dist-comb}

Let $\{\X^n\}$ be a quasi-visual approximation of width $w\in \N_0$ for a metric space $(S,d)$. 
The covers $\X^n$ provide a
discrete approximation of $S$ that becomes finer as $n$ increases. We
introduce a combinatorial quantity (depending on the width $w$) that records
the level at which this approximation allows us to distinguish
two given points for the first time.

\begin{definition}
  \label{def:mxy} 
  Let $\{\X^n\}$ be a sequence of covers of a set
  $S$. Given $w\in \N_0$, the \emph{$U_w$-proximity function} \[m_w\colon S\times S\to \N_0 \cup \{\infty\}\] 
  (with respect to $\{\X^n\}$) is defined as follows. For $x,y\in S $ we set   
  \begin{align*} 
    m_w(x,y)
    \coloneqq 
    \sup\{n\in \N_0: {}&\text{there exist $X,Y\in \X^n$}
      \text{ with $x\in X$, $y \in Y$,}
    \\ 
    &
    \text{and $U_w(X)\cap U_w(Y)\ne \emptyset$}\}.
  \end{align*} 
\end{definition}

Of course, 
the number $m_w(x,y)$ 
depends on $\{\X^n\}$ and $w$. We will frequently suppress the lower index from the notation if the width $w$ is fixed in the context.
Clearly, \[m_w(x,y) =
m_w(y,x) \text{ for all $x,y\in S$, }
\] and $m_w(x,y)$ is  non-decreasing in $w$. Furthermore, for $n\in \N_0$ and $X^n\in \X^n$, we have
\begin{equation}
  \label{eq:mxy_Xn}
  m_w(x,y) \geq n \text{ for all } x,y\in X^n. 
\end{equation}

\begin{remark}
There are other natural combinatorial quantities associated with a sequence $\{\X^n\}$ of covers of a set $S$ that can be used to measure the proximity of points in $S$. In particular, we may consider the following quantity for  two points $x,y\in S$:
\begin{align*}
m_w'(x,y)
\coloneqq
\inf\{n\in \N_0 : {}&\text{there exist $X,Y\in \X^n$}
\text{ with $x\in X$, $y \in Y$, }
\\
&
\text{and $U_w(X)\cap U_w(Y)= \emptyset$}\},
\end{align*}
where $m_w'(x,y)=\infty$ if the set over which the infimum is taken happens to be empty. One can easily verify that 
$ m'_w(x,y)\le m_w(x,y)+1$ for all $x,y\in S$.
\end{remark}

The $U_w$-proximity function $m_w$ is defined purely in combinatorial terms, and $S$ is not required to carry a metric. In the rest of the section, we fix the width $w$ and discuss the properties of this function in a metric setting.

\begin{lemma}
  \label{lem:mxy_qv}
  Let  $\{\X^n\}$ be  a quasi-visual approximation of width $w\in \N_0$ for a bounded metric space $(S,d)$. Then for $x,y\in S$ we have
  \begin{align*}
    m_w(x,y) = \infty
             \; \text{ if and only if } \;
             x=y.
  \end{align*}
\end{lemma}
\begin{proof}
  Clearly,  $m_w(x,x) = \infty$ for all $x\in S$. To see the other implication, recall from
Corollary~\ref{cor:qv_tiles_shrink} that 
  \begin{equation*}
    \lim_{n\to\infty} \sup_{X\in \X^n} \diam(X)=0.
  \end{equation*}
This implies that for distinct $x,y\in S$ we have $m_w(x,y)\in
  \N_0$. 
\end{proof}

The $U_w$-proximity function allows us to control distances.

\begin{lemma}
  \label{lem:qv_metric}
  Let $\{\X^n\}$ be a quasi-visual approximation of width $w\in \N_0$ for a
 bounded  metric space $(S,d)$. Then there is a constant $C(\asymp)$ such
  that for all distinct points $x,y\in S$ we have
  \begin{equation*}
    d(x,y) \asymp \diam(X^m),
  \end{equation*}
  where $m=m_w(x,y)$ and $X^m\in \X^m$ is an arbitrary $m$-tile
  that contains~$x$.
\end{lemma}

 Note that $m\in \N_0$ by Lemma~\ref{lem:mxy_qv} and so the statement makes sense. 
\begin{proof} Let $x,y\in S$ be   distinct, and 
  $m=m_w(x,y)\in \N_0$.  Then there are $m$-tiles $X,Y\in \X^m$ with $x\in X$, $y\in Y$,
  and $U_w(X)\cap U_w(Y)\neq \emptyset$. 
  The latter statement  in turn implies  that there is a
  chain of $m$-tiles 
  \[X_0=X,\, X_1,\,\dots,\, X_w,\, Y_w,\, \dots, Y_1,\, Y_0
  =Y
  \] from $X$ to $Y$ of length $2w+1$.

  Now let  $X^m\in \X^m$ be an arbitrary $m$-tile containing $x$.
  Then 
  \begin{align*}
    d(x,y)
    &\leq
    \diam(X_0) +\dots + \diam(X_w) + \diam(Y_w) + \dots +
      \diam(Y_0)
      \\
    &\lesssim
    \diam(X)
    \asymp
    \diam (X^m),
  \end{align*}
  where we used condition~\ref{item:qv_approx1} in 
  Definition~\ref{def:qv_approx} repeatedly.

  On the other hand, we can find $(m+1)$-tiles $X^{m+1},Y^{m+1}\in \X^{m+1}$ with $x\in X^{m+1}$ and
  $y\in Y^{m+1}$. The definition of $m_w(x,y)$ now implies that $U_w(X^{m+1})
  \cap U_w(Y^{m+1}) = \emptyset$. This in turn gives
  \begin{equation*}
    d(x,y)\geq \dist(X^{m+1},Y^{m+1}) \gtrsim \diam(X^{m+1}) \asymp \diam(X^m),
  \end{equation*}
 where we used conditions \ref{item:qv_approx2} and \ref{item:qv_approx3} in 
  Definition~\ref{def:qv_approx}. The statement follows.
\end{proof}

The following corollary is an immediate consequence of the previous lemma  and Definition~\ref{def:visual}.

\begin{cor}
  \label{cor:dxy_mxy}
    Let $\{\X^n\}$ be a visual approximation of width $w\in \N_0$ with visual parameter $\Lambda>1$ for a
    bounded  metric space $(S,d)$. Then
    there is a constant $C(\asymp)$ such 
  that for all distinct points $x,y\in S$ we have
  \begin{equation*}
    d(x,y) \asymp \Lambda^{-m}, 
  \end{equation*}
  where $m=m_w(x,y)$.
\end{cor}

We now show that snowflake equivalences are characterized by their
preservation of visual approximations.

\begin{proposition}
  \label{prop:visual_metric_approximation}
  Let $\{\X^n\}$ be a visual approximation of width $w\in \N_0$
  for a bounded metric
  space $(S,d)$. Suppose $d'$ is another metric on $S$. 
  Then the following statement is true: 
  \begin{align*}
    &\text{$\{\X^n\}$ is a visual approximation of width $w$ for
    $(S,d')$}\\
%    \intertext{if and only if}
 \Longleftrightarrow\quad   &\text{$\id_S\colon (S,d) \to (S,d')$ is a snowflake equivalence.}
  \end{align*}
\end{proposition}

Recall that the latter condition means that for some $\alpha>0$ we have 
$d'(x,y) \asymp
d(x,y)^\alpha$ for all $x,y\in S$. The
proof will show that the visual parameters $\Lambda,
\Lambda'>1$ of the visual approximation $\{\X^n\}$ for $(S,d)$ and $(S,d')$, respectively,  
then satisfy $\Lambda' = \Lambda^\alpha$. In particular, the identity map $\id_S \colon (S,d) \to
(S,d')$ is bi-Lipschitz if and only if $\Lambda = \Lambda'$. 

\begin{proof}
  Let $\Lambda>1$ be the visual parameter of the visual approximation
  $\{\X^n\}$ of $(S,d)$.
  
  ``$\Leftarrow$''
  Assume first that $\id_S\colon (S,d) \to (S,d')$ is a snowflake
  equivalence, meaning that there exists $\alpha>0$ such that
  $d'(x,y) \asymp d(x,y)^\alpha$ for all $x,y\in S$. Set
  $\Lambda'\coloneqq \Lambda^\alpha>1$, and let $n\in\N_0$ be arbitrary. 
  Then for all $X^n\in \X^n$ we have
  \[\diam_{d'} (X^n)
    \asymp
    (\diam_d(X^n))^\alpha
    \asymp
    (\Lambda^{-n})^\alpha
   =
      (\Lambda')^{-n},\]
   by condition~\ref{item:visual1} in Definition~\ref{def:visual}. Moreover, by condition~\ref{item:visual2},  for all $X^n, Y^n\in \X^n$ with $U_w(X^n)\cap U_w(Y^n)=\emptyset$
   we have
   \[\dist_{d'}(X^n,Y^n)
      \asymp
      (\dist_d(X^n,Y^n))^\alpha
      \gtrsim
      (\Lambda^{-n})^\alpha
      =
      (\Lambda')^{-n}.
  \]
  Hence
  $\{\X^n\}$ is a visual approximation of width $w$
  with visual parameter $\Lambda'$ for $(S,d')$ according to
  Definition~\ref{def:visual}.

  \smallskip
  ``$\Rightarrow$''
  Conversely, suppose $\{\X^n\}$ is a visual approximation of width $w$ with parameter
  $\Lambda'>1$ for $(S,d')$. Set
  \[\alpha\coloneqq \log(\Lambda')/\log(\Lambda)>0.
  \] 
 Then $\Lambda'= \Lambda^\alpha$.
  
  The $U_w$-proximity function $m=m_w$ from
  Definition~\ref{def:mxy} depends only on $\{\X^n\}$
  and $w$ and is
  independent of the metrics $d$ and $d'$. Then  Corollary~\ref{cor:dxy_mxy} implies that for all distinct points $x,y\in S$ we have
  \begin{align*}
    d'(x,y)
    \asymp
    (\Lambda')^{-m(x,y)}
    =
    (\Lambda^\alpha)^{-m(x,y)}
    =
    (\Lambda^{-m(x,y)})^\alpha
    \asymp
    d(x,y)^\alpha. 
  \end{align*}
  Thus $\id_S\colon (S,d) \to (S,d')$ is a snowflake equivalence, as
  desired. 
\end{proof}

\section{Dynamical quasi-visual approximations}\label{subsec:dynqvisual} 
In the theory of (discrete-time) dynamical systems, one considers self-maps $g\colon S\to S$ on a space $S$, which typically carries a metric~$d$. In order for our  notion of a (quasi-)visual approximation for the metric space $(S,d)$ to be useful in such settings, it is natural to require that it respects the induced dynamics. We will formulate the precise condition momentarily.

In the following, let $(S,d)$ be a bounded metric space with a self-map $g\colon S\to S$ (we do not assume any regularity properties of $g$). We denote by $g^n$ for $n\in \N_0$ the $n$-th iterate of~$g$, where $g^0\coloneq \id_S$. 

We  call a sequence $\{\X^n\}_{n\in\N}$ of covers of $S$ {\em dynamical} if it satisfies  
\begin{equation}\label{eq:dyngXgen_1}
 g(X^{n+1})\in \X^n \;\text{ for all $n\in \N$ and $X^{n+1}\in 
 \X^{n+1}$.}
\end{equation}
On a more intuitive level, this means that the map $g$ shifts the level of $n$-tiles by $1$ for all $n\geq 2$. This condition immediately generalizes to the following one:
\begin{equation}\label{eq:dyngXgen_k}
 g^k(X^{n+k})\in \X^n \;\text{ for all $n\in \N$, $k\in \N_0$, and $X^{n+k}\in 
 \X^{n+k}$.}
\end{equation}

In particular, a (quasi-)visual approximation $\{\X^n\}_{n\in\N_0}$ of $(S,d)$ is called {\em dynamical} if the subsequence $\{\X^n\}_{n\in\N}$ satisfies \eqref{eq:dyngXgen_1}. Since for these approximations we require $\X^0=\{S\}$ by  definition, we do not impose  any dynamical restrictions on $\X^0$.

In Section~\ref{sec:semihyp}, we will discuss in detail a particular instance of such a situation (namely, when $g$ is a rational map on the Riemann sphere and $(S,d)$ is its Julia set equipped with the spherical metric).  In this section, however, we will only derive some distortion properties in the general case. 

Suppose $\{\X^n\}$ is a dynamical quasi-visual approximation of width $w\in \N_0$ for $(S,d)$. By the dynamical property~\eqref{eq:dyngXgen_1}, intersecting $n$-tiles are mapped to intersecting $(n-1)$-tiles by $g$ for all $n\geq 2$. Consequently, \[g\big(U_w(X)\big)\subset U_w\big(g(X)\big)\] for all tiles $X\in \X^n$, $n\in \N$. It follows immediately that
the $U_w$-proximity function $m_w$ with respect to $\{\X^n\}$ satisfies 
\[
m_w(g(x), g(y))\ge m_w(x,y)-1 
\]
for all $x, y\in S$. 
Iterating this inequality, we obtain 
\begin{equation}\label{eq:dynmw}
 m_w(g^n(x), g^n(y))\ge m_w(x,y)-n   
\end{equation}
for all $n\in \N$ and $x,y\in S$.

Inequality~\eqref{eq:dynmw} leads to a metric 
distortion property for the iterates of $g$. 
In order to formulate it, we fix constants $\tau, \rho\in (0,1)$
for $\{\X^n\}$ as  in  \eqref{eq:qv_approx8} and  in
condition~\eqref{eq:qv_approx4p} of Lemma~\ref{lem:sub_shrink},
respectively. Again, we set $\nu = \log(1/\rho)/\log(1/\tau)\in
(0,1]$ so that $\rho= \tau^\nu$ as in~\eqref{eq:def_nu}.

\begin{lemma}
  \label{lem:dgxy}
  For each $R>0$ there exists 
  a constant $C=C(R)\ge 1$ such that 
  \begin{equation}\label{eq:equicont}
d(g^n(x), g^n(y))\le C \cdot \biggl(\frac{d(x,y)}{\diam(Z^{n+1})}
\biggr)^{\nu},
\end{equation}
whenever $n\in \N$, $Z^{n+1}\in \X^{n+1}$, $z_0\in Z^{n+1}$, and 
$x,y\in B(z_0, R \diam(Z^{n+1}))$.
\end{lemma}

Note that \eqref{eq:equicont} implies, in particular, that the map $g$ is continuous on $S$. 

\begin{proof} Let $n$, $Z^{n+1}$, $z_0$, $x$, and $y$ be as in the statement, and define $m\coloneqq m_w(x,y)$. We may assume that $g^n(x)\ne g^n(y)$, which in turn implies $x\ne y$.  Then $m\in \N_0$ by Lemma~\ref{lem:mxy_qv}. We choose tiles $X^{n+1}\in \X^{n+1}$ and $X^m\in \X^{m}$ with $x\in X^{n+1}\cap X^{m}$ (if $m=n+1$, we have $X^m=X^{n+1}$). 

Suppose first that $m> n+1$.
Then by 
Lemma~\ref{lem:qv_metric} and \eqref{eq:qv_approx8} we have 
\begin{equation}\label{eq:xydist}
  d(x,y)\asymp \diam(X^m)\geq \tau^{m-n-1} \diam(X^{n+1}).    
\end{equation}
We consider now the image points $g^n(x)$ and $g^n(y)$.  By \eqref{eq:dynmw}, we have \[k\coloneqq m_w(g^n(x), g^n(y)) \ge m-n.\]  
We choose a tile $Y^{k}\in \X^{k}$ with $g^n(x)\in Y^{k}$. 
Lemma~\ref{lem:qv_metric} and  \eqref{eq:qv_approx4p} imply that 
\[
d(g^n(x), g^n(y))\asymp \diam(Y^{k})\lesssim 
\rho^{k}\diam(S)\lesssim \rho^{m-n}.
\]
 Putting this together with \eqref{eq:def_nu} and \eqref{eq:xydist}, we see that 
\begin{equation}\label{eq:distortion1}
    d(g^n(x), g^n(y))\lesssim \rho^{m-n} =\tau^{\nu(m-n)}\lesssim
\biggl(\frac{d(x,y)}{\diam(X^{n+1})}\biggr)^{\nu}. 
\end{equation}

Suppose now that $m\leq n+1$. Using again Lemma~\ref{lem:qv_metric} and  \eqref{eq:qv_approx4p} we obtain
\[\frac{d(x,y)}{ \diam(X^{n+1})}\asymp \frac{\diam(X^m)}{ \diam(X^{n+1})}\gtrsim (1/\rho)^{n+1-m} \geq 1,\]
while 
\[d(g^n(x), g^n(y))\le \diam(S)\asymp 1.\] 
It follows that the estimate~\eqref{eq:distortion1} still holds in this case (possibly after adjusting the uniform multiplicative constant). 

We have shown that for a uniform constant $C_0\ge 1$ (independent of the items specified at the beginning of the proof) we have 
\[
d(g^n(x), g^n(y))\le C_0\cdot \biggl(\frac{d(x,y)}{\diam(X^{n+1})}\biggr)^\nu.
\]

It remains to replace $\diam(X^{n+1})$ with $\diam(Z^{n+1})$ in the above inequality at the expense of a multiplicative constant that may depend on $R$. By the standing assumption, we have
\[x\in X^{n+1}\cap B(z_0, R\diam(Z^{n+1})),\]
that is,
the tile $X^{n+1}$ meets the ball $B(z_0, R\diam(Z^{n+1}))$. 
By Lemma~\ref{lem:metricball} we then have  $\diam(Z^{n+1})\asymp \diam(X^{n+1})$ with $C(\asymp)=C(R)$. The statement follows. \end{proof}

\section{Quasi-visual approximations and quasisymmetries}
\label{sec:quasi-visu-appr-1}

One of our principal reasons to consider quasi-visual
approximations is that they are closely related to
quasisymmetries. The following statement is the main result in
this respect.

\begin{theorem}
  \label{thm:qv-qs} 
  Let $(S,d_1)$ be a bounded metric space with a quasi-visual
  approximation $\{\X^n\}$ of width $w\in \N_0$. Suppose $d_2$ is
  another metric on $S$ such that $(S,d_2)$ is bounded. Then 
    $\id_S\colon (S,d_1) \to (S,d_2)$ is a quasisymmetry
   if and only if
    $\{\X^n\}$ is a quasi-visual approximation of $(S,d_2)$.
\end{theorem}

An analogous statement has already appeared in \cite{BM22} (which goes
back to the proof of \cite[Lemma~18.10]{BM} as well as
\cite[Theorem~4.2]{Me02}). 
Since our definition of a quasi-visual approximation here is more general, and in order to make the exposition self-contained, we provide a full proof.

In this section, we henceforth assume the hypotheses of Theorem~\ref{thm:qv-qs}---namely, that $(S,d_1)$ is a bounded metric space with
a quasi-visual approximation $\{\X^n\}$ of width $w\in \N_0$, and that
$d_2$ is another metric on $S$ such that $(S,d_2)$ is bounded. 

Metric notions with respect to $d_i$ will be denoted with a
subscript $i$ for $i=1,2$. For example, the diameter in $(S,d_1)$
is denoted by $\diam_1$, the diameter in $(S,d_2)$ by $\diam_2$,
and so on.

We first show the ``$\Rightarrow$'' implication in
Theorem~\ref{thm:qv-qs}. 

\begin{lemma}
  \label{lem:qv_qs_inv}
  Suppose $\id_S\colon(S,d_1) \to (S,d_2)$ is a quasisymmetry. Then
  $\{\X^n\}$ is also a quasi-visual approximation  of
  width $w$ for $(S,d_2)$. 
\end{lemma} 

\begin{proof}
  We have to verify that $\{\X^n\}$ satisfies  conditions
  \ref{item:qv_approx1}--\ref{item:qv_approx4} in
  Definition~\ref{def:qv_approx} for $(S,d_2)$ with the same width $w$.
  
  \smallskip
  \ref{item:qv_approx1}
  Let $n\in \N_0$ and $X,Y\in \X^n$ with $X\cap Y\ne \emptyset $ be arbitrary. Then there are points $z\in X\cap Y$, 
  $x\in X$,  and $y\in Y$   with $d_1(x,z)\asymp \diam_1(X)$
  and $d_1(y,z) \asymp \diam_1(Y)$, where $C(\asymp)=3$. Since
  condition~\ref{item:qv_approx1} of 
  Definition~\ref{def:qv_approx} is satisfied for $d_1$, we know
  that
  \begin{equation*}
    d_1(x,z) \asymp \diam_1(X) \asymp \diam_1(Y)\asymp d_1(y,z).
  \end{equation*}
  The quasisymmetric equivalence of $d_1$ and $d_2$ together with
  Lemma~\ref{lem:qs_diam} now implies  that
  \begin{equation*}
    \diam_2(X)\asymp d_2(x,z) \asymp d_2(y,z) \asymp \diam_2(Y).
  \end{equation*}
  So $\{\X^n\}$ satisfies \ref{item:qv_approx1} in 
  Definition~\ref{def:qv_approx}
  for $(S,d_2)$. 

  \smallskip
  \ref{item:qv_approx3}
  The verification of this condition  is very similar to the one  for
  \ref{item:qv_approx1}, and so we will skip the details.  

  \smallskip 
  \ref{item:qv_approx2}
  Let $n\in \N_0$ and $X,Y\in \X^n$  with $U_w(X)\cap U_w(Y)=
  \emptyset$, as well as $x\in X$ and $y\in Y$ be arbitrary. 
  We can choose  $x'\in X$ with
  $d_1(x,x') \asymp \diam_1(X)$, where $C(\asymp)=3$. Since
  \ref{item:qv_approx2} holds for $d_1$, we know that
  \begin{equation*}
    d_1(x,x') 
    \leq 
    \diam_1(X) 
    \lesssim
    \dist_1(X,Y) 
    \leq d_1(x,y).
  \end{equation*}
  The quasisymmetric equivalence of $d_1$ and $d_2$ together with Lemma~\ref{lem:qs_diam} now implies
  \begin{equation*}
    \diam_2(X) \asymp d_2(x,x') \lesssim d_2(x,y).
  \end{equation*}
  Taking the infimum over all $x\in X$ and $y\in Y$, we obtain
  $$\diam_2(X) \lesssim \dist_2(X,Y).$$ This  is condition~\ref{item:qv_approx2} for $d_2$.  

  \smallskip
  \ref{item:qv_approx4}
  Let $\lambda\in (0,1)$ and $k_0\in \N$ be  constants as in 
  condition~\ref{item:qv_approx4} of
  Definition~\ref{def:qv_approx}
    for  $(S,d_1)$.  Suppose $n\in \N_0$, $\ell\in \N$, $X\in \X^n$ and $Y\in \X^{n+\ell k_0}$ with $X \cap Y
    \ne \emptyset$ are arbitrary.
    Then we can find points  $x\in  X \cap Y $,  
    $x'\in X$ with $\diam_1(X)\le 3 d_1(x,x')$, and  
   $y\in Y$ with $\diam_2(Y)\le 3 d_2(x,y)$.
    
  Choosing intermediate tiles of levels $n+k_0, \dots, n+(\ell-1)k_0$ containing $x$  and applying \ref{item:qv_approx4} in 
  Definition~\ref{def:qv_approx} repeatedly for $(S,d_1)$, we obtain
  \[
    d_1(x,y) \leq \diam_1(Y)
    \le \la^\ell \diam_1(X)\le 3 \la^\ell d_1(x,x').
    \]
Now  $\id_S\colon (S,d_1) \to (S,d_2)$ is an $\eta$-quasisymmetry for some homeomorphism $\eta\:[0,\infty)\ra [0,\infty)$. Then it follows that 
\[
\frac13 \diam_2(Y)\le d_2(x,y)\le \eta(3 \la^\ell)d_2(x,x')\le \eta(3 \la^\ell)
\diam_2(X), 
\]
and so
\[
\diam_2(Y)\le 3 \eta(3 \la^\ell)\diam_2(X).
\]
  Since $\eta$ is a
  homeomorphism on $[0,\infty)$ and $\lambda\in (0,1)$, there exists $\ell_0\in \N$ such that
  $\widetilde{\lambda}\coloneqq 2 \eta (3\la^{\ell_0}) <1$. With this choice of $\widetilde{\lambda}\in (0,1)$  and  $\widetilde{k}_0 \coloneqq \ell_0 k_0\in \N$,  it  follows 
  that
 condition~\ref{item:qv_approx4} in 
 Definition~\ref{def:qv_approx} holds for
  $(S,d_2)$. 
\end{proof}

We now prove the ``$\Leftarrow$'' implication in
Theorem~\ref{thm:qv-qs}. 

\begin{lemma}
  \label{lem:qv_qs}
  Assume that $\{\X^n\}$ is a quasi-visual
  approximation of $(S,d_2)$ of width $w'\in \N_0$. Then the
  identity map $\id_S\colon(S,d_1) \to (S,d_2)$ is a quasisymmetry.
\end{lemma}

\begin{proof}
Our hypotheses imply that  $\{\X^n\}$ is a quasi-visual approximation of width $\max\{w,w'\}$ for both $(S, d_1)$ and $(S, d_2)$. Hence, without loss of generality, we may assume that $w'=w$. 

  We need to find a homeomorphism $\eta\colon [0,\infty) \ra
  [0,\infty)$ such that for all  $t>0$ and $x,y,z\in S$ the following implication holds:
  \begin{equation} \label{eq:qsimpl}
     d_1(x,y) \leq t d_1(x,z)
    \Rightarrow
    d_2(x,y) \leq \eta(t) d_2(x,z).
  \end{equation}
  
 So let $t>0$ and $x,y,z\in S$ with $d_1(x,y) \leq t d_1(x,z)$
  be arbitrary. We may assume that $x\neq y$. Then also $x\ne z$.
  
  In the following, we denote by $X^k$ $k$-tiles that
  contain the point $x$ for various $k\in\N_0$.
 
  By condition~\eqref{eq:qv_approx4p} in
  Lemma~\ref{lem:sub_shrink} and by \eqref{eq:qv_approx8} 
  there are constants $\rho,\tau \in (0,1)$ and $C\ge 1$ independent
  of $t,x,y,z$ such that
  \begin{equation}
    \label{eq:quantver}
    \diam_i(X^{k+\ell}) 
    \le
    C \rho^\ell \diam_i(X^{k})
 \end{equation}
 and 
   \begin{equation}
    \label{eq:quantvertau}
    \diam_i(X^{k+\ell}) 
    \ge
     \tau^\ell \diam_i(X^{k})
 \end{equation} for all $k,\ell\in \N_0$ and $i\in \{1,2\}$. 
 Note that $\ell=0$ is allowed here, because in this case
 \eqref{eq:quantver} and  \eqref{eq:quantvertau} are
 trivial.

Now let $m\coloneqq m_w(x,y)\in \N_0$ and
  $n \coloneqq m_w(x,z)\in \N_0$, where $m_w$ is the $U_w$-proximity function with respect to $\{\X^n\}$ and width $w=w'$ as  in Definition~\ref{def:mxy}. 
    By Lemma~\ref{lem:qv_metric} there exists a constant 
  $C(\asymp)>0$ independent of $x,y,z$ such that
  \begin{equation}\label{eq:Kimpl}
d_i(x,y) \asymp   \diam_i(X^m)  \text { and } 
d_i(x,z) \asymp   \diam_i(X^n) 
  \end{equation}
  for $i=1,2$.  This means that by  enlarging the original constant $C$ in \eqref{eq:quantver} if necessary (and thus avoiding introducing new constants) we may assume that 
   \begin{equation}\label{eq:Kimpl1}
\frac 1Cd_i(x,y) \le \diam_i(X^m)\le C  d_i(x,y) 
\end{equation}
and
\begin{equation}\label{eq:Kimpl2}
  \frac 1Cd_i(x,z)\le   \diam_i(X^n)\le C  d_i(x,z)
\end{equation}
for $i=1,2$. Together with our assumption $d_1(x,y) \leq t d_1(x,z)$,  this implies in
particular that
\begin{equation}
  \label{eq:diamXmXn}
  \diam_1(X^m) \leq t C^2  \diam_1(X^n).
\end{equation}

  The idea of the proof is now to translate the previous inequalities 
  for the metric $d_1$ to  a relation between  $m$, $n$, and  $t$. Using  this for the metric $d_2$, we will obtain  a bound for  $\diam_2(X^m)$ in terms of  $\diam_2(X^n)$. This will lead to the desired estimate for $d_2(x,y)/d_2(x,z)$ by \eqref{eq:Kimpl1} and   \eqref{eq:Kimpl2}. We consider two cases, where we will use the constant  $\nu = \log(1/\rho)/\log(1/\tau)\in
 (0,1]$ with $\rho= \tau^\nu$(see \eqref{eq:def_nu}).

  \smallskip
  \emph{Case 1:} $m\leq n$. Then by
  \eqref{eq:diamXmXn} and \eqref{eq:quantver} we
  have \begin{align*} \diam_1(X^m) &\leq tC^2 \diam_1(X^n) = tC^2
    \diam_1(X^{m+(n-m)})
         \\
                                   &\leq tC^3\rho^{n-m}
                                     \diam_1(X^m),
  \end{align*}
  and so  $t C^3\rho^{n-m}\geq 1$.  Now    \eqref{eq:def_nu} shows that
  \begin{equation}
     \label{eq:tau_est}
  \tau^{m-n}= \rho^{(m-n)/\nu} \leq (tC^3)^{1/\nu}=C^{3/\nu}\, t^{1/\nu}.
   \end{equation}
  It follows that
    \begin{align*}
      d_2(x,y)& \le C \diam_2(X^m)
      &&    \text{by \eqref{eq:Kimpl1}}\\
              &\leq C \tau^{-(n-m)}\diam_2(X^{m+(n-m)})
      &&\text{by \eqref{eq:quantvertau}}\\
              &=C \tau^{m-n}\diam_2(X^{n})
      && \\
              &\leq C^{1+3/\nu} \,t^{1/\nu} \diam_2(X^n)
      &&\text{by \eqref{eq:tau_est}} \\
              &\le C^{2+3/\nu} \,t^{1/\nu} d_2(x,z)
      &&\text{by \eqref{eq:Kimpl2}}\\
              &= \eta_1(t) d_2(x,z), &&
     \end{align*}
where we set $\eta_1(t)\coloneqq C_1 \,t^{1/\nu}$ with $C_1 = C^{2+3/\nu}$. Clearly $\eta_1\:[0,\infty)\ra [0,\infty)$ is a
homeomorphism. 
  
   \smallskip
  \emph{Case 2:} $m> n$. 
 Then  by  \eqref{eq:diamXmXn} and  \eqref{eq:quantvertau} we have 
   \begin{align*}
   tC^2 \diam_1(X^n)  &\geq 
    \diam_1(X^m) = \diam_1(X^{n+(m-n)}) \\
    &\geq 
   \tau^{m-n} \diam_1(X^n), 
  \end{align*}
  and so $tC^2\ge \tau^{m-n}$.  Now again \eqref{eq:def_nu}
 shows that 
\begin{equation} \label{eq:rho_est}
    \rho^{m-n}
    =
    \tau^{(m-n)\nu}
    \leq
    C^{2\nu} t^{\nu}.     
\end{equation} It follows  that 
 \begin{align*}
   d_2(x,y)& \le C \diam_2(X^m) =C  \diam_2(X^{n+(m-n)})
   &&\text{by \eqref{eq:Kimpl1}}
   \\
           &\leq  C^2\rho^{m-n}\diam_2(X^{n})
   &&\text{by  \eqref{eq:quantver}}
   \\
           &\le  C^{2+ 2\nu} \,t^{\nu} \diam_2(X^n)
   && \text{by \eqref{eq:rho_est}}\\
           &\le C^{3+2\nu} t^{\nu} d_2(x,z)
   && \text{by  \eqref{eq:Kimpl2}}\\
           &=\eta_2(t) d_2(x,z), &&
 \end{align*}
where we set $\eta_2(t)\coloneqq C_2 \,t^{\nu}$ with $C_2 =
C^{3+2\nu}$. Clearly, $\eta_2\:[0,\infty) \ra [0,\infty)$ is a
homeomorphism.

We now set
\begin{equation}
  \label{eq:eta_powerqs}
  \eta(t)
  \coloneq
  K \max\{t^\nu, t^{1/\nu}\}
  \geq
  \max\{\eta_1(t), \eta_2(t)\}
\end{equation}
for $t\ge 0$, where $K= \max\{C_1, C_2\}\geq 1$.  Then
$\eta\:[0,\infty)\ra [0, \infty)$ 
is a homeomorphism satisfying \eqref{eq:qsimpl}.
The statement follows.
\end{proof}

Quasisymmetries with a distortion function as in
\eqref{eq:eta_powerqs} are called \emph{power
  quasisymmetries}. Our proof shows that $\id_S\colon (S,d_1) \to
(S,d_2)$ in Lemma~\ref{lem:qv_qs} is in fact of this type. 
This is not a  surprise and follows from the general fact that quasisymmetries between uniformly perfect metric spaces are power quasisymmetries 
(see \cite[Theorem~11.3]{He} and the related Corollary~\ref{cor:qv-are-uni-perfect} below).

\begin{proof}[Proof of Theorem~\ref{thm:qv-qs}]
  The statement follows immediately from Lemmas~\ref{lem:qv_qs_inv} and~\ref{lem:qv_qs}. 
\end{proof}

To provide a reference for future applications,
we rephrase Theorem~\ref{thm:qv-qs} in the following more general, but essentially equivalent form.

\begin{cor}\label{cor:qv-qs}
 Let $\varphi\: S\ra T$ be a bijection between 
bounded metric spaces $(S,d)$ and $(T,\varrho)$,  the sequence $\{\X^n\}$ be a quasi-visual approximation of $(S,d)$, and $\{\mathbf{Y}^n\}$ be the image of $\{\X^n\}$ under $\varphi$, that is, \[
\mathbf{Y}^n\coloneqq \{ \varphi(X):X\in \X^n\} \text{ for } n\in \N_0.
\]
Then $\varphi$ is a quasisymmetry if and only if  $\{\mathbf{Y}^n\}$ is a quasi-visual 
approximation of $(T,\varrho)$.   
\end{cor}

\begin{proof} We define a new metric $d'$ on $S$ by setting
\[
d'(x,y)\coloneqq \varrho(\varphi(x), \varphi(y))
\]
for $x,y\in S$. Then the map $\varphi$ is an isometry 
between $(S,d')$ and $(T,\varrho)$.

This and Theorem~\ref{thm:qv-qs}  imply that $\varphi\:(S,d)\ra (T, \varrho)$ is a quasisymmetry if and only if 
the identity map $\id_S\: (S,d)\ra (S,d')$ is a quasisymmetry 
if and only if $\{\X^n\}$ a  quasi-visual approximation of
$(S,d')$ if and only if the image 
$\{\mathbf{Y}^n\}$ of $\{\X^n\}$ under $\varphi$ (as defined in the statement) is a  quasi-visual approximation of  $(T,\varrho)$. The statement follows. 
\end{proof}

The map $\id_S \colon (S,d_1) \to (S,d_2)$ in Theorem~\ref{thm:qv-qs} and the map
$\varphi\colon (S,d)\ra (T,\varrho)$  in  Corollary~\ref{cor:qv-qs}
are again power quasisymmetries.

\section{Combinatorics of quasi-visual approximations}
\label{sec:comb-visu-appr}
In this section we address three closely related questions. To formulate them, suppose
$\{\X^n\}$ is a sequence of covers of a set $S$.
\begin{enumerate}[label=(\arabic*)]
\smallskip
\item Let $S$ be equipped with a metric $d$.  Suppose $\{\X^n\}$ is a 
  (quasi-) visual approximation of width $w\in \N_0$ for $(S,d)$ and 
  $m_w$ is the $U_w$-proximity function with respect to $\{\X^n\}$ as in 
  Definition~\ref{def:mxy}. Which properties does $m_w$ have?
\smallskip
\item
  Is it possible to define a metric on $S$ that is (quasi-)visual for
  the cover sequence $\{\X^n\}$?   
\smallskip
\item
  Let $S$ be equipped with a metric $d$ and
suppose  
$\{\X^n\}$ is a quasi-visual approximation for $(S,d)$. Is there a
(different) metric $\varrho$ on $S$ such that $\{\X^n\}$ is
a visual approximation for $(S,\varrho)$?  
\end{enumerate}

To answer these questions, we introduce a combinatorial version of (quasi-)visual approximations of a metric space.

\begin{definition}[Combinatorially visual approximations]
  \label{def:comb_visual}
  Let $w\in \N_0$ and  $\{\X^n\}$ be a sequence of coverings of a set $S$ where
  $\X^0= \{S\}$. Suppose $m_w$ is the 
$U_w$-proximity function on $S$ with respect to
  $\{\X^n\}$ as in Definition~\ref{def:mxy}. We say that
  $\{\X^n\}$ is a \emph{combinatorially visual approximation  of width $w$} for $S$ if there exists a constant $C\geq 0$
  such that the 
  following conditions are true: 
  \begin{enumerate}
  \smallskip
  \item
    \label{item:comb_visuali}
    $m_w(x,y)\in \N_0$ whenever $x,y\in S$ are distinct.  
  \smallskip
  \item%[(ib)]
    \label{item:comb_visualii}
    For all $n\in \N_0$ and $X^n\in \X^n$  there exist $x,y\in X^n$ such that 
    \[
      m_w(x,y)  \leq n + C.  
  \]

  \smallskip
  \item%[(ii)]
    \label{item:comb_visualiii}
    Let $n\in \N_0$,  $X^n, Y^n\in \X^n$, and suppose that 
    \[
    U_w(X^n) \cap U_w(Y^n) = \emptyset 
    \](see \eqref{eq:def_UwX^n}).
    Then
  \[
      m_w(x,y) \leq n +C
      \text{ for all $x\in X^n$, $y\in Y^n$.}
    \]
    
    \smallskip
  \item%[(iii)]
    \label{item:comb_visualiv}
    For all $x,y,z\in S$, we have  
  \begin{equation*}
    m_w(x,y) \geq \min\{m_w(x,z), m_w(z,y)\} - C.
  \end{equation*}
  \end{enumerate}
\end{definition}

Note that in condition~\ref{item:comb_visualii} we always have $m_w(x,y)\ge n$ whenever $x,y\in X^n$ by \eqref{eq:mxy_Xn}. Condition~\ref{item:comb_visualii} is closely related to \ref{item:visual1} in
Definition~\ref{def:visual}, while condition~\ref{item:comb_visualiii} above
corresponds to \ref{item:visual2} in
Definition~\ref{def:visual}. The quantity $m_w(x,y)$ may in many
ways be thought of as a \emph{Gromov product} of the points $x,y\in S$. If the space $S$ satisfies condition~\ref{item:comb_visualiv}, where the $U_w$-proximity function $m_w$ is replaced
by the Gromov product, then it is called \emph{Gromov hyperbolic}; see \cite{BS,GH} and Section~\ref{sec:tile-graph} below.

The following theorem answers the questions from the beginning of this section.

\begin{theorem}
  \label{thm:_comb_qv} 
  Let $\{\X^n\}$ be a sequence of covers of a set $S$ and
  $w\in \N_0$. Then
  the following are equivalent:
  \begin{enumerate}
  \smallskip
  \item
    \label{item:thm_comb_qv2}
    There is a metric $\varrho$ on $S$ such that $(S,\varrho)$ is bounded and $\{\X^n\}$ is a
    visual approximation of width $w$ for $(S,\varrho)$.
    \smallskip
  \item
    \label{item:thm_comb_qv1}
    There is a metric $d$ on $S$ such that $(S,d)$ is bounded and $\{\X^n\}$ is a
    quasi-visual approximation of width $w$ for $(S,d)$.
    \smallskip  
  \item
    \label{item:thm_comb_qv3}
    $\{\X^n\}$ is a combinatorially visual approximation of width $w$ for $S$ as in
    Definition~\ref{def:comb_visual}.  
  \end{enumerate}
\end{theorem}

The implication \ref{item:thm_comb_qv2}$\,\Rightarrow\,$\ref{item:thm_comb_qv1} 
was shown in Lemma~\ref{lem:visual-is-also-quasi}. 
We will prove the implications \ref{item:thm_comb_qv1}$\,\Rightarrow\,$\ref{item:thm_comb_qv3} and 
\ref{item:thm_comb_qv3}$\,\Rightarrow\,$\ref{item:thm_comb_qv2} in Sections~\ref{sec:diam-levels-dist} and 
\ref{sec:from-quasi-visual}, respectively. We also note that the metrics $\varrho$ and $d$ on $S$ from 
\ref{item:thm_comb_qv2} and \ref{item:thm_comb_qv1} are quasisymmetrically equivalent (as follows from Theorem~\ref{thm:qv-qs} and Lemma~\ref{lem:visual-is-also-quasi}). Before we turn to the proofs, we record the following corollary; compare Proposition~\ref{prop:vis_width1_ex}. 

\begin{cor}\label{cor:qv-are-uni-perfect}
A bounded metric space $(S,d)$ admits a quasi-visual approximation if and only if it is uniformly perfect.
\end{cor}
\begin{proof} 
  Necessity follows directly from Proposition~\ref{prop:vis_width1_ex} and 
  Lemma~\ref{lem:visual-is-also-quasi}. To show  sufficiency, suppose $(S,d)$ has a quasi-visual approximation $\{\X^n\}$. By   Theorem~\ref{thm:_comb_qv}, there is a metric $\varrho$ on $S$ such that $(S, \varrho)$ is bounded and $\{\X^n\}$ is a visual approximation of $(S,\varrho)$. But then $\{\X^n\}$ is also a quasi-visual approximation for $(S,\varrho)$ by Lemma~\ref{lem:visual-is-also-quasi}, and thus $\id_S\colon (S,\varrho)\to (S,d)$ is a quasisymmetry by Theorem~\ref{thm:qv-qs}. Since $(S,\varrho)$ is uniformly perfect by Proposition~\ref{prop:vis_width1_ex} and uniform perfectness is preserved under quasisymmetries (see Lemma~\ref{lem:qs_properties}), we conclude that $(S,d)$ is uniformly perfect as well. 
\end{proof}

\subsection{From quasi-visual to combinatorially visual}
\label{sec:diam-levels-dist}

Here we will show the implication \ref{item:thm_comb_qv1}$\,\Rightarrow\,$\ref{item:thm_comb_qv3} of
Theorem~\ref{thm:_comb_qv}. First, we need some preparation.
Throughout this subsection, $\{\X^n\}$ is a quasi-visual
approximation of (fixed) width $w\in\N_0$ of a bounded metric
space $(S,d)$.

\begin{lemma}
  \label{lem:tiles_level_diam} Let $k,n\in \N_0$, $X^k \in
  \X^k$, $X^n\in \X^n$, and 
  $X^k \cap X^n \ne \emptyset$.  Assume that
  \begin{equation*}
    \diam(X^n) \lesssim \diam(X^k).
  \end{equation*}
  Then there is a constant $C\geq 0$, depending on $C(\lesssim)$
  and the constants from Definition~\ref{def:qv_approx},
  such that 
  \begin{equation*}
    n \geq k - C.
  \end{equation*}
\end{lemma}

\begin{proof}
  If $n\geq k$, there is nothing to prove; so we may  assume  that $n< k$
  and set $\ell \coloneqq k -n \in \N$. 
 Let $\rho\in (0,1)$ and $C_1$ be the constants from
  Lemma~\ref{lem:sub_shrink}, and let $C_2= C(\lesssim)$ be the
  constant from the hypotheses of the lemma (i.e., $\diam(X^n) \leq C_2
  \diam(X^k)$). Then
   \begin{align*}
    \diam(X^k)
    =
    \diam(X^{n+\ell})
    \leq
    C_1\rho^\ell \diam(X^n)
    \leq
    C_1 C_2 \rho^\ell \diam(X^k). 
  \end{align*}
  It follows that $C_1C_2 \rho^\ell \geq 1$, or equivalently
  \begin{equation*}
    \ell \leq  C\coloneqq\frac{\log(C_1 C_2)}{\log(1/\rho)}\in [0,\infty).
  \end{equation*}
  Thus
    $n\geq k - C$,
  and the statement follows.   
\end{proof}

\begin{proposition}
  \label{prop:qv_comb}
  Let $\{\X^n\}$ be a quasi-visual approximation of width $w\in\N_0$ for a bounded
  metric space $(S,d)$. Then $\{\X^n\}$ is a
  combinatorially visual approximation of width $w$ for $S$.  
\end{proposition}
This means that the $U_w$-proximity function $m$, defined with respect to the sequence $\{\X^n\}$, 
satisfies the conditions in Definition~\ref{def:comb_visual}. 

\begin{proof} First note that $\X^0=\{S\}$ by the assumption in   Definition~\ref{def:qv_approx}.  We will now verify the four conditions 
  \ref{item:comb_visuali}--\ref{item:comb_visualiv}
  in Definition~\ref{def:comb_visual} for the function $m$.
  
  \smallskip
\ref{item:comb_visuali} This was shown in Lemma~\ref{lem:mxy_qv}. 

 \smallskip
\ref{item:comb_visualii} To see this, let $n\in \N_0$ and $X^n \in \X^n$ be
  arbitrary.  We can choose $x,y\in X^n$ with $d(x,y)
  \geq \frac{1}{2} \diam(X^n)$. Then $x\ne y$ and so $m\coloneq m(x,y)\in \N_0$ by \ref{item:comb_visuali}. We can find a tile $X^{m}\in \X^m$ with $x\in X^m$. Then by 
  Lemma~\ref{lem:qv_metric} we have 
  \begin{align*}
    \diam(X^{m}) \asymp d(x,y) \geq \frac{1}{2} \diam(X^n), 
  \end{align*}
  and so Lemma~\ref{lem:tiles_level_diam} implies there is a uniform constant
  $C_1\geq0$ such that
  \begin{equation*}
    m=m(x,y) \leq n+ C_1. 
  \end{equation*}
  Condition~\ref{item:comb_visualii} follows.
    
  \smallskip
  \ref{item:comb_visualiii} 
  Let $n\in \N_0$,  $X^n, Y^n\in \X^n$ be tiles such that 
    $U_w(X^n) \cap U_w(Y^n) = \emptyset$, and $x\in X^n$, $y\in Y^n$. Then $x\ne y$ and so again $m\coloneq m(x,y)\in \N_0$. We choose a tile $X^{m}\in \X^m$ with $x\in X^m$.
  Then using Lemma~\ref{lem:qv_metric} and condition~\ref{item:qv_approx2} in
  Definition~\ref{def:qv_approx}, we see that  
  \begin{align*}
    \diam(X^m)
    \asymp
    d(x,y)
    \geq
    \dist(X^n,Y^n)
    \gtrsim
    \diam(X^n).
  \end{align*}
  Thus it follows from Lemma~\ref{lem:tiles_level_diam} that
  \begin{equation*}
   m= m(x,y) \leq n+C_2
  \end{equation*}
  for a uniform constant $C_2\geq0$, proving
  \ref{item:comb_visualiii}. 

  \smallskip
  \ref{item:comb_visualiv} Let $x,y,z\in S$ be arbitrary. If these points are not all distinct, then the desired inequality holds with $C=0$ as follows  Lemma~\ref{lem:mxy_qv}.  
  So we may assume that the points $x$, $y$, $z$ are all  distinct.  Then 
  $m\coloneq m(x,y)$, $n\coloneq m(x,z)$, and $k\coloneq m(y,z)$  belong to $\N_0$. So we can find tiles $X^n\in \X^n$ and $Y^k\in \X^k$ with $x\in X^n$ and $y\in Y^k$.  
  
  Without loss of generality, we may assume that \[\diam(X^{n}) \geq
  \diam(Y^k).\] 
  We now choose a tile $X^m\in \X^m$ with $x\in X^m$.
  Using Lemma~\ref{lem:qv_metric} we see that 
  \begin{align*}
    \label{eq:XYZ1}
    \diam(X^{m}) &\asymp d(x,y)\le 
    d(x,z)+d(y,z)\\
    &\lesssim
    \diam(X^{n}) + \diam(Y^{k}) \leq 2\diam(X^n).
  \end{align*} 
  Lemma~\ref{lem:tiles_level_diam} then implies that for a uniform constant $C_3\geq0$ we have 
  \begin{align*}
    m(x,y)=m &\geq n-C_3 =m(x,z) - C_3\\
    & \geq \min\{m(x,z), m(z,y)\} - C_3.
  \end{align*}
Condition~\ref{item:comb_visualiv} follows. 

\smallskip
The previous considerations show  that the required conditions in Definition~\ref{def:comb_visual} hold for the $U_w$-proximity function $m$  simultaneously with   the  constant $C\coloneqq \max\{C_1, C_2, C_3\}$.
\end{proof}

\subsection{From combinatorially visual to visual}
\label{sec:from-quasi-visual}

We now prove the implication \ref{item:thm_comb_qv3}$\,\Rightarrow\,$\ref{item:thm_comb_qv2} in
Theorem~\ref{thm:_comb_qv}.

Throughout this subsection, $\{\X^n\}$ is a combinatorially
visual approximation of width $w\in \N_0$ for a set $S$. 
The construction of a visual metric on $S$ follows a standard
procedure from the theory of Gromov hyperbolic spaces. 
Namely, we fix $\Lambda>1$ and define a function $q\colon S\times S\to [0,\infty)$ by setting 
\[
  q(x,y)\coloneqq\Lambda^{-m(x,y)}
\]
for $x,y\in S$,
where $m$ denotes  the $U_w$-proximity function on $S$ with respect to $\{\X^n\}$. Here we use the convention that $\Lambda^{-\infty}=0$. In particular, $q(x,y)=0$ if $x=y$. While $q$ will not be a metric on $S$ in general, it is bi-Lipschitz
equivalent to a metric when $\Lambda>1$ is sufficiently close to $1$.

\begin{lemma}
  \label{lem:qL_qmetric}
  The function $q$ as defined above is a quasi-metric on the set $S$ (see Definition~\ref{def:qmetric}). If
  $\Lambda >1$ is sufficiently small, then $q$ is bi-Lipschitz
  equivalent to a metric $\varrho$ on~$S$. 
\end{lemma}

\begin{proof}
  Let $x,y,z\in S$ be arbitrary.
  Using condition~\ref{item:comb_visuali} in Definition~\ref{def:comb_visual},
  we see that $q(x,y) =0$ if and only
  if $x=y$. Clearly, we also have $q(x,y) =
  q(y,x)$, since the $U_w$-proximity function $m$ is symmetric. Finally, using condition~\ref{item:comb_visualiv} with the associated constant $C\ge0$ in Definition~\ref{def:comb_visual},
  we see that 
  \begin{align*}
    q(x,y)
    &=
    \Lambda^{-m(x,y)}
    \leq
      \Lambda^C \max\{\Lambda^{-m(x,z)},\, \Lambda^{-m(z,y)}\}
      \\
    &=
    \Lambda^C \max\{q(x,z),\, q(z,y)\}.
  \end{align*}
  It follows that $q$ is a $\Lambda^C$-quasi-metric on $S$.

  To show the second statement, let $\Lambda>1$ be sufficiently
  small such that $\Lambda^C\leq 2$. Theorem~\ref{thm:qmetric_metric} now implies that there is a metric $\varrho$ on $S$
  that is bi-Lipschitz equivalent to $q$. The proof is complete.
\end{proof}

\begin{proposition}
  \label{prop:comb_qv_2_visual}
  Let $\{\X^n\}$ be a combinatorially visual  approximation of width $w\in \N_0$ for a set $S$. Suppose $\varrho$ is a metric on $S$ provided by Lemma~\ref{lem:qL_qmetric} for some $\Lambda>1$.  Then $\{\X^n\}$ is a visual approximation of width $w$ for
  $(S,\varrho)$ with the visual parameter $\Lambda$. 
\end{proposition}

\begin{proof}
  Let $\varrho$ be a metric on $S$ provided by Lemma~\ref{lem:qL_qmetric} for some $\Lambda>1$. Then $\varrho(x,y)
    \asymp
    \Lambda^{-m(x,y)}$ for all $x,y\in S$, where $m$ is the $U_w$-proximity function with respect to the sequence $\{\X^n\}$. In the ensuing proof, metric concepts refer to this metric $\varrho$. 
  
  First note that $\X^0= \{S\}$ by the assumption in Definition~\ref{def:comb_visual}. Now let $n\in \N_0$ and $X^n\in \X^n$ be arbitrary. Using
  \eqref{eq:mxy_Xn}, we see that
  \begin{align*}
    \varrho(x,y)
    \asymp
    \Lambda^{-m(x,y)}
    \leq
    \Lambda^{-n}
  \end{align*}
  for all $x,y\in X^n$. Taking the supremum over all $x,y\in X^n$, we conclude that 
  \begin{align*}
    \diam(X^n) \lesssim \Lambda^{-n}.
  \end{align*} 
By condition~\ref{item:comb_visualii} in
  Definition~\ref{def:comb_visual}, there are $x,y\in X^n$
  with 
  \begin{align*}
    \varrho(x,y)
    \asymp
    \Lambda^{-m(x,y)}
    \geq \Lambda^{-C} \Lambda^{-n}. 
  \end{align*}
  This shows that
  \begin{align*}
    \diam(X^n) \gtrsim \Lambda^{-n}, 
  \end{align*}
  and it follows that $\diam(X^n) \asymp \Lambda^{-n}$, where $C(\asymp)$ is a uniform constant independent of $n\in \N_0$ and $X^n\in \X^n$. This establishes condition~\ref{item:visual1} in Definition~\ref{def:visual}.

  To see the second condition of visual approximations, 
let $n\in \N_0$ and  $X^n,
  Y^n\in \X^n$ be 
  $n$-tiles with $U_w(X^n) \cap U_w(Y^n)= \emptyset$. Then condition~\ref{item:comb_visualiii} in Definition~\ref{def:comb_visual} implies that 
  for all $x\in X^n$ and  $y\in Y^n$ we have
  \begin{align*}
    \varrho(x,y)
    \asymp
    \Lambda^{-m(x,y)}
    \geq
    \Lambda^{-C} \Lambda^{-n}.
  \end{align*} 
 Taking the infimum over all $x\in X^n$ and $y\in Y^n$, we obtain 
  \begin{align*}
    \dist(X^n, Y^n) \gtrsim \Lambda^{-n}, 
  \end{align*}
 and so condition~\ref{item:visual2} in
  Definition~\ref{def:visual} holds.

  This finishes the proof that $\{\X^n\}$ is a visual
  approximation of width $w$ for $(S,\varrho)$ with  visual parameter $\Lambda$. 
\end{proof}

Theorem~\ref{thm:_comb_qv} is now an immediate consequence of Lemma~\ref{lem:visual-is-also-quasi} and Propositions~\ref{prop:qv_comb} and~\ref{prop:comb_qv_2_visual}.

\subsection{From quasi-visual to visual} 
We end this section with several natural applications of Theorem~\ref{thm:_comb_qv} to visual metrics. The first one is a criterion when a quasi-visual metric
is in fact a visual metric for a given sequence of covers of a bounded metric space.

\begin{lemma}
  \label{lem:qvisual_visual}
  Let $\{\X^n\}$ be a quasi-visual approximation of width $w\in\N_0$ for a bounded
  metric space $(S,d)$. Then $\{\X^n\}$ is a visual approximation of width $w$ for
  $(S,d)$ with  visual parameter $\Lambda >1$ if and only if
  \begin{equation*}
    d(x,y) \asymp \Lambda^{-m(x,y)}
  \end{equation*}
  for all $x,y\in S$, where $m$ is the $U_w$-proximity function on $S$ with respect to $\{\X^n\}$.
\end{lemma}

\begin{proof}
  The implication ``$\Rightarrow$'' follows directly from Corollary~\ref{cor:dxy_mxy}.

  \smallskip
  Conversely, assume that $\{\X^n\}$ is a quasi-visual approximation of width $w$ for $(S,d)$
  and $d(x,y)\asymp \Lambda^{-m(x,y)}$ for all $x,y\in S$ with a
  constant $\Lambda >1$. 
  By Theorem~\ref{thm:_comb_qv} we know
  that $\{\X^n\}$ is also a combinatorially visual approximation of width $w$ for $S$. We denote by  $C\geq0$ the uniform constant as in  Definition~\ref{def:comb_visual} for the $U_w$-proximity function $m$. 
  
  Let $n\in \N_0$ and $X^n\in \X^n$ be
  arbitrary. By \eqref{eq:mxy_Xn}, we know that $m(x,y)\geq n$ for all $x,y\in X^n$. At the same time, from condition~\ref{item:comb_visualii} in
  Definition~\ref{def:comb_visual} we know that $m(x,y)\leq n + C$ for some $x,y\in X^n$. It follows that
  \begin{align*}
    \diam(X^n) &= \sup\{d(x,y) : x,y \in X^n\} \\
    &\asymp \sup\{\Lambda^{-m(x,y)} : x,y \in X^n\} = \Lambda^{-\inf\{m(x,y) :\,  x,y \in X^n\}} \\ 
    &\asymp \Lambda^{-n},
  \end{align*}
  meaning that condition~\ref{item:visual1} in Definition~\ref{def:visual}
  is satisfied.

  Now let $n\in\N_0$ and $X,Y\in \X^n$ be tiles with $U_w(X^n) \cap U_w(Y^n)
  =\emptyset$. Then using condition~\ref{item:comb_visualiii} in
  Definition~\ref{def:comb_visual} 
  we obtain 
  \begin{align*}
    \dist(X,Y) &= \inf\{d(x,y): x\in X, \, y\in Y\} \\
    &\asymp \inf\{\Lambda^{-m(x,y)}: x\in X, \, y\in Y\} \gtrsim \Lambda^{-n},
  \end{align*}
  meaning that condition~\ref{item:visual2} in Definition~\ref{def:visual}
  is satisfied. Since $\X^0 = \{S\}$ by the assumption in Definition~\ref{def:qv_approx}, we conclude that $\{\X^n\}$ is a visual approximation of width $w$ for
  $(S,d)$ with the visual 
  parameter~$\Lambda$. 
  This shows the second implication,
 and  the proof of the lemma is complete.  
\end{proof}

Our second application is a characterization of visual
approximations (and metrics), which shows how the definition of a visual metric used in
\cite{BM} relates to the one used in this paper.

\begin{cor}
  \label{cor:visul_charac}
 Let $\{\X^n\}$ be a sequence of covers of a bounded
  metric space $(S,d)$, where $\X^0=\{S\}$. Then $\{\X^n\}$ is a visual approximation of width $w\in \N_0$ for
  $(S,d)$ with visual parameter $\Lambda >1$ if and only if $\{\X^n\}$ is a combinatorially visual approximation of width $w$ for $S$ and 
  \begin{equation*}
    d(x,y) \asymp \Lambda^{-m(x,y)}
  \end{equation*}
  for all $x,y\in S$, where $m$ is the $U_w$-proximity function on $S$ with respect to $\{\X^n\}$.
\end{cor}

\begin{proof}
      The implication ``$\Rightarrow$'' follows immediately from Theorem~\ref{thm:_comb_qv} and Corollary~\ref{cor:dxy_mxy}. The other implication is established in the same way as in the proof of necessity in  Lemma~\ref{lem:qvisual_visual}.
\end{proof}

\section{The tile graph}
\label{sec:tile-graph}

Here we connect the setup from the previous sections to the
theory of Gromov hyperbolic spaces; for general background on this topic see \cite{Gr, GH, BS}. 

Suppose $\{\X^n\}$ is a sequence of covers of a set $S$ such that 
$\X^0=\{S\}$. We then let $$\X:=\bigsqcup_{n\in \N_0} \X^n$$  be the formal disjoint union of all $\X^n$.  
This means that two sets $X^n\in \X^n$ and $X^k\in \X^k$ for $n\neq
k$ are distinct elements of $\X$, 
even though they may be the same as subsets of $S$. 
We call each element $X\in \X$ a \emph{tile} and denote by $|X|$ its \emph{level}, that is, $|X|= n$ when $X$ represents a set in $\X^n$.

 We define the \emph{tile graph $\Gamma= \Gamma(\{\X^n\})$}
  associated with $\{\X^n\}$ as follows. 
 The vertices of $\Gamma$ are given by the elements 
of
 $\X$. Two distinct vertices of $\Gamma$, represented  by tiles $X,Y\in \X$, are connected by an edge if
and only if
\begin{equation}\label{eq:incidence}
  X\cap Y \neq \emptyset
  \text{ and }
  \big|\abs{X}- \abs{Y}\big|\leq 1.
\end{equation}
In other words, two distinct tiles are connected by an edge if they intersect
and their levels differ by at most $1$.

We denote by $\abs{X-Y}$ the combinatorial  distance between two vertices $X,Y\in \X$ in $\Gamma$. This defines a metric on  $\X$. In the following, when we speak of the tile graph $\Gamma$,  
we will often just  consider it as a metric space with $\X$ as the underlying set and the combinatorial distance as the metric,  
but sometimes we will also use the structure of $\Gamma$ as a graph by using the incidence relation for tiles given in \eqref{eq:incidence}.

The \emph{Gromov product}
on $\Gamma$ is then defined as
\begin{equation}
  \label{eq:def_Gromov}
  (X\cdot Y)
  \coloneqq
  \frac{1}{2}\big(\abs{X} + \abs{Y} - \abs{X-Y}\big) 
\end{equation}
for  $X,Y\in \X$. More precisely, this is the Gromov product with
respect to the base point $S$, that is, the (only) tile of level
$0$ in $\X$ (see \cite[Section~3]{BS00} and \cite[Corollary~1.1.B]{Gr}). 

The main result of this section is the following statement.

\begin{theorem}
  \label{thm:Gromov_hyp}  
Let  $\{\X^n\}$ be a combinatorially visual approximation of a set $S$. Then the tile graph $\Gamma=\Gamma(\{\X^n\})$ is \emph{Gromov hyperbolic}, that is, there
exists a constant $C\geq 0$ such that
  \begin{equation*}
    (X\cdot Y) \geq \min \{(X \cdot Z), (Z\cdot Y)\} -C
  \end{equation*}
  for all $X,Y,Z\in \X=\bigsqcup_{n\in\N_0} \X^n$.
\end{theorem}

For the proof we need some preparation. 
For the remainder of  this section, we
assume that $\{\X^n\}$ is a combinatorially visual approximation of  width $w\in \N_0$  for a given  set $S$ and denote by $C_\comb$ the respective constant as in Definition~\ref{def:comb_visual}. 

It will be convenient to extend the proximity function
$m=m_w$ given in 
Definition~\ref{def:mxy} to $\X$. To this end, we define 
\begin{equation}
  \label{eq:defmXY}
  m(X,Y)
  \coloneqq
  \min\{m(x,y) : x\in X,\, y\in Y\}
\end{equation}
for
$X,Y\in \X$. 
Since $m(x,y)\in \N_0\cup\{\infty\}$, the minimum is indeed
attained. 

This new quantity satisfies an inequality analogous to the inequality \ref{item:comb_visualiv} in Definition~\ref{def:comb_visual}.

\begin{lemma}
  \label{lem:mXY_Gromov}
  For all $X,Y,Z\in \X$ we have
  \begin{equation*}
    m(X,Y) \geq \min\{m(X,Z), m(Z,Y)\} -C_\comb.
  \end{equation*}
\end{lemma}

\begin{proof}
  Let $X,Y,Z\in \X$ be arbitrary. Then, using \ref{item:comb_visualiv} of Definition~\ref{def:comb_visual}, we see that for all $z\in Z$,
  \begin{align*}
    m(X,Y)
    &=
    \min\{ m(x,y) : x\in X,\, y\in Y\}
    \\
    &\geq
      \min\{ \min\{m(x,z),\, m(z,y)\} - C_\comb : x\in X,\, y\in Y\}
    \\  
    &=
      \min\{\min_{x\in X}m(x,z),\, \min_{y\in Y}m(z,y)\} -C_\comb.
   \end{align*}
   Minimizing the above expression over $z\in Z$, we conclude that 
   \begin{align*}
    m(X,Y)
    &\geq
   \min\{m(X,Z), m(Z,Y)\} -C_\comb.\qedhere
  \end{align*}
\end{proof}

We will show that the quantity $m(X,Y)$ and the Gromov product $(X\cdot Y)$ are the same up to a uniformly bounded additive
error term. 

\begin{proposition}
  \label{prop:mXY_XY}
  There is a constant $C\geq 0$ such that
  \begin{equation*}
    m(X,Y) - C
    \leq
    (X\cdot Y)
    \leq
    m(X,Y) + C
  \end{equation*}
  for all $X,Y\in \X$. 
\end{proposition}

Given this proposition, the proof of Theorem~\ref{thm:Gromov_hyp} is immediate. 
\begin{proof}[Proof of Theorem~\ref{thm:Gromov_hyp}]
  Let $C_1\geq 0$ be the constant from
  Proposition~\ref{prop:mXY_XY}. Then, by Lemma~\ref{lem:mXY_Gromov}, we obtain  
  \begin{align*}
    (X\cdot Y)
    &\geq
    m(X,Y) -C_1
    \geq
    \min\{m(X,Z),\, m(Z,Y)\} - C_1-C_\comb
    \\
    &\geq
      \min\{(X\cdot Z) -C_1,\, (Z\cdot Y) - C_1\} -C_1-C_\comb
      \\
    &=
    \min\{(X\cdot Z),\, (Z \cdot Y)\} -2C_1 - C_\comb
  \end{align*}
  for all $X,Y,Z\in \X$. This shows the statement with the constant $C=2C_1+C_\comb$.
\end{proof}

It remains to prove Proposition~\ref{prop:mXY_XY}.

\begin{proof}[Proof of Proposition~\ref{prop:mXY_XY}] 
If we unravel the definitions, we see that in order to establish the statement, it is enough to find a constant $C\ge0$ such that 
  \begin{align}\label{eq:levgr}
       \abs{X}+ \abs{Y}-2m(X,Y) -C
    &\leq
    \abs{X-Y}\\
    &\leq
    \abs{X}+ \abs{Y}-2m(X,Y) +C\notag
  \end{align}
for  $X,Y\in \X$.

Now let $X,Y\in \X$ be arbitrary and define
\begin{equation*}
  n\coloneqq \abs{X},\, 
  k\coloneqq\abs{Y},
  \text{ and }m\coloneqq m(X,Y).
\end{equation*}
Without loss of generality, we may
assume that $n\leq k$. Let the minimum that determines $m(X,Y)$ be attained for $x_0\in X$ and $y_0\in Y$,  that is,  
\begin{align*}
  m(x_0, y_0) = m(X,Y). 
\end{align*}
Then there are $m$-tiles $X^m$ and $Y^m$ with $x_0 \in X^m$ and  $y_0
\in Y^m$ such that $U_w(X^m) \cap U_w(Y^m) \neq \emptyset$, where $w$ is the width of
the combinatorially visual approximation $\{\X^n\}$. This
in turn means that there is a chain of $m$-tiles
\begin{equation}
  \label{eq:chainXmYm}
  Z_0 = X^m,\, Z_1,\dots, Z_w,\, Z'_w, \dots, Z'_1,\, Z'_0 = Y^m
\end{equation}
 of length $2w+1$ from $X^m$ to $Y^m$.  

 Now the following upper bound for $m$ is valid.  

\smallskip
{\em Claim.} We have $m\le n+3C_\comb$.

\smallskip
Indeed, by \ref{item:comb_visualii} in Definition~\ref{def:comb_visual}
   there are points $x',x''\in X$ with $m(x',x'') \leq n+ C_\comb$. Thus using
   \ref{item:comb_visualiv} we obtain
   \begin{align*}
     n+C_\comb
     \geq
     m(x',x'')
     \geq
     \min\{m(x',x_0), m(x_0,x'')\} -C_\comb.
   \end{align*}
  Here we may assume that the last minimum is attained for $m(x',x_0)$.
  Then $m(x',x_0)\le n+2C_\comb$.
  Now by the choice of $x_0$ and $y_0$ we have 
  $m(x',y_0)\ge m(x_0, y_0)=m$. So 
  applying  \ref{item:comb_visualiv} in Definition~\ref{def:comb_visual} once more, we see that 
   \begin{align*}
     n+ 2C_\comb &\geq m(x_0, x')\\
     &\geq
     \min\{m(x_0,y_0), m(y_0,x')\} -C_\comb
     =
     m-C_\comb.
   \end{align*}
 The claim follows.

\smallskip
In order to establish \eqref{eq:levgr}, we will now consider two 
cases for $m\in \N_0$
depending on its relation to $n$: $n\le m$ and $m\leq n$.

\smallskip
{\em Case 1: $n\leq m$.}

\smallskip
Recall that we also have $|X|=n\le k=|Y|$. Then for 
 the lower bound in \eqref{eq:levgr} we observe that 
  \begin{align*}
    \abs{X-Y}
    &\geq
    k-n
    =
    n+ k -2n
    \geq
      n+ k -2m
    \\
    &=
    \abs{X} + \abs{Y} -2m(X,Y).
  \end{align*}

   To see the other inequality, assume for the moment that $m\le k$. Then there is a path in
  $\Gamma$ joining $X$ and $Y$ given as the concatenation of the following three paths:
  \begin{align*}
    &X=X^n, X^{n+1},\dots, X^{n+(m-n)} = X^m;
    \\
    &\text{the $m$-chain from $X^m$ to $Y^m$ in \eqref{eq:chainXmYm}};
    \\
    &Y^m, Y^{m+1}, \dots,
    Y^{m+(k-m)} = Y^k=Y.
  \end{align*}
  Here each $X^j$ is a $j$-tile with $x_0\in X^j$, and $Y^j$ is a $j$-tile
  with $y_0\in Y^j$. If $k\le m$ we replace the third path with a similar path of the form 
  \[
  Y^m, Y^{m-1}, \dots,
    Y^{m-(m-k)} = Y^k=Y.
  \]
  In both cases, this  concatenated path has length
  \begin{align*}
    N
    &=
    m-n +2w +1 + |k-m|&\\
    &\le
        m-n +2w +1 + |k-n|+|n-m| & \\
    &\le
      3C_\comb+2w+1+k-n+3C_\comb
      && \text{(by the Claim)}\\
    &=
    n+k -2n+ 6C_\comb +2w +1 &\\
    &\le n+k -2m+ 12C_\comb +2w +1 &&  \text{(by the Claim)}.
  \end{align*}
  It follows that
  \begin{align*}
    \abs{X-Y}\leq N \leq \abs{X} + \abs{Y} -2m(X,Y) +12C_\comb + 2w +1.
  \end{align*}
  
  These considerations show that in Case~1 inequality   \eqref{eq:levgr} is true with the uniform constant $C =12C_\comb+  2w
  +1$.

\smallskip
{\em Case 2:} $m\le n$. 

\smallskip

This  case is more involved.  First note that there is a path joining $X$ and $Y$ in
  $\Gamma$ given as the concatenation of the following three paths: 
  \begin{align*}
    &X=X^n, X^{n-1},\dots, X^{n-(n-m)} = X^m;  \quad 
    \\
    &\text{the $m$-chain from $X^m$ to $Y^m$ in \eqref{eq:chainXmYm}};
    \\
    &Y^m, Y^{m+1}, \dots,
    Y^{m+(k-m)} = Y^k=Y.
  \end{align*}
  Here again each $X^j$ is a $j$-tile with $x_0\in X^j$, and $Y^j$ is a $j$-tile
  with $y_0\in Y^j$.
 
  The concatenated path
  has length
  \begin{align*}
    \widetilde N
    =
    n-m+2w +1 + k-m
    =
    n+k -2m +2w +1,
  \end{align*}
  and so
  \begin{align*}
    \abs{X-Y}\leq \widetilde N = \abs{X} + \abs{Y} -2m(X,Y)+2w +1.
  \end{align*}
  This gives an upper bound for $|X-Y|$ as in 
\eqref{eq:levgr} with the uniform constant $C=2w +1$.

  \smallskip
  To see the other inequality, we first note that by Theorem~\ref{thm:_comb_qv}, we can
  equip $S$ with a visual metric $\varrho$ for $\{\X^n\}$. Let
  $\Lambda>1$ be the visual parameter of $\varrho$. Then we have 
  $\diam(Z) \asymp \Lambda^{-\abs{Z}}$ for every $Z\in \X$ (see
  Definition~\ref{def:visual}~\ref{item:visual1}). Here and in the following diameters are
  understood to be with respect to the metric $\varrho$. 

Now  let $X=X_0, X_1, \dots, X_N=Y$ be
  a path in $\Gamma$ that joins $X$ and $Y$ and has  minimal length $N=
  \abs{X-Y}$.
   Since the levels of consecutive tiles in the path
  may differ by at most $1$, we have
  \[
    \abs{X_j} \geq \abs{X} -j = n-j
                \text{ and }
             \abs{X_{N-j}} \geq \abs{Y} -j = k-j, 
\]
and so 
\[
    \diam(X_j) \lesssim \Lambda^{-n+j}
                 \text{ and }
                 \diam(X_{N-j}) \lesssim \Lambda^{-k+j}
\]
  for $j=0,\dots, N$. Moreover, we have 
  $\varrho(x_0, y_0) \asymp \Lambda^{-m}$, by
  Lemma~\ref{lem:qvisual_visual}. It now follows that for arbitrary $N'\in \{0, \dots, N\}$ we have 
  \begin{align*}
    \Lambda^{-m}\asymp
    \varrho(x_0, y_0)
    &\leq
      \diam(X_0\cup \dots \cup X_N) 
      \\
    &\leq
    \diam(X_0\cup \dots \cup X_{N'})
    + \diam(X_{N'+1}  \cup \dots \cup X_N) 
    \\
    &\leq
      \sum_{j=0}^{N'}\diam(X_j)
      + \sum_{j=0}^{N-N'-1} \diam(X_{N-j})
    \\
    &\lesssim
    \sum_{j=0}^{N'}\Lambda^{-n+j}+
    \sum_{j=0}^{N-N'-1}\Lambda^{-k+j}
    \lesssim \Lambda^{-n +N'} +
    \Lambda^{-k+N-N'}.
  \end{align*}
To obtain an essentially optimal inequality here, we choose $N'= \lfloor N_0/2\rfloor$, where $N_0= N-
  (k-n)$. With this choice of $N'$ it follows that 
  \[
  \Lambda^{-m} \lesssim \Lambda^{-n +N'} +
    \Lambda^{-k+N-N'} \asymp \Lambda^{-n+N_0/2}. 
      \]
  This implies that  there is a uniform constant $C'\geq 0$ such that
  \[
    m \geq n - \tfrac12 N_0 -C' = n -\tfrac{1}{2}(N-(k-n)) -C',
   \]
  and so 
\begin{align*}
   \abs{X-Y}&= N \geq n+k - 2m -2C'\\
    &=
    \abs{X} + \abs{Y} -2m(X,Y) -2C'. 
\end{align*}
This is the desired lower bound in \eqref{eq:levgr} and completes the argument in Case~2. 

\smallskip 
Since our two cases exhaust all possibilities,  
the proof of Proposition~\ref{prop:mXY_XY} is complete. 
\end{proof}
Theorem~\ref{thm:Gromov_hyp} now follows.

\section{Identifying the boundary at infinity with the space}
\label{sec:ident-bound-at}

In this section, we consider a quasi-visual approximation $\{\X^n\}$ of a bounded metric space $(S,d)$. By Theorem~\ref{thm:_comb_qv}, we know that $\{\X^n\}$ also provides a combinatorially visual approximation for $S$.  
According to
Theorem~\ref{thm:Gromov_hyp} this in turn implies that the associated tile graph $\Gamma=
\Gamma(\{\X^n\})$ is Gromov hyperbolic.    Hence, $\Gamma$ has a boundary at infinity
$\partial_\infty \Gamma$ that can be equipped with a visual
metric $d_\infty$ (we will review the relevant definitions momentarily). The question then arises how $(\partial_\infty \Gamma, d_\infty)$
relates to $(S, d)$.
As we will see in this section, under the assumption that 
$(S,d)$ is complete, we can naturally
identify $(S,d)$ and $(\partial_\infty \Gamma, d_\infty)$ by a quasisymmetric homeomorphism. Completeness of $(S,d)$  is
clearly necessary for such an identification as $(\partial_\infty \Gamma, d_\infty)$ is always complete (see \cite[Proposition~6.2]{BS00}).

For now we suppose  that $\{\X^n\}$ is a sequence of covers of a set $S$ with $\X^0=\{S\}$. As in Section~\ref{sec:tile-graph}, we can then define the   associated tile graph $\Gamma=\Gamma(\{\X^n\})$. Again we consider $\Ga$ as a metric space with the underlying set 
given by $\X:=\bigsqcup_{n\in \N_0} \X^n$ and the combinatorial distance  
as the metric on $\X$.  In particular, when we speak of a sequence 
 in $\Gamma$, we mean a sequence in the vertex set $\X$ of $\Gamma$, or equivalently, a sequence of tiles.
 
We now assume that $\Gamma$ is  Gromov hyperbolic. 
In order to define the boundary at infinity $\partial_\infty \Gamma$ of $\Gamma$, we say that a sequence $\{X_i\}$ in $\Gamma$ \emph{converges at
  infinity}\footnote{This is often called ``convergence \emph{to} infinity'' in the literature.} (in the sense of Gromov hyperbolic spaces) if 
 \begin{equation}\label{eq:convinfty}
 \lim_{i,j \to \infty} (X_i \cdot X_j) =\infty. 
 \end{equation}
 Note that, in this case, $|X_i|\to \infty$ as $i\to\infty$. 

If $\{X_i\}$ and $\{Y_i\}$ are two sequences in $\Ga$ converging at infinity,
 then we call them {\em equivalent}, written $\{X_i\}\sim \{Y_i\}$,  if 
\begin{equation}\label{eq:equiseqinfty}
\lim_{i\to \infty} (X_i \cdot Y_i) =\infty.
\end{equation}

It  easily follows from the Gromov hyperbolicity of $\Ga$ that $\sim$  defines an equivalence relation on the set of all sequences  $\{X_i\}$ 
in $\Ga$ that converge at infinity. We denote the equivalence class of such a sequence
$\{X_i\}$ under the relation $\sim$ by $[\{X_i\}]$. Then, by definition, the {\em boundary at infinity} 
 of $\Ga$ is the set of all such equivalences classes, i.e., 
\[
\partial_\infty \Gamma\coloneqq \{[\{X_i\}]: \{X_i\} \text { is a sequence in $\Ga$ 
converging at infinity} \}.
\]

We want to define a map from $S$ to $\partial_\infty \Gamma$ that will give a natural identification between these two sets. For this purpose, let $x\in S$ be arbitrary. Since  $\X^n$  is a cover of $S$, we can find $X^n\in \X^n$ with $x\in X^n$ for each $n\in \N_0$. It is easy to see that  $\{X^n\}$ is a geodesic ray in the tile graph
$\Gamma$. We call $\{X^n\}$ a \emph{natural geodesic} of
$x$. Note that, in contrast to the sequences in $\Ga$ considered earlier, the running index $n\in \N_0$ of the tiles in a natural geodesic $\{X^n\}$ also indicates the level $n$ of the tile~$X^n$. 

\begin{lemma}[Natural geodesics]
  \label{lem:nat_geo}
  For natural geodesics of points in $S$ the following statements are true: 
  \begin{enumerate}
  \smallskip
  \item
    \label{item:nat_geo1}
    Let $x\in S$ and $\{X^n\}$ be a natural geodesic of $x$. Then
    $\{X^n\}$ converges at infinity as in \eqref{eq:convinfty}.
  \smallskip
  \item
    \label{item:nat_geo2}
    If  $\{X^n\}$ and $\{Y^n\}$ are two natural geodesics of a  
   point $x\in S$, then these sequences are equivalent as in \eqref{eq:equiseqinfty}.
   \end{enumerate}
\end{lemma}

\begin{proof}
  \ref{item:nat_geo1}
  Let $\{X^n\}$ be a natural geodesic of a point $x\in S$. Then for all $n,k\in \N_0$ we have $\abs{X^n}= n$, $\abs{X^k} = k$, and $|X^n-X^k|=|n-k|$. 
  Hence  
  \[(X^n\cdot  X^{k}) =\tfrac12(n+k-|n-k|)= \min\{n,k\} \to \infty \]
as $n,k\to \infty$. It  follows that 
  $\{X^n\}$ converges at infinity.

  \smallskip
  \ref{item:nat_geo2}
  Let $\{X^n\}$ and  $\{Y^n\}$ be as in the statement. Then for all $n\in \N_0$ we have 
  $\abs{X^n}= \abs{Y^n} = n$,  and $\abs{X^n-Y^n}\leq 1$ since these  tiles intersect in $x$. Hence
  \begin{equation*}
    (X^n \cdot Y^n) \geq  n -1/2 \to \infty  \text{ as $n\to \infty$},
  \end{equation*}
  and so $\{X^n\} \sim \{Y^n\}$.  
\end{proof}

We can now define a map  from $S$ to $\partial_\infty \Gamma$ as follows. 
If $x\in S$ is arbitrary, then, by what we have seen, we can find a natural geodesic $\{X^n\}$ of
$x$ in $\Ga$. By Lemma~\ref{lem:nat_geo}~\ref{item:nat_geo1} the sequence
$\{X^n\}$ converges at infinity and so determines a point 
$ [\{X^n\}]$ in
$\partial_\infty \Gamma$.  By Lemma~\ref{lem:nat_geo}~\ref{item:nat_geo2} 
the equivalence class $[\{X^n\}]$ only depends on $x$ and not on the choice of the  natural geodesic $\{X^n\}$ of
$x$. In this way, we obtain a well-defined map 
\begin{equation}\label{eq:defnatid}
  \Phi\colon S \to \partial_\infty \Gamma, \quad x\mapsto [\{X^n\}].
\end{equation}

From now on until the end of this section, we suppose that the sequence $\{\X^n\}$ is a quasi-visual approximation of a bounded metric space $(S,d)$; in particular, the associated tile graph $\Gamma=
\Gamma(\{\X^n\})$ is Gromov hyperbolic by Theorems~\ref{thm:_comb_qv} and~\ref{thm:Gromov_hyp}. In this setting, we will call the map $\Phi$ from \eqref{eq:defnatid}
 the \emph{natural identification} between $S$ and $\partial_\infty \Gamma$, since, as we will see in 
Proposition~\ref{prop:nat_homeo}, the map $\Phi$ is a bijection when $(S,d)$ is complete. Later in this section, we will also show  that $\Phi$ is in
fact a quasisymmetry if we equip $ \partial_\infty \Gamma$ with a suitable {\em visual metric} (in the sense of Gromov hyperbolic spaces). 

 Before we turn to the  proof of these statements, we first show a few auxiliary facts relating  metric concepts on $S$ to the Gromov hyperbolic 
 geo\-metry of $\Gamma$. Here it is useful to first promote our given quasi-visual approximation 
 $\{\X^n\}$ to a visual approximation by changing the original metric 
 metric $d$ on $S$. Indeed, by 
 Theorem~\ref{thm:_comb_qv} we can find a metric $\varrho$ on $S$ such that $(S, \varrho)$ is still a bounded metric space and $\{\X^n\}$ is a visual approximation of $(S, \varrho)$ (of the same width) as in Definition~\ref{def:visual}. We denote by  
 $\Lambda>1$ the associated visual parameter as in this  definition.

In the following, metric notions (such as $\diam$) in $S$ are always
understood to be in terms of $\varrho$ unless otherwise indicated. Moreover, as in Section~\ref{sec:tile-graph}, we set $\X\coloneqq \bigsqcup_{n\in \N_0} \X^n$.

\begin{lemma}
  \label{lem:dXYLmXY}
  Let $X,Y\in \X$ be arbitrary. Then
  \begin{equation*}
    \diam(X\cup Y)
    \asymp
    \Lambda^{-m(X,Y)}
    \asymp
    \Lambda^{-(X\cdot Y)}. 
  \end{equation*}
  Here the constants $C(\asymp)$ are independent of the tiles $X$ and $Y$.
\end{lemma}

\begin{proof} We need the following elementary fact.
  
  \begin{claim}
      $\diam(X\cup Y) \asymp \sup\{\varrho(x,y) \colon x\in X,\,
  y\in Y\}$.
  \end{claim}

  Indeed, let us denote the right hand side by $D(X,Y)$. Then clearly
  \[ \diam(X\cup Y)= \max \{D(X,Y), \diam(X), \diam(Y)\}.
  \]
  Note that for $x,x'\in
  X$  and $y\in Y$ we have
  \begin{equation*}
    \varrho(x,x') \leq \varrho(x,y) + \varrho(x',y),  
  \end{equation*}
   which implies that 
   \[
   \diam(X) = \sup\{\varrho(x,x'):
  x,x'\in X\}\leq 2 D(X,Y).
  \] Similarly,  $\diam(Y) \leq 2
  D(X,Y)$, and it follows that
  \begin{equation*}
    D(X,Y) \leq \diam(X\cup Y) \leq 2 D(X,Y), 
  \end{equation*}
  proving the claim with $C(\asymp) = 2$.

  \smallskip
We now obtain the following estimates:
  \begin{align*}
    \diam(X\cup Y) 
    &\asymp
      \sup\{\varrho(x,y) \colon x\in X,\, y\in Y\}
      &&\text{by the Claim}
    \\
    &\asymp
      \sup\{\Lambda^{-m(x,y)} : x\in X,\, y\in Y\}
      &&\text{by Corollary~\ref{cor:dxy_mxy}}
    \\
    &=
      \Lambda^{-m(X,Y)}
      &&\text{by \eqref{eq:defmXY}}
    \\
    &\asymp 
      \Lambda^{-(X\cdot Y)}
      &&\text{by Proposition~\ref{prop:mXY_XY}.}
  \end{align*}
The statement follows.
\end{proof}
 The previous lemma immediately implies that  a sequence $\{X_i\}$ in $\Gamma$ converges at
  infinity (as in \eqref{eq:convinfty}) if and only if 
  \begin{equation}\label{eq:diamXiXj}
  \diam(X_i\cup X_j) \to 0 \text{ as } i,j\to \infty. 
\end{equation}
In this case,
\begin{equation}\label{eq:coninftydiam0}
  \diam(X_i) \to 0 \text{ as } i\to \infty.  
\end{equation}

For the rest of this section, we will make the additional assumption that the bounded metric space $(S,d)$  is complete. Since the identity map $\id_S\: (S, \varrho)\ra (S, d)$ is a quasisymmetry by Lemma~\ref{lem:qv_qs}, the space $(S, \varrho)$ is also bounded and complete by Lemma~\ref{lem:qs_properties}.  

\begin{lemma}
  \label{lem:conv_infty_conv_seq} Let 
  $\{X_i\}$ be  a sequence in $\Gamma$. Then we have the following equivalence:
  \begin{align*}
    &\text{$\{X_i\}$ converges at infinity}\\
%    \intertext{if and only if}
 \Longleftrightarrow \quad 
 &\text{there is a point $x\in S$ such that }
    \diam (\{x\}\cup X_i)\to 0  \text{ as } i\to \infty. 
  \end{align*}
 Moreover, the point $x$ with this property is unique. 
\end{lemma}

The condition  $\diam (\{x\}\cup X_i)\to 0$ as $i\to \infty$ is equivalent to saying 
that we have convergence $X_i\to \{x\}$ as $i\to \infty$ with respect 
to {\em Hausdorff distance} of sets
(see \cite[Section~7.3.1]{BBI01} for the definition of this concept).

\begin{proof}
  ``$\Rightarrow$''
  Assume $\{X_i\}$ converges at infinity.
  We choose a point $x_i \in X_i$ for each $i\in \N_0$. Then  by
  \eqref{eq:diamXiXj} we have 
  \begin{equation*}
    \varrho(x_i, x_j) \leq \diam(X_i \cup X_j) \to 0
  \end{equation*}
  as $i,j\to \infty$. Hence $\{x_i\}$ is a Cauchy sequence. Since
  $(S,\varrho)$ is complete, this sequence has a limit $x\in S$. 
   
  Now \eqref{eq:coninftydiam0} implies that 
  \[ 
    \diam (\{x\}\cup X_i)\le \varrho(x, x_i)+\diam(X_i)\to 0
    \text{ as $i\to \infty$}.
  \] 
%  as $i\to \infty$.

  \smallskip
  ``$\Leftarrow$''
  Let $x$ be as in the statement. Then we have 
  \begin{align*}
    \diam(X_i \cup X_j)
    \leq
    \diam(\{x\} \cup X_i) + \diam(\{x\} \cup X_j) \to 0,
  \end{align*}
  as $i,j\to \infty$. Thus $\{X_i\}$ converges at infinity by
  the discussion after the proof of  Lemma~\ref{lem:dXYLmXY} (see \eqref{eq:diamXiXj}). 

  \smallskip
 To prove the uniqueness statement, let  $x'\in S$ be another point with $\diam(\{x'\} \cup X_i) \to
  0$ as $i\to \infty$. Then 
  \[
    \varrho(x,x')\le \diam (\{x\}\cup X_i)+ \diam (\{x'\}\cup
    X_i) \to 0
    \text{ as $i\to \infty$}.
  \] 
  Hence $\varrho(x,x')=0$ and so $x=x'$.
\end{proof}

\begin{lemma}
  \label{lem:seq_infty_equiv}
  Suppose $\{X_i\}$ and $\{Y_i\}$ are two sequences in $\Gamma$
  converging at infinity. Let $x$ and $y$ be the associated points in $S$ for
   $\{X_i\}$ and $\{Y_i\}$, respectively, as in   Lemma~\ref{lem:conv_infty_conv_seq}. 
  Then the following statements are true:
  \begin{enumerate}
  \smallskip
  \item
    \label{item:diam_dxy}
    $ \displaystyle  
    \varrho(x,y) =
    \lim_{i\to \infty} \diam(X_i \cup Y_i),
    $
     \smallskip
  \item
    \label{item:XiYixiyi}
    $ \displaystyle
    \{X_i\} \sim \{Y_i\}$
   if and only if 
   $x = y$. 
  \end{enumerate}
\end{lemma}

\begin{proof}  \ref{item:diam_dxy} Clearly, for each $i\in \N_0$ we have 
\begin{align*}
\varrho(x,y)&\le  \diam(\{x\}\cup X_i) + \diam(X_i\cup Y_i) + \diam(\{y\}\cup Y_i), 
\\
\diam(X_i\cup Y_i)& \leq \diam(\{x\}\cup X_i) + \varrho(x,y) + \diam(\{y\}\cup Y_i).
\end{align*}
Letting $i\to \infty$ and using    that
  \[
    \lim_{i\to\infty} \diam(\{x\}\cup X_i)= \lim_{i\to\infty} \diam(\{y\}\cup Y_i)= 0
    \] 
by the choice of $x$ and $y$, we see that 
\[ 
\varrho(x,y)\le \liminf_{i\to \infty} \diam(X_i\cup Y_i)\le 
\limsup_{i\to \infty} \diam(X_i\cup Y_i)\le \varrho(x,y).
\]
 The statement follows.

    \smallskip
    \ref{item:XiYixiyi} By definition, $\{X_i\} \sim \{Y_i\}$ if and only if $(X_i\cdot Y_i)
  \to \infty$ as $i\to \infty$. By Lemma~\ref{lem:dXYLmXY} this is  equivalent to \[
  \diam(X_i \cup Y_i) \to 0
  \text { as $i\to \infty$}.
  \]
    By  \ref{item:diam_dxy} this in turn is equivalent to $\varrho(x,y)=0$ and the statement  follows.
\end{proof}

\begin{cor}\label{cor:seqpt} 
Let $x\in S$, $\{Y_i\}$ be a sequence in $\Gamma$ that converges at
infinity, and $\Phi \colon S \to \partial_\infty \Gamma$ be the map as in~\eqref{eq:defnatid}. Then we have the following equivalence: 
\begin{align*}
  [\{Y_i\}] = \Phi(x)
  \;\text{ if and only if }\;
  \diam (\{x\}\cup Y_i)\to 0  \text{ as } i\to \infty.  
\end{align*}
\end{cor}

In other words, if $x\in S$ then
$\{Y_i\}$ is a sequence in the equivalence class $\Phi(x)$ (for
example, a natural geodesic of $x$) if and only $x$ is the
unique point associated with this sequence as in
Lemma~\ref{lem:conv_infty_conv_seq}. 
\begin{proof}
  Suppose first that $\{X^n\}$ is a natural geodesic of $x\in
  X$. Then $\{X^n\}$ converges at infinity by
  Lemma~\ref{lem:nat_geo}~\ref{item:nat_geo1}, and so by
  \eqref{eq:coninftydiam0} we have  $\diam(X^n)\to 0$ as $n\to
  \infty$. Since $x\in X^n$ for each $n\in \N_0$, we conclude
  that
  \begin{equation*}
    \diam (\{x\}\cup X^n)=\diam(X^n)\to 0 \text{ as $n\to \infty$.}
  \end{equation*}
  In particular, $x$ is the unique point in $S$ associated with $\{X^n\}$ as in Lemma~\ref{lem:conv_infty_conv_seq}.

 Now let $\{Y_i\}$ be an arbitrary sequence in $\Gamma$ that converges at infinity and $y\in S$ be the point associated with $\{Y_i\}$ as in Lemma~\ref{lem:conv_infty_conv_seq}, that is, $\diam(\{y\} \cup Y_i)\to 0$ as $i\to\infty$. By Lemma~\ref{lem:seq_infty_equiv}~\ref{item:XiYixiyi} and \eqref{eq:defnatid} we have
  \begin{equation*} x=y
    \;\Leftrightarrow\;
    \{X^n\} \sim \{Y_i\}
    \;\Leftrightarrow\;
    [\{Y_i\}] = [\{X^n\}] = \Phi(x).
  \end{equation*}
  At the same time, by the uniqueness part of Lemma~\ref{lem:conv_infty_conv_seq} we have
  \begin{equation*} x=y
    \;\Leftrightarrow\;
    \lim_{i\to \infty} \diam(\{x\} \cup Y_i)
    =0.
  \end{equation*}
 The statement follows.
\end{proof}

After these preparations, we are now ready to prove that  $\Phi$ is a bijection.

\begin{proposition}
  \label{prop:nat_homeo}
Suppose $(S, d)$ is a complete and bounded metric space. Then the  map  $\Phi \colon S \to \partial_\infty \Gamma$ given by \eqref{eq:defnatid} is a bijection. 
\end{proposition}

\begin{proof} To show injectivity, suppose $x,y\in S$ and $\Phi(x)=\Phi(y)$.
By the definition of $\Phi$, this means that if $\{X^n\}$ and  $\{Y^n\}$ are natural 
geodesics of $x$ and $y$, respectively, then $\{X^n\}\sim \{Y^n\}$. 
At the same time,  Corollary~\ref{cor:seqpt} implies that $x$ and $y$ are the points associated with $\{X^n\}$ and $\{Y^n\}$ as in Lemma~\ref{lem:conv_infty_conv_seq}, respectively. Hence $x=y$ by Lemma~\ref{lem:seq_infty_equiv}~\ref{item:XiYixiyi}, and the injectivity of $\Phi$ follows.

 To show surjectivity, consider an arbitrary point in $\partial_\infty \Ga$ given by the equivalence class $[\{X_i\}]$ of a sequence $\{X_i\}$ in $\Ga$ converging at infinity. Let $x\in S$ be the point  associated with   $\{X_i\}$ as in Lemma~\ref{lem:conv_infty_conv_seq}  and $\{Y^n\}$ be a natural geodesic 
 of $x$. Then Corollary~\ref{cor:seqpt} implies that the point associated with the sequence $\{Y^n\}$ is also given by $x$. By Lemma~\ref{lem:seq_infty_equiv}~\ref{item:XiYixiyi}, we then have $\{X_i\}\sim \{Y^n\}$; so the definition of $\Phi$ 
 gives
 \[ 
 \Phi(x)=[ \{Y^n\}]=[\{X_i\}].
 \]
 The surjectivity of $\Phi$ follows. 
\end{proof}

The last proposition justifies why in our setting we call the map 
$\Phi\: S\ra \partial_\infty \Ga$ the natural identification of 
$S$ and $\partial_\infty \Ga$. We want to study this identification also 
 from a metric point of view. For this we have to  review the concept of  
 a {\em visual metric} on $\partial_\infty \Ga$. This concept is different, but, as we will see in  Proposition~\ref{prop:vis_Gromov_hyp},  closely related to the notion of a visual metric from Definition~\ref{def:vis_metric}.

For two points $X^\infty, Y^\infty\in
\partial_\infty \Gamma$ we define their {\em Gromov product} by setting
\begin{equation}
  \label{eq:def_Gromov_infty}
  (X^\infty\cdot Y^\infty)\coloneqq
  \sup\big\{\liminf_{i\to \infty} (X_i \cdot Y_i) : \{X_i\} \in X^\infty,\, \{Y_i\} \in Y^\infty\big\}
  \in [0,\infty]. 
\end{equation}

Recall that  a point in $\partial_\infty \Gamma$ is an equivalence class of sequences in $\Ga$ converging at infinity. In   \eqref{eq:def_Gromov_infty}, we take the supremum  over all sequences $\{X_i\}$ and $\{Y_i\}$ that are contained in (or, more intuitively, represent) the given points 
$X^\infty$ and $Y^\infty$ in  $\partial_\infty \Gamma$, respectively. 
The Gromov hyperbolicity of $\Ga$ implies that, up to a fixed additive constant, 
the particular choice of these sequences  does not matter here. More precisely (see  \cite[Remarks~III.3.17 (5), p.~433]{BH99}),
there exists a constant $c_0\ge 0$ independent of $X^\infty,\ Y^\infty
\in \partial_\infty \Gamma$ such that for all  $\{X_i\} \in X^\infty$ and 
$\{Y_i\} \in Y^\infty$ we have 
\begin{equation}\label{eq:Grprodinfty}
 (X^\infty\cdot Y^\infty)-c_0\le \liminf_{i\to \infty} (X_i \cdot Y_i)\le 
 (X^\infty\cdot Y^\infty).
\end{equation}
In particular, this implies  that $(X^\infty\cdot Y^\infty)=\infty$ if and only if 
$X^\infty=Y^\infty$.

A metric $d_\infty$ on $\partial_\infty \Gamma$ is called {\em visual} (in the  sense of Gromov hyperbolic spaces) if there is a constant $\Lambda_\infty>1$ such that 
\begin{equation}\label{eq:vismetrGr}
d_\infty(X^\infty,Y^\infty)
  \asymp
  \Lambda_\infty^{-(X^\infty\cdot Y^\infty)}
\end{equation}
for all $X^\infty, Y^\infty\in \partial_\infty \Gamma$ with $C(\asymp)$ independent of  $X^\infty$  and $Y^\infty$. Here we use our previous  convention  that 
$\Lambda_\infty^{-\infty}=0$. We call the constant  $\Lambda_\infty$ the \emph{visual parameter} of $d_\infty$.

The next statement shows   that visual metrics on $S$ and $\partial_\infty
\Gamma$ are in exact correspondence under the natural identification
 $\Phi\: S\ra \partial_\infty \Ga$.

\begin{proposition}
  \label{prop:vis_Gromov_hyp}
  Suppose  $(S,\varrho)$ is  a  complete and bounded  metric space and $\{\X^n\}$ is  a visual
  approximation of $(S,\varrho)$. Then the following statements are true for the associated tile graph $\Gamma=\Gamma(\{\X^n\})$:
  \begin{enumerate}
  \smallskip
  \item
    \label{item:vis_Gromov_hyp1}
    The natural identification  $\Phi\colon (S,\varrho) \to (\partial_\infty \Gamma,
    d_\infty)$ is a snowflake equivalence (see Section~\ref{sec:quasisymmetries}) for each metric $d_\infty$ on $\partial_\infty\Gamma$ that is visual  (in the sense of Gromov hyperbolic spaces).
  \smallskip  
  \item
    \label{item:vis_Gromov_hyp2}
    Setting
    \begin{equation*}
      d_\infty(X^\infty, Y^\infty) \coloneqq \varrho(\Phi^{-1}(X^\infty), \Phi^{-1}(Y^\infty))
    \end{equation*}
    for  $X^\infty, Y^\infty \in \partial_\infty \Gamma$, defines a visual metric on
    $\partial_\infty \Gamma$.
  \smallskip  
  \item
    \label{item:vis_Gromov_hyp3}
    Conversely, let $d_\infty$ be  a visual metric on
    $\partial_\infty \Gamma$ (in the sense of Gromov hyperbolic spaces). 
    Then setting 
    \begin{equation*}
      \widetilde{\varrho}(x,y) \coloneqq d_\infty(\Phi(x),  \Phi(y)),
    \end{equation*}
    for $x,y\in S$, defines a metric on $S$  that is visual for $\{\X^n\}$ (in the
    sense of Definition~\ref{def:vis_metric}). 
  \end{enumerate}
\end{proposition}

  As we will see, the visual parameter $\Lambda_\infty$ of $d_\infty$ in \ref{item:vis_Gromov_hyp2} is the
same as the visual parameter $\Lambda$ of $\varrho$, i.e.,  $\Lambda_\infty= \Lambda$. Similarly, in
\ref{item:vis_Gromov_hyp3}, the width of $\widetilde{\varrho}$
equals the one of $\varrho$, and again we have $\Lambda_\infty= \widetilde{\Lambda}$ for the visual parameters $\widetilde{\Lambda}$ and $\Lambda_\infty$ associated with  
 $\widetilde{\varrho}$ and
$d_\infty$, respectively.

\begin{proof}[Proof of Proposition~\ref{prop:vis_Gromov_hyp}]
  \ref{item:vis_Gromov_hyp1}
  Let $d_\infty$ be a visual metric on $\partial_\infty \Gamma$
  with visual parameter $\Lambda_\infty>1$,  and, as before,  let $\Lambda>1$ be  the
  visual parameter of the visual approximation $\{\X^n\}$ for $(S,\varrho)$ as in Definition~\ref{def:visual}. We set   \[\alpha \coloneqq  \log (\Lambda) /\log (\Lambda_\infty) >0.\] Then $\Lambda = \Lambda_\infty^\alpha$.

Now let $x,y\in S$ be arbitrary and define $X^\infty\coloneqq  \Phi(x)$ 
  and 
  $Y^\infty \coloneqq\Phi(y)$. Choosing $\{X_i\} \in X^\infty$ and $\{Y_i\}\in Y^\infty$,  we have
  \begin{equation}
\diam (\{x\}\cup X_i)\to 0 \text{ and } \diam (\{y\}\cup Y_i) \text{ as } i\to \infty
\end{equation}
 by Corollary~\ref{cor:seqpt}. It follows that 
 \begin{align}\label{eq:metric_estimates}
  \varrho(x,y)
  &=
  \lim_{i\to \infty} \diam(X_i \cup Y_i)
  &&\text{by
    Lemma~\ref{lem:seq_infty_equiv}~\ref{item:diam_dxy}}
    \\
  &\asymp
    \limsup_{i\to \infty} \Lambda^{-(X_i \cdot Y_i)}
  &&\text{by Lemma~\ref{lem:dXYLmXY}}\notag
  \\
  &\asymp
    \Lambda^{-(X^\infty\cdot Y^\infty)} &&\text{by \eqref{eq:Grprodinfty} } \notag
    \\
    &=
   \Lambda_\infty^{-\alpha(X^\infty\cdot Y^\infty)} \notag\\
 &\asymp
    d_\infty(X^\infty,Y^\infty)^\alpha \notag \\
    &= d_\infty(\Phi(x), \Phi(y))^\alpha.\notag
\end{align}
Thus $\Phi\colon (S,\varrho) \to (\partial_\infty \Gamma,
d_\infty)$ is a snowflake equivalence.

\smallskip
\ref{item:vis_Gromov_hyp2}  Since $\Phi\colon S\to \partial_\infty \Gamma$ is a bijection (Proposition~\ref{prop:nat_homeo}), $\Phi^{-1}$ is well defined and so we can define  $d_\infty$ as in the statement. 
Obviously, $d_\infty$ is a metric on  $\partial_\infty \Gamma$,  and for  $X^\infty,Y^\infty\in \partial_\infty \Gamma$ we have 
\[
 d_\infty(X^\infty,Y^\infty) = \varrho(x,y),
 \]
where $x\coloneqq \Phi^{-1}(X^\infty),\, y\coloneqq   \Phi^{-1}(Y^\infty)\in S$. Now, using the estimates in \eqref{eq:metric_estimates}, we see that  
 \[
  d_\infty(X^\infty,Y^\infty) = \varrho(x,y) \asymp
    \Lambda^{-(X^\infty\cdot Y^\infty)}. 
 \]
This shows  that  $d_\infty$ is a visual metric on $\partial_\infty \Gamma$
with visual parameter $\Lambda_\infty = \Lambda$.  

\smallskip
\ref{item:vis_Gromov_hyp3}
Let $d_\infty$ be a visual metric on $\partial_\infty \Gamma$
with visual parameter $\Lambda_\infty>1$
and suppose the visual approximation  $\{\X^n\}$ for $(S,\varrho)$ has width $w$
(and visual parameter $\Lambda>1$) as in 
Definition~\ref{def:visual}. Set 
$\widetilde{\varrho}(x,y) \coloneqq  d_\infty(\Phi(x),\Phi(y))$ for
$x,y\in S$. Then, $\widetilde{\varrho}$ is a metric on $S$, and  by the estimates in  \eqref{eq:metric_estimates}, we have  \begin{equation*}
  \widetilde{\varrho}(x,y) 
  =
  d_\infty(\Phi(x), \Phi(y))
  \asymp
  \varrho(x,y)^{1/\alpha}
\end{equation*}
for $x,y\in S$ with $\alpha =  \log (\Lambda) /\log (\Lambda_\infty)$.
This means that $\id_S \colon (S,\varrho) \to (S,\widetilde{\varrho})$ is a snowflake equivalence. Proposition~\ref{prop:visual_metric_approximation} implies that $\{\X^n\}$ is a visual approximation for $(S,\widetilde{\varrho})$ with visual parameter $\widetilde{\Lambda}= \Lambda^{1/\alpha} = \Lambda_\infty$ and width $\widetilde{w}=w$. The statement follows. \end{proof}

We can now show that the natural identification is a quasisymmetry.

\begin{proposition}
  \label{prop:qv_visual_Gromov_hyp}
  Let $(S,d)$ be a complete and bounded metric space, $\{\X^n\}$ be a
  quasi-visual
  approximation of $(S, d)$, and   $d_\infty$ be  a visual metric on $\partial_\infty\Gamma$, where $\Gamma=\Gamma(\{\X^n))$. Then the natural identification $\Phi\colon (S,d) \ra 
  (\partial_\infty\Gamma, d_\infty)$ is a quasisymmetry.
\end{proposition}
In particular, it follows that  $\Phi$ is homeomorphism.
\begin{proof}   Theorem~\ref{thm:_comb_qv} implies that there is a metric $\varrho$ on $S$ with respect
  to which $\{\X^n\}$ is visual. Then by
  Theorem~\ref{thm:qv-qs} the map 
  \begin{equation*}
    \id_S\colon (S,d) \ra  (S,\varrho)
  \end{equation*}
  is a quasisymmetry and by
  Proposition~\ref{prop:vis_Gromov_hyp}~\ref{item:vis_Gromov_hyp1} the map
  \begin{equation*}
    \Phi\colon (S,\varrho) \to (\partial_\infty \Gamma,d_\infty)
  \end{equation*}
  is a snowflake equivalence, and hence also a quasisymmetry. Therefore, $\Phi\colon (S,d) \to (\partial_\infty \Gamma,d_\infty)$ is  a quasisymmetry as it is  the
  composition of the previous  two maps. The statement follows. 
\end{proof}

\section{From Gromov-hyperbolicity to quasi-visual approximations}
\label{sec:from-grom-hyperb}
In the previous section, we discussed how to naturally identify a given bounded metric space with the boundary at infinity of the (Gromov hyperbolic) tile graph associated with a quasi-visual approximation of the space.  Here we discuss the reverse direction. Namely, we consider a sequence $\{\X^n\}$ of covers of a bounded metric space $(S,d)$, where $\X^0= \{S\}$,
and define the tile 
graph $\Gamma=\Gamma(\{\X^n\})$ as in Section~\ref{sec:tile-graph}. We assume that 
$\Gamma$ is Gromov hyperbolic. In general, it is not true that
this implies that $\{\X^n\}$ is a quasi-visual approximation of $(S,d)$. However, from  
suitable combinatorial neighborhoods of tiles in the tile graph $\Gamma$ one can produce sets  that form  a quasi-visual approximation (under additional natural assumptions on the space $S$). 

More precisely, since $\Gamma$ is assumed to be Gromov hyperbolic, the boundary at infinity $\partial_\infty\Gamma$ of $\Gamma$ is well-defined. We assume it is equipped with a visual metric in the sense of Gromov hyperbolic spaces, and we fix such a metric $d_\infty$ on $\partial_\infty\Gamma$ from now on. As in the last section, we consider the map $\Phi \colon S \to \partial_\infty \Gamma$ as in \eqref{eq:defnatid},
%\[\Phi\colon S\to \partial_\infty\Gamma, \quad x\mapsto [\{X^n\}],\]
which sends each $x\in S$ to a natural geodesic $\{X^n\}$ of $x$ in $\Gamma$ (i.e., $X^n$ satisfy $x\in X^n\in \X^n$ for each $n\in \N_0$). We note that the map $\Phi$ is well-defined by Lemma~\ref{lem:nat_geo} (the presented proof works for arbitrary  sequences $\{\X^n\}$ of covers of $S$).

We will consider two different regularity conditions on the map $\Phi$:
\begin{enumerate}[label=(\arabic*)]
\smallskip
  \item 
  $\Phi\colon (S,d) \to (\partial_\infty \Gamma, d_\infty)$ is
  a quasisymmetry;
  \smallskip
  \item 
  $\Phi\colon (S,d) \to (\partial_\infty \Gamma, d_\infty)$ is a 
  snowflake equivalence.
\end{enumerate}
\smallskip\noindent
Now $(\partial_\infty \Gamma, d_\infty)$  is always complete see \cite[Proposition~6.2]{BS00}). So for $\Phi$ to satisfy either condition, by Lemma~\ref{lem:qs_properties} it is necessary that $(S,d)$ is complete as well.  At the same time, being
uniformly perfect is necessary (and sufficient) for $(S,d)$ to have a quasi-visual approximation by Corollary~\ref{cor:qv-are-uni-perfect}. For these reasons, we will require that the bounded metric space  $(S,d)$ is complete and uniformly perfect.

In general, the original sequence of covers $\{\X^n\}$ is not necessarily a quasi-visual approximation of $(S,d)$ under the above
assumptions (see the discussion in Section~\ref{sec:necess-neighb-clust}).
To remedy this, we will need to substitute the tiles in $\X^n$ by their appropriate
neighborhoods in the tile graph $\Gamma$. 

Let $r\in \N_0$. The
\emph{neighborhood cluster of radius $r$ of a tile $X\in \X$}  is given by
\begin{equation}\label{eq:cluster}
  V_r(X)
  \coloneqq
  \bigcup \{Y\in \X : \abs{X-Y}\leq r\}.
\end{equation}
Here and below, we use the notation from Section~\ref{sec:tile-graph}; that is, $\X$ denotes the formal disjoint union of the sets $\X^n$, and $\abs{X-Y}$ denotes the combinatorial distance between $X,Y\in \X$ in $\Gamma$. Note that, in contrast to the neighborhoods
$U_w(X)$ defined in \eqref{eq:def_UwX^n}, the tiles in $V_r(X)$
are not necessarily of the same level as $X$. 
For each $n\in\N_0$ we now define
\begin{equation}\label{eq: Vn-nbhd}
  \V^n
  =
  \V^n_r(\{\X^n\})
  \coloneqq
  \{V_r(X^n) : X^n\in \X^n\}. 
\end{equation}
We denote by $\V_r$, or simply $\V$ when $r$ is fixed, the formal disjoint union of all sets $\V^n$ as above. It is immediate from the construction that each $\V^n$ is a  cover of $S$ and that $\V^0 = \{S\}$.

\begin{theorem}
  \label{thm:Gromov_implies_qv}
  Let $(S,d)$ be a bounded metric space that is complete and uniformly perfect,  let  $\{\X^n\}$ be a sequence of
  covers of $S$ with $\X^0= \{S\}$ such that the tile graph $\Gamma=\Gamma(\{\X^n\})$ is Gromov hyperbolic, and let  
    $\{\V^n\}$ be the sequence of covers of $S$ defined by \eqref{eq: Vn-nbhd} for some fixed radius $r\in \N_0$. 
   
   Then for every sufficiently large $r$ the following statements are true for the map $\Phi \colon S \to \partial_\infty \Gamma$ as in \eqref{eq:defnatid}: 
  \begin{enumerate}
  \smallskip
  \item \label{item:Gr_hyp_qv1}
    If $\Phi$ is a quasisymmetry, then $\{\V^n\}$ is  a quasi-visual approximation of width $w=1$ 
    for $(S,d)$. 
  \smallskip  
  \item \label{item:Gr_hyp_qv2} 
   If $\Phi$ is a snowflake equivalence, then $\{\V^n\}$ is a visual approximation of width $w=1$ for
    $(S,d)$. 
  \end{enumerate}
\end{theorem}

From now on until the end of this section, we assume the hypotheses of Theorem~\ref{thm:Gromov_implies_qv}.  Since 
$\Gamma=\Gamma(\{\X^n\})$ is Gromov hyperbolic, there exists a constant  $C_\Gamma\ge 0$ such that 
\begin{equation}\label{eq: grom-ineq}
  (X\cdot Y) \geq \min\{(X\cdot Z), (Z\cdot Y)\} - C_\Gamma
\end{equation}
for all $X,Y,Z\in \X$. In addition, let $\Lambda>1$ be the visual parameter of the fixed visual metric $d_\infty$ on $\partial_\infty
\Gamma$. 
We define 
\begin{equation*}
  \varrho(x,y)
  \coloneqq
  d_\infty(\Phi(x), \Phi(y))
\end{equation*}
for $x,y\in S$, which is clearly a metric on $S$. Then
\begin{equation*}
  \varrho(x,y)
  \asymp
  \Lambda^{-(X^\infty \cdot Y^\infty)},
\end{equation*}
where $X^\infty\coloneqq  \Phi(x)\in \partial_\infty
\Gamma$ and  $Y^\infty\coloneqq  \Phi(y) \in \partial_\infty
\Gamma$. In particular, $(S,\varrho)$ is bounded. Furthermore, since $\Phi\colon(S,\varrho)
  \to (\partial_\infty \Gamma, d_\infty)$ is an isometry and $\Phi\colon (S,d) \to (\partial_\infty \Gamma, d_\infty)$ is assumed to be (at least) a quasisymmetry, by Lemma~\ref{lem:qs_properties} the space $(S,\varrho)$ is complete and uniformly perfect. 

Now the following statement is true.
 
\begin{proposition}
  \label{prop:Gromov_visual}
  For each sufficiently large radius $r\in \N_0$, the sequence $\{\V^n\}$ forms a visual approximation of width $w=1$ for $(S,\varrho)$.
\end{proposition} 
Note that $S$ is equipped with the metric $\varrho$ here and not the original metric $d$. 

Before we turn to the proof of this proposition, we first point out how it implies 
 Theorem~\ref{thm:Gromov_implies_qv}.

\begin{proof}[Proof of Theorem~\ref{thm:Gromov_implies_qv}]
  Let the radius $r\in \N_0$ be sufficiently large so that $\{\V^n\}$ is a 
  visual approximation of width $w=1$ for $(S,\varrho)$ as in
  Proposition~\ref{prop:Gromov_visual}. 
  Recall also that $\Phi\colon(S,\varrho)
  \to (\partial_\infty \Gamma, d_\infty)$ is an isometry by the definition of the metric $\varrho$. 

  \smallskip
   Assume first that $\Phi\colon (S,d) \to (\partial_\infty
  \Gamma, d_\infty)$ is a quasisymmetry. Then
  \begin{align} \label{eq: id-as-comp}
    \id_S\colon 
    (S,d)
    \xrightarrow{\Phi}
    (\partial_\infty \Gamma, d_\infty)
    \xrightarrow{\Phi^{-1}}
    (S,\varrho)
  \end{align}
  is a quasisymmetry as well. Since $\{\V^n\}$ is a quasi-visual approximation of width $w=1$ for $(S,\varrho)$ by Lemma~\ref{lem:visual-is-also-quasi}, it follows from Theorem~\ref{thm:qv-qs} that $\{\V^n\}$ is also a quasi-visual approximation of width $w=1$ for $(S,d)$. This shows \ref{item:Gr_hyp_qv1}.

  \smallskip
  Now assume that that $\Phi\colon (S,d) \to (\partial_\infty
  \Gamma, d_\infty)$ is a snowflake equivalence. Then the map in \eqref{eq: id-as-comp} %$\id_S=\Phi^{-1} \circ \Phi \colon (S,d) \to (S,\varrho)$
  is also a snowflake equivalence. The desired statement \ref{item:Gr_hyp_qv2} now follows directly from
   Proposition~\ref{prop:visual_metric_approximation}, and the proof of the proposition
is complete.
\end{proof}

It remains to prove Proposition~\ref{prop:Gromov_visual}. To this end, we will verify the two conditions
in Definition~\ref{def:visual} in the ensuing two lemmas, where all metric notions in $S$ are with respect to the metric $\varrho$.

\begin{lemma}
  \label{lem:Ud_Ln}
  For each sufficiently large (fixed)  $r\in \N_0$, we have
  \begin{equation*}
    \diam(V^n)
    \asymp
    \Lambda^{-n}
  \end{equation*}
  for all $n\in \N_0$ and $V^n\in \V^n= \V^n_r(\{\X^n\})$. 
\end{lemma}

This means $\{\V^n\}$ satisfies condition~\ref{item:visual1} in
Definition~\ref{def:visual} for $(S,\varrho)$.

\begin{proof}
  Let $n\in \N_0$ and $V^n=V_r(X^n)\in \V^n$ be arbitrary, where $X^n\in \X^n$ and the radius $r\in \N_0$ is to be fixed later.

  We choose $x\in X^n$. Since $S$ is uniformly perfect, there is $y\in S$
  with
  \begin{equation*}
    \varrho(x,y)
    \asymp
    \Lambda^{-n},
  \end{equation*}
  where $C(\asymp)$ depends on the uniform perfectness constant of
  $(S,\varrho)$ and its (finite) diameter. On the other hand, we have 
  \begin{equation*}
    \varrho(x,y)
    \asymp
    \Lambda^{-(X^\infty\cdot Y^\infty)},
  \end{equation*}
  where $X^\infty\coloneqq \Phi(x) \in \partial_\infty\Gamma$ and  $Y^\infty\coloneqq \Phi(y)\in \partial_\infty\Gamma$. 
  Let $\{X^i\}$ and $\{Y^i\}$ be natural geodesics   for $x$ and
  $y$ in $\Gamma$, respectively. By \eqref{eq:Grprodinfty} we know that
  \begin{equation*}
    \Lambda^{-(X^\infty\cdot Y^\infty)}
    \asymp
    \limsup_{i\to \infty} \Lambda^{-(X^i\cdot Y^i)}.
  \end{equation*} It follows that 
  \begin{equation}\label{eq: gr-prod-bounded}
    (X^i\cdot Y^i)
    \ge
    n -C_1
  \end{equation}
  for all sufficiently large $i\in \N_0$ and a uniform constant $C_1\ge 0$. 

  Without loss of generality, we may assume that $X^n$ is the $n$-tile in our chosen natural geodesic $\{X^i\}$ of $x$. We want to show that $|X^n-Y^n|$ is uniformly bounded from above, where $Y^n$ is the $n$-tile in the chosen natural geodesic $\{Y^i\}$ of $y$. 
  
  Indeed, using \eqref{eq: grom-ineq}, for each $k\in \N_0$ we see that
  \begin{align*}
    (X^n \cdot Y^n)
    \geq
    \min\{(X^n \cdot X^{n+k}),\, (X^{n+k}\cdot Y^{n+k}),\,
    (Y^{n+k}\cdot Y^n)\} - 2 C_\Gamma. 
  \end{align*}
  Note that $(X^n\cdot X^{n+k}) = (Y^n\cdot Y^{n+k}) =n$. Substituting this into the previous inequality and using \eqref{eq: gr-prod-bounded} we obtain  
  \begin{align*}
    (X^n\cdot Y^n)
      \geq
      n- C
  \end{align*}
  for the uniform constant $C\coloneqq C_1+2C_\Gamma>0$. Since
  \begin{align*}
    (X^n\cdot Y^n)
    =
    \frac{1}{2}(\abs{X^n} + \abs{Y^n} - \abs{X^n-Y^n})
    =
    n- \frac{1}{2} \abs{X^n-Y^n},
  \end{align*}
  we conclude that 
  \begin{equation*}
    \abs{X^n-Y^n}\leq 2C. 
  \end{equation*}

  It follows that whenever the radius $r$ satisfies $r\geq\ceil{2C}$, we have $y\in Y^n\subset
  V^n=V_r(X^n)$, and thus
  \begin{equation*}
    \diam(V^n)\geq \varrho(x,y) \gtrsim \Lambda^{-n}. 
  \end{equation*}
  We have shown the desired lower bound on $\diam(V^n)$. 
  
  \smallskip
  To show the respective upper bound, let $z\in V^n=V_r(X^n)$ be arbitrary. Then
  there is $Z\in \X$ with $z\in Z$ that satisfies 
  $\abs{X^n- Z}
  \leq r$ and so $n-r\le \abs{Z} \le  n+r$. 
  
  Let
  $\{Z^i\}$ be a natural geodesic of $z$ and set $Z^\infty:=\Phi(z)=[\{Z^i\}]$. We may assume that the tile $Z$ is part of this natural geodesic or, equivalently, that $Z^i=Z$ for $i=|Z|$. Then for each $k\in \N_0$ we have 
  \[
  \abs{Z - Z^{n+k}}= \big||Z|-|Z^{n+k}|\big|\le \big||Z|-n\big|+k\le r+k,
  \]
  and so 
  \begin{equation*}
    \abs{X^{n+k}- Z^{n+k}}
    \leq
    \abs{X^{n+k} - X^n} + \abs{X^n - Z} + \abs{Z - Z^{n+k}}
    \leq
    k + r + r+k.
  \end{equation*}
  It follows that 
  \begin{align*}
    (X^{n+k} \cdot Z^{n+k})
    &=
    \frac{1}{2}(\abs{X^{n+k}} + \abs{Z^{n+k}} - \abs{X^{n+k} -
      Z^{n+k}})
    \\
    &\geq
    \frac{1}{2}( n+ k + n+k - 2k -2r) = n -r,   
  \end{align*}
 which implies
  \begin{align*}
    \varrho(x,z)
        \asymp
  \Lambda^{-(X^\infty \cdot Z^\infty)}
    \asymp
    \limsup_{k\to \infty}
    \Lambda^{-(X^{n+k} \cdot Z^{n+k})}
    \lesssim
    \Lambda^{-n}.
  \end{align*}
  Since $z\in V^n$ was arbitrary, we conclude that $\diam(V^n) \lesssim \Lambda^{-n}$,
  finishing the proof of the lemma.
\end{proof}

\begin{lemma}
  \label{lem:Gromov_visualii}
  For each  sufficiently large (fixed) $r\in \N_0$ the implication
  \begin{equation*}
    U_1(V) \cap U_1(V') = \emptyset
    \, \Rightarrow \, 
    \dist(V, V') \gtrsim \Lambda^{-n}.
  \end{equation*}
  is true for all $n\in \N_0$ and $V, V'\in\V^n =
  \V^n_r(\{\X^n\})$. 
\end{lemma}

This means $\{\V^n\}$ satisfies condition~\ref{item:visual2}
in Definition~\ref{def:visual} with width $w=1$. 
 The proof will show that  $r > 2C_\Gamma$ is sufficient for the implication to be true. 

\begin{proof}
  Let $n\in \N_0$ and $V,V'\in \V^n= \V^n_r(\{\X^n\})$ be arbitrary, where the radius $r\in \N_0$ is to be fixed later.

  Then $V = V_r(X^n)$ for some
  $X^n\in \X^n$. Now let $x\in V$ be arbitrary. Then there is $\widetilde{X}\in \X$ with $x\in \widetilde{X}$ and $\abs{X^n-\widetilde{X}}\leq r$, and so $\widetilde{X}$ satisfies 
  \begin{equation*}
     n-r \leq \abs{\widetilde{X}} \leq n+ r.
  \end{equation*}
  Let $\{\widetilde{X}^i\}$ be a natural geodesic of $x$. We
  may assume it includes $\widetilde{X}$, meaning that
  $\widetilde{X}^\ell = \widetilde{X}$, where $\ell \coloneqq
  \abs{\widetilde{X}}$ is the level of $\widetilde{X}$. 
  
 Now  consider $\widetilde{V}\coloneqq
  V_r(\widetilde{X}^n)$. Since $\abs{\widetilde{X}^n -
    \widetilde{X}} = \abs{n-\ell}\leq r$, we have
  $\widetilde{X}\subset \widetilde{V}$, and so $x\in V\cap \widetilde{V} \neq
  \emptyset$. It follows that
  \begin{equation*}
    \widetilde{V}\subset U_1(V).
  \end{equation*}

  Similarly, let $y\in V'$ be arbitrary. By a similar reasoning as for $x$, 
   there is a natural
  geodesic $\{\widetilde{Y}^i\}$ of $y$ such that
  $\widetilde{V}' \coloneqq V_r(\widetilde{Y}^n)$ satisfies
  \begin{equation*}
    \widetilde{V}'\subset U_1(V').
  \end{equation*}

  After these preliminaries, let us now assume that the given sets $V,V'\in \V^n$ satisfy the condition $U_1(V)
  \cap U_1(V') = \emptyset$. The above considerations then imply that $\widetilde{V} \cap \widetilde{V}' = \emptyset$, and so $
    \abs{\widetilde{X}^n - \widetilde{Y}^n} \geq 2r$.
  It follows that
  \begin{equation*}
    (\widetilde{X}^n \cdot \widetilde{Y}^n)
    = \frac{1}{2} (\abs{\widetilde{X}^n} + \abs{\widetilde{Y}^n}
    - \abs{\widetilde{X}^n - \widetilde{Y}^n})
    \leq
    \frac{1}{2}( n+ n -2r) = n - r.
  \end{equation*}
  Using \eqref{eq: grom-ineq}, we
 conclude that  for all $k \in \N_0$,
  \begin{align*}
    n-r
    &\geq
      (\widetilde{X}^n \cdot \widetilde{Y}^n)
      \\
    &\geq
    \min\{(\widetilde{X}^n \cdot \widetilde{X}^{n+k}),
    (\widetilde{X}^{n+k}\cdot \widetilde{Y}^{n+k}),
      (\widetilde{Y}^{n+k}\cdot \widetilde{Y}^n)\} - 2C_\Gamma
    \\
    &=
      \min\{n,
      (\widetilde{X}^{n+k}\cdot \widetilde{Y}^{n+k}), n\}
      -2C_\Gamma, 
  \end{align*}
  and so
  \begin{equation}\label{eq:nrCGa}
  n - r +2 C_\Gamma \geq  \min\{n,
      (\widetilde{X}^{n+k}\cdot \widetilde{Y}^{n+k})\}. 
  \end{equation}
  
  Now suppose that $r> 2C_\Gamma$. Then for the  left hand side in
  \eqref{eq:nrCGa} we have  $n-r+2C_\Gamma<
  n$. Therefore, the minimum on the right hand side is attained by 
  $(\widetilde{X}^{n+k}\cdot \widetilde{Y}^{n+k})$, which implies that 
  \begin{equation*}
    n
    \geq
    (\widetilde{X}^{n+k}\cdot \widetilde{Y}^{n+k}). 
  \end{equation*}
  By \eqref{eq:Grprodinfty}, we then have
  \[
    \varrho(x,y) 
    \asymp \Lambda^{-(X^\infty \cdot Y^\infty)} \asymp 
    \limsup_{k\to \infty}
    \Lambda^{-(\widetilde{X}^{n+k}\cdot
    \widetilde{Y}^{n+k})}
    \geq
    \Lambda^{-n}, 
  \]
  where $X^\infty\coloneqq \Phi(x)=[\{\widetilde{X}^i\}] \in \partial_\infty\Gamma$ and $Y^\infty\coloneqq \Phi(y)=[\{\widetilde{Y}^i\}]\in \partial_\infty\Gamma$. 
  
  Since the points $x\in V$ and $y\in V'$
  were arbitrary, we conclude that $\dist(V, V') \gtrsim
  \Lambda^{-n}$, which finishes the proof of the lemma.
\end{proof}

Proposition~\ref{prop:Gromov_visual} is now an immediate consequence of Lemmas~\ref{lem:Ud_Ln} and~\ref{lem:Gromov_visualii}.

\subsection{Comparing the tile graphs for \texorpdfstring{$\{\X^n\}$ and $\{\V^n\}$}{\{Xⁿ\} and \{Vⁿ\}}}
\label{sec:comparing-GXGV}  
Our starting point in this section was a sequence $\{\X^n\}$ of covers of a bounded metric space $(S,d)$ from
which we constructed a new sequence of covers $\{\V^n\}$. The
question arises how the associated tile graphs $\Gamma_{\X}\coloneqq \Gamma(\{\X^n\})$ and $\Gamma_{\V}\coloneqq
\Gamma(\{\V^n\})$ relate to each other (as metric spaces). We denote the corresponding combinatorial distances by $\abs{X-Y}_{\X}$
and $\abs{V-V'}_{\V}$, respectively (for $X,Y\in \X$ and
$V,V'\in \V$). Here $r\in \N_0$ is fixed, and we write $V(X)=
V_r(X)$ in the following to simplify notation. 

We have a natural surjective map 
\begin{equation}\label{eq: graph-map}
\Psi\colon \Gamma_\X \to \Gamma_\V, \quad X\in \X\mapsto V(X)\in \V.
\end{equation} We claim that this map is in fact a {\em quasi-isometry} (actually a \emph{rough similarity}), which is shown in the next lemma (for a discussion of these concepts, see \cite[Section~2]{BS00}).

\begin{lemma}
  \label{lem:GXGV_qi}
  The surjective map $\Psi\colon \Gamma_\X \to \Gamma_\V$ given in \eqref{eq: graph-map} satisfies the inequality
  \begin{equation*}
    \frac{1}{2r+1}\abs{X-Y}_{\X} 
    \leq
    \abs{V(X) - V(Y)}_{\V}
    \leq
    \frac{1}{2r+1} \abs{X-Y}_{\X} +1
  \end{equation*}
  for all $X,Y\in \X$.
\end{lemma}

\begin{proof} Let $X,Y\in \X$ be arbitrary. 
  We prove the inequality on the left first. To this end, let $V_0 =
  V(X),\,  V_1,\,  \dots,\, V_N= V(Y)$ be a path of length $N\coloneqq
   \abs{V(X) - V(Y)}_{\V}
  \geq 0$ 
  in $\Gamma_{\V}$. By construction, for each $i$ there is $X_i\in \X$ such that $V_i= V(X_i)\in
  \V$, where $X_0=X$ and $X_N=Y$. Furthermore, we have $\abs{X_{i-1}- X_{i}}_{\X}\leq 2r+1$ for each $i=1,\dots,N$, since $V_{i-1}\cap V_{i}\neq \emptyset$. It follows that 
  \begin{equation*}
    \abs{X-Y}_{\X}
    \leq
    \sum_{i=1}^{N} \abs{X_{i-1} -X_{i}}_{\X}
    =
    (2r+1)N.
  \end{equation*}
  The desired inequality follows.

  \smallskip
  We will now show the inequality on the right. To this end, let $X_0 =
  X,\,  X_1,\,  \dots,\, X_N =Y$ be a path of length
  $N\coloneqq |X-Y|_{\X}\geq 0$ in $\Gamma_{\X}$. We  set $X_i\coloneqq Y$ for all $i>N$. Note  that
  \begin{align*}
    X_r \subset V(X_0),\ X_{r+1} \subset V(X_{2r+1}),  \text{ and }  X_r \cap X_{r+1} \neq \emptyset, 
  \end{align*}
  which means that
  $V(X_0) \cap V(X_{2r+1}) \neq \emptyset$. Similarly, when setting $V_k:=V(X_{(2r+1)k})$, we obtain $V_k \cap V_{k+1} \neq \emptyset$ for all $k\in \N_0$. Since, by construction, $V_k=V(Y)$ for $k:=\ceil*{\frac{N}{2r+1}}$, we have constructed a path $V_0=V(X)$, $V_1$, $\dots$, $V_k=V(Y)$ of length $k$ in $\Gamma_\V$. This implies
  that 
  \begin{align*}
    \abs{V(X)- V(Y)}_{\V} \leq k = \ceil*{\frac{N}{2r+1}} 
    \le \frac{1}{2r+1}\abs{X-Y}_{\X} +1. 
  \end{align*}
  This establishes the desired inequality and concludes the proof. 
\end{proof}

We obtain the following consequence.

\begin{cor}
The tile graph $\Gamma_{\X}$ is Gromov hyperbolic if and only if $\Gamma_{\V}$ is Gromov hyperbolic. Moreover, in this case the map $\Psi\colon \Gamma_\X\to \Gamma_\V$, given by $X\mapsto V(X)$, induces a bijection  $\Psi_\infty \colon
\partial_{\infty} \Gamma_{\X}\to \partial_\infty \Gamma_{\V}$, given by $[\{X_i\}]\mapsto [\{V(X_i)\}]$. Furthermore, if $d_{\infty,\X}$ and $d_{\infty,\V}$ denote visual metrics on $\partial_\infty \Gamma_{\X}$  and $\partial_\infty \Gamma_{\V}$ (in the sense of Gromov hyperbolic spaces), respectively, then the map
\[ \Psi_\infty \colon
(\partial_{\infty}\Gamma_{\X},d_{\infty,\X}) \to (\partial_\infty \Gamma_{\V}, d_{\infty,\V})\]
is a snowflake equivalence. 
\end{cor}

\begin{proof}  The first claim follows from the fact that quasi-isometries between geodesic metric spaces preserve Gromov hyperbolicity
(see \cite[Th\'eor\`eme~5.12, p.~88]{GH}). Formally, to apply this result we have to add edges (of length $1$) between adjacent vertices in both spaces $\Gamma_\X$ and $\Gamma_\V$ to make them geodesic. 

The two remaining claims follow from \cite[Proposition~6.3 and Theorem~6.5]{BS00}.
\end{proof}

\subsection{An example: necessity of neighborhood clusters in Theorem~\ref{thm:Gromov_implies_qv}}
\label{sec:necess-neighb-clust}
In this subsection we present an example of a sequence $\{\X^n\}$ of covers of
a bounded metric space $(S,d)$, where $\X^0=\{S\}$, with the following properties:
\begin{enumerate}[label=(\arabic*)]
\smallskip
\item
  The associated tile graph $\Gamma=\Gamma(\{\X^n\})$ is
  Gromov hyperbolic.
  \smallskip
\item
  The map $\Phi\colon (S,d) \to (\partial_\infty
  \Gamma, d_\infty)$ as defined in \eqref{eq:defnatid} is a quasisymmetry for each  visual
  metric $d_\infty$ on $\partial_\infty \Gamma$ (in the sense of Gromov hyperbolic spaces).
  \smallskip
\item
  $\{\X^n\}$ is not a quasi-visual approximation of $(S,d)$. 
\end{enumerate}
\smallskip\noindent
This shows that in Theorem~\ref{thm:Gromov_implies_qv} 
it is
indeed necessary to take
neighborhood clusters as in \eqref{eq:cluster} to obtain  a
(quasi-)visual approximation (of width $w=1$) for $(S,d)$.

To discuss an example with these properties, let  $S=[0,1]$ be the closed unit interval. We  equip it with the standard
  Euclidean metric given by $d(x,y)\coloneqq \abs{x-y}$ for 
  $x,y\in S$. 
  We denote by $\I^n$ the set of all closed  dyadic 
  intervals of level $n\in \N_0$, meaning all sets of the form
  $[(j-1)/2^n, j/2^n]$ for $j=1,\dots, 2^n$. It is immediate
  to see  that $\{\I^n\}$ is a visual approximation (of width $w=0$)  for $(S,d)$.

  We now define a new sequence $\{\X^n\}$ of covers of $S$ by inserting a sufficiently fine cover between every two consecutive elements of $\{\I^n\}$.  More precisely, we define
  \begin{equation*}
    \X^{2k} \coloneqq \I^k \quad \text{ and }
    \quad
    \X^{2k+1} \coloneqq \I^{2k},
  \end{equation*}
  for $k\in \N_0$. Note that $\X^0=\I^0=\{S\}$.

  \smallskip
  {\it Claim~1: }
  $\{\X^n\}$ is not a quasi-visual approximation of $S$.
  \smallskip

By construction, each set $Y \in
  \X^{2k+1} = \I^{2k}$ for $k\in \N_0$  is contained in a (single) set $X\in \X^{2k}= \I^{k}$. Since  
  \[
  \diam(X)/\diam(Y)=2^{-k}/2^{-2k}=2^k,
  \]
  we conclude that $\{\X^n\}$ does not satisfy
  condition~\ref{item:qv_approx3} in
  Definition~\ref{def:qv_approx}, and hence is not a quasi-visual
  approximation of $S$.

  \smallskip

As in Section~\ref{sec:tile-graph}, we consider the tile graphs $\Gamma_{\I}\coloneq\Gamma(\{\I^n\})$ and
  $\Gamma_{\X}\coloneq\Gamma(\{\X^n\})$ with the vertex sets $\I \coloneq \bigsqcup_{n\in
    \N_0} \I^n$ and $\X\coloneq\bigsqcup_{n\in \N_0}\X^n$, respectively. For $I,I'\in \I$ we denote by $\abs{I - I'}_{\I}$ the combinatorial distance in $\Gamma_\I$,  and for $X,X'\in \X$ by 
  $\abs{X-X'}_{\X}$ the combinatorial distance in $\Gamma_\X$. It is now a rather standard fact (which also follows from Theorem~\ref{thm:Gromov_hyp} in conjunction
  with Theorem~\ref{thm:_comb_qv}) that the tile graph 
  $\Gamma_{\I}$ is Gromov hyperbolic. 
    
  \smallskip
  {\it Claim~2: }
  $\Gamma_{\X}$ is Gromov hyperbolic.
  \smallskip

To see this, note that since $\I^k=\X^{2k}$ for each $k\in\N_0$, we have a canonical embedding \[f\colon \Gamma_\I\to \Gamma_\X, \quad I\in \I^k\mapsto I\in\X^{2k}.\] It is immediate that the image $f(\I)=\bigsqcup_{k\in \N_0} \X^{2k}$ is \emph{cobounded} in $\X$: namely, for each $X\in \X$ there exists $I=\iota(X)\in \I$ such that 
  \begin{equation}
    \label{eq:GI_cobounded}
    \abs{f(I) - X}_{\X} \leq 1.
  \end{equation}
  Indeed, if $X\in \X^{2k}=\I^k$ for some $k\in \N_0$, then we set $\iota(X)\coloneq X\in \I^k$. Otherwise, $X\in \X^{2k+1}$ for some $k\in \N_0$ and $X$ is contained in a (single) set $\iota(X)\in \I^k=\X^{2k}$. Inequality~\eqref{eq:GI_cobounded} then holds with the chosen $I=\iota(X)$ in both cases.

Next, we claim that $f$ satisfies
  \begin{equation}
    \label{eq:psi_qi}
    \abs{I - I'}_{\I}
    \leq
    \abs{f(I)- f(I')}_{\X}
    \leq
    2 \abs{I-I'}_{\I} 
  \end{equation}
  for all $I,I'\in \I$. To show the lower bound, let 
  \[X_0=f(I),\, X_1,\,  \dots,\,  X_N=f(I')\] be a path of length $N\coloneq \abs{f(I)- f(I')}_{\X}\geq 0$ in $\Gamma_\X$. It is easy to see that then \[ I=\iota(X_0),\, \iota(X_1),\,  \dots,\,  I'=\iota(X_N)\]
  is a path in $\Gamma_\I$, which implies the desired inequality. 
  
  To show the upper bound, note that we have the implication 
  \[|I-I'|_\I=1 \quad \Rightarrow \quad |f(I)-f(I')|_\X\leq 2.\]
  Indeed, $|f(I)-f(I')|_\X=1$ when $I,I'$ are of the same level, and $|f(I)-f(I')|_\X=2$ otherwise. The desired estimate now easily follows from the triangle inequality.

  Since the embedding $f\colon \Gamma_\I\to \Gamma_\X$ satisfies \eqref{eq:GI_cobounded} and 
  \eqref{eq:psi_qi}, it is a quasi-isometry. As  a  quasi-isometry between  geodesic metric spaces preserves Gromov 
  hyperbolicity (see \cite[Th\'{e}or\`{e}me~5.12, p.\ 88]{GH}), we again 
  conclude that the tile graph $\Gamma_{\X}$ is also Gromov hyperbolic.

\smallskip

  Now let $\Phi_{\X}\colon S \to
  \partial_\infty \Gamma_{\X}$ be the map as in \eqref{eq:defnatid}. Recall that Lemma~\ref{lem:nat_geo} uses only the Gromov
  hyperbolicty of $\Gamma_{\X}$ and so  $\Phi_{\X}$ is
  well-defined.

  \smallskip
  
  {\it Claim~3:} The map $\Phi_{\X}\colon (S, d) \to
  (\partial_\infty \Gamma_{\X}, d_{\infty,\X})$ is a quasisymmetry for each  visual
  metric~$d_{\infty,\X}$ on $\partial_\infty \Gamma_
  X$ (in the sense of Gromov hyperbolic spaces).
  
  \smallskip
 To see this, first note that the quasi-isometry $f\colon \Gamma_\I\to \Gamma_\X$ from the proof of Claim~2 induces  a bijection
  \[f_\infty \colon \partial_\infty \Gamma_{\I} \to
  \partial_\infty \Gamma_{\X}, \quad [\{I_i\}]\mapsto [\{f(I_i)\}]\]
  (see \cite[Proposition~6.3]{BS00}). Moreover, if $d_{\infty, \I}$ and $d_{\infty, \X}$ denote visual metrics on $\partial_\infty \Gamma_{\I}$ and $\partial_\infty \Gamma_{\X}$, respectively, then the map $f_\infty \colon (\partial_\infty \Gamma_{\I}, d_{\infty, \I}) \to
  (\partial_\infty \Gamma_{\X}, d_{\infty, \X})$ is a quasisymmetry (see
  \cite[Theorem~6.5]{BS00}).

  Now let $x\in S$ be arbitrary and $\{X^n\}$ be a natural geodesic of $x$ 
  in $\Gamma_{\X}$, that is, $x\in X^n\in
  \X^n$ for all $n\in \N_0$. Then $\{X^n\}$ converges at infinity (by  Lemma~\ref{lem:nat_geo}), and thus the subsequence $\{X^{2n}\}$ converges at infinity as well and is equivalent to $\{X^n\}$. Moreover,  
  the sequence $\{I^n\}$ with 
  $I^n\coloneq X^{2n}\in \I^{n}=\X^{2n}$ for $n\in \N_0$ is a natural geodesic of
  $x$ in $\Gamma_{\I}$. It follows that 
  \[
    \Phi_{\X}(x)
    =
    [\{X^n\}]
    =
    [\{X^{2n}\}]
    =
    [\{f(I^n)\}]
    =
    f_\infty([\{I^n\}])
    =
    (f_\infty \circ \Phi_{\I})(x), 
  \]
  where $\Phi_{\I} \colon S\to
  \partial_\infty \Gamma_{\I}$ denotes the natural identification between
  $S$ and $\partial_\infty \Gamma_{\I}$. Since $x$ was arbitrary, we conclude that $\Phi_{\X} =
  f_\infty \circ \Phi_{\I}$. The statement now follows from Proposition~\ref{prop:qv_visual_Gromov_hyp},

\smallskip
The three claims above show that the sequence $\{\X^n\}$ has all the
desired properties as listed at the beginning of this subsection. In particular, this example fails to satisfy  condition~\ref{item:qv_approx3} in
Definition~\ref{def:qv_approx}. If we modify  this construction, it is
possible to violate conditions~\ref{item:qv_approx1} and~\ref{item:qv_approx2} in Definition~\ref{def:qv_approx} as well by inserting suitable layers $\X^{2k+1}$.  

In the example presented, it is not
hard to see that $r=1$ suffices in
Theorem~\ref{thm:Gromov_implies_qv}. However, by inserting more
than one layer (between each $\I^n$ and $\I^{n+1}$), we may
arrange that $r$ is required to be arbitrarily large. Finally, a
similar construction of inserting layers of small sets works in
greater generality (when we start from some quasi-visual approximation of a given bounded metric space).

\section{Dynamical quasi-visual approximations of Julia sets}\label{sec:semihyp}

In this section we consider quasi-visual approximations for  Julia sets of rational maps on the Riemann sphere $\CDach\coloneqq \C\cup\{\infty\}$.  These   sets are always uniformly perfect (see \cite{Hi} and \cite{MdR}), and hence have a quasi-visual approximation by Corollary~\ref{cor:qv-are-uni-perfect}.  However, in general such a quasi-visual approximation is not related to the underlying dynamics. So the question arises 
how to produce quasi-visual approximations of Julia sets  in a  natural dynamical way. 
We will see that  a certain class of  rational maps is relevant here,  namely  {\em semi-hyperbolic rational maps}. More specifically,  we  will show that the Julia set of a rational map $g$ admits a quasi-visual approximation that is {\em dynamical} as defined in Section~\ref{subsec:dynqvisual} 
if and only if it is semi-hyperbolic (see Corollary~\ref{cor:charsemi} and Theorems~\ref{thm:semi_hyp_admits_approx} and~\ref{thm:Adconverse} for the precise formulation of the two 
implications involved here). In order to  discuss this, we first have to recall some general background. 

\subsection{Rational maps and dynamically generated covers of Julia sets}

In the following, the Riemann sphere $\CDach$ will be equipped with the {\em spherical metric}. This is a length metric on 
$\CDach$ whose restriction to $\C$ is given by the length element 
\[
ds=\frac{2|dz|}{1+|z|^2},
\]
where $z$ is a variable in $\C$.  We use the subscript $\sigma$ to indicate that metric concepts are with respect to this metric. 

A {\em region} $U\sub \CDach$ is an open and connected set. 
A holomorphic map $f\: U\ra V$ between two regions $U,V\sub \CDach$ 
 is called {\em proper} if $f^{-1}(K)\sub U$ is  compact whenever $K\sub V$ is a compact set. If $f$ is proper, then   every value 
in $V$ is attained the same (finite) number of times, counting
multiplicities (see \cite[Theorem~4.24]{Fo}). 
We call this number the {\em degree of~$f$ (in $U$)}, denoted by
$\deg(f)\in \N$. So if we denote by $\deg(f,z)$ the {\em local degree 
of $f$ at $z\in U$}, then we have 
\[
\sum_{z\in f^{-1}(w)} \deg(f,z)=\deg(f)
\]
for each $w\in V$. In particular, this applies  when $f$ is a non-constant rational 
map and  $U=V=\CDach$, in which case  $\deg(f)$ agrees with the  {\em topological degree} of $f$.  

 For the remainder of this subsection   $g\:\CDach \ra \CDach$ will be  a rational map with $\deg(g)\ge 2$.  As in Section~\ref{subsec:dynqvisual} we denote by  $g^n$ for $n\in \N_0$ the {$n$-th iterate of $g$},  where $g^0\coloneqq \id_{\CDach}$. We write  $\Jul(g)$ for  the {\em Julia set} of $g$ (see \cite[Definition~4.2]{Mi06} for the definition, and \cite[Section~4]{Mi06} for general background in complex dynamics).
It is a standard fact that  $\Jul(g)\sub \CDach$ is a non-empty compact set  without isolated points \cite[Lemma~4.8 and Corollary~4.14]{Mi06}. As we remarked earlier, $\Jul(g)$ is actually  uniformly perfect (see \cite{Hi} and \cite{MdR}).  It is also {\em completely invariant} in the sense that   
\begin{equation}\label{eq: Jul_inv}
  g^{-1}(\Jul(g))
  =
  \Jul(g)=g(\Jul(g))  
\end{equation}
(see \cite[Lemma~4.3]{Mi06}).

If $W\sub \CDach$ is  a region  and $\widetilde W$ is  a component of 
$g^{-1}(W)$, then $\widetilde W$ is also a subregion of $\CDach$ and the map 
$g|_{\widetilde W}\: \widetilde W\ra W$ is
proper. In particular, $g(\widetilde W)= W$ (see 
\cite[Lemma~A.8]{BM}, where this statement is formulated in greater generality).

Now suppose  $\V^1$ is a finite  family of  subregions  of $\CDach$ that cover 
the Julia set $\Jul(g)$ of $g$.
We want to use the dynamics of $g$ under iteration to define an associated sequence  of finite covers of $\Jul(g)$. Here sets in 
$\V^1$ that do not meet $\Jul(g)$ are irrelevant. So we will use the term {\em admissible} for such a cover $\V^1$ if 
\[
V^1\cap \Jul(g)\ne \emptyset \text { for all $V^1\in \V^1$.}
\]
Given an admissible cover $\V^1$ of $\Jul(g)$ and $n\in \N$,   we  denote by $\V^n$ the \emph{pull-back of $\V^1$ under $g^{n-1}$}, that is,
\begin{equation}\label{eq: pullback_Vn}
\V^n:=\{\text{components $V^n$ of $g^{-(n-1)}(V^1)$}:\, V^1\in \V^1\}.
\end{equation}
If $V^n$ is such a component, then $g^{n-1}(V^n)=V^1\in \V^1$
by what we have seen above. 

Since $\Jul(g)$ is completely invariant, 
it  easily follows that if 
 $\V^1$ is an admissible finite cover of $\Jul(g)$, then for each $n\in \N$ the family $\V^n$ is also an admissible finite cover of $\Jul(g)$ by  subregions of $\CDach$. 

 Now the following fact is true.
 
\smallskip

 {\em Claim.} We have 
 \begin{equation}\label{eq:dyngV}
 g^k(V^{n+k})\in \V^n \text{ for all $n\in \N$, $k\in \N_0$, and $V^{n+k}\in 
 \V^{n+k}$.}
\end{equation}

\smallskip
Indeed, if $V^{n+k}\in \V^{n+k}$, then $V^{n+k}$ is a component of $g^{-(n+k-1)}(V^1)$ for some $V^1\in \V^1$. In particular,
$g^k(V^{n+k})$ is a connected set that is sent into  $V^1$
by $g^{n-1}$.
 Hence $g^k(V^{n+k})\sub V^n$ for some $V^n\in \V^n$. This in turn implies that the region $V^{n+k}$ is contained in a component
   $\widetilde {V}^{n+k}$ of $g^{-k}(V^n)$. Then 
 $g^k(\widetilde {V}^{n+k})=V^n$,  and so 
 $g^{n+k-1}(\widetilde {V}^{n+k})=g^{n-1}(V^n)=V^1$.
 Since $V^{n+k}\sub \widetilde {V}^{n+k}$ and $V^{n+k}$ is an entire component of $g^{-(n+k-1)}(V^1)$,  this is only possible 
 if $V^{n+k}=\widetilde {V}^{n+k}$.  Consequently, 
 $g^k(V^{n+k})= g^k(\widetilde{V}^{n+k})=V^n\in \V^n$, as claimed.

 \smallskip
 
 We now define 
\begin{equation}\label{eq: Xn_induced_by_Vn}
\X^n\coloneqq\{V^n\cap \Jul(g): V^n\in \V^n\}
\end{equation}
for $n\in \N$. According to this definition, if  $X^n\in \X^n$, then we can choose an element $V^n\in \V^n$ with $X^n=V^n\cap \Jul(g)$. We call each such a set $V^n$ an  
 {\em ambient} of $X^n$. Finally, we set $\X^0\coloneqq \{\Jul(g)\}$.

Note that  $\{\X^n\}$ is a sequence of finite covers of $\Jul(g)$ by non-empty relatively open subsets of $\Jul(g)$. We refer to $\{\X^n\}$ as the sequence of covers of $\Jul(g)$ {\em generated by $g$ and $\V^1$}.  
It follows immediately from  \eqref{eq:dyngV} and the complete invariance  of $\Jul(g)$ 
 that the sequence $\{\X^n\}$ is {\em dynamical} (see Section~\ref{subsec:dynqvisual}) in the sense that 
 \begin{equation}\label{eq:dyngX}
 g^k(X^{n+k})\in \X^n \text{ for all $n\in \N$, $k\in \N_0$, and $X^{n+k}\in 
 \X^{n+k}$.}
\end{equation}

Now the question naturally arises  whether this sequence $\{\X^n\}$  of covers of $\Jul(g)$ is a quasi-visual approximation of the Julia set $\Jul(g)$ equipped with the spherical metric $\sigma$. We will give an answer in the next subsection.

\subsection{Semi-hyperbolic rational maps and dynamical quasi-visual approximations of Julia sets} \hfill

We first need to discuss more standard terminology from complex dynamics.
Let $g\:\CDach \ra \CDach$ be  again a rational map with $\deg(g)\ge 2$.

A point  $c\in \CDach$ is called a {\em critical point} of $g$ if $g$ is not a local homeomorphism near $c$. 
A point $p \in  \CDach $ is  {\em periodic (under $g$)}
if  there exists $n\in \N$ such that $g^n(p)=p$. In this case,
the minimal number $n\in \N$ for which $g^n(p)=p$ is called the {\em period of $p$} and $(g^n)'(p)$ is referred to as  the {\em multiplier} of the {\em periodic cycle} $\{p,\, g(p), \dots,\, g^{n-1}(p)\}$. This definition of the multiplier only makes sense if the point $\infty$ does not belong  to the periodic cycle; if it does, the definition has to be adjusted (see \cite[p.~45]{Mi06}).  We say that 
$g$ has a {\em parabolic cycle} if there exists a periodic cycle of $g$ whose multiplier is a root of unity. 

Finally, a point $p\in \CDach$ is called {\em non-recurrent} if 
\[ 
p\not \in  \overline{\{g^n(p): n\in \N\}}.
\]

We can now define the special class of rational maps that will play a crucial role in this section. 

\begin{definition}[Semi-hyperbolic rational maps]\label{def:semi_hyp} A rational map $g\colon \CDach\to \CDach$ with $\deg(g)\ge 2$ is called \emph{semi-hyperbolic} if
$g$ has no parabolic cycles and every critical point of $g$ in $\Jul(g)$ is non-recurrent. 
 \end{definition}

 These maps admit various characterizations. 
 
\begin{proposition} \label{prop:semihyp} Let $g\:\CDach\ra \CDach$ be a rational map with 
$\deg(g)\ge 2$. Then the following conditions are equivalent:
 \begin{enumerate}
 \smallskip
  \item
    \label{item:sem1}
   $g$ is semi-hyperbolic.
\smallskip
    \item
    \label{item:sem2} There exist constants $r_0>0$ and $p\in \N$ such that 
    for all $w_0\in \Jul(g)$ and  $n\in \N$  we have 
    $\deg(g^n|_{U^n})\le p$
    for each component $U^n$ of $g^{-n}(B_\sigma(w_0, r_0))$.
 \smallskip   
  \item
    \label{item:sem3} There exist constants  $r_0>0$, $p\in \N$, $C\ge 1$, and $\lambda\in (0,1)$  such that for all $w_0\in \Jul(g)$ and $n\in \N$ we have  
    \[ \deg(g^n|_{U^n})\le p \text{ and } \diam_\sigma(U^n)<C\lambda^n\]
    for each component $U^n$ of $g^{-n}(B_\sigma(w_0, r_0))$. 
  \end{enumerate}
\end{proposition}

Essentially, this is well-known and goes back to \cite{Ma93}.
The argument there is rather involved. A  shorter   proof (that is somewhat difficult to read though) was given in 
\cite{CJY94} for polynomials, but  with small modifications the argument there is also valid   for general rational 
maps. 

We will  prove the following statement. 

\begin{theorem}\label{thm:semi_hyp_admits_approx}
Let $g\colon \CDach \to \CDach$ be a semi-hyperbolic rational map,
$\V^1$ be an admissible finite cover of $\Jul(g)$ by   subregions of $\CDach$, and 
$\{\X^n\}$ be the associated sequence of covers of $\Jul(g)$ generated by $g$ and  $\V^1$. 

Then if $\V^1$ is sufficiently fine, the sequence $\{\X^n\}$ is  a dynamical  quasi-visual approximation of width~$w=1$ for $(\Jul(g),\sigma)$. 
\end{theorem}

The condition that $\V^1$ is {\em sufficiently fine}  simply means that for the given map 
$g$ there exists $\delta_0>0$ such that the conclusion of the theorem is true 
whenever we have $\diam_\sigma(V^1)<\delta_0$ for all $V^1\in \V^1$. 

A proof of Theorem~\ref{thm:semi_hyp_admits_approx} will be given in Section~\ref{subsec:thm:semiapprox}. 
We do not know whether the theorem in this generality is also true for width $w=0$. It can be established with $w=0$ at least for
\emph{postcritically-finite} rational maps by using suitable  Markov partitions from \cite{CGZ22}.

For general (semi-hyperbolic) rational maps $g$, for example, when $\Jul(g)$ is totally disconnected, it is not so obvious how to define a natural dynamical sequence 
$\{\X^n\}$ of covers of $\Jul(g)$ from a given cover $\X^1$  in order to obtain a  quasi-visual approximation. 
This is the reason why we used the sets  $\V^n$ 
 of ambients.
 
 The situation improves if we impose a mild topological assumption on the Julia set, namely if we assume that $\Jul(g)$ is locally connected.  This  is true for various classes of rational maps, for example, for semi-hyperbolic rational maps $g$ with  connected Julia sets \cite{Mi11}.
 
In this case, 
 one can  start with a finite cover $\X^1$  of $\Jul(g)$ by connected and
relatively open (non-empty) subsets of $\Jul(g)$. For each $n\in \N$ one 
can then directly define the {\em pull-back $\X^n$ of $\X^1$ under $g^{n-1}$} as 
\[\X^n:=\{\text{components $X^n$ of $g^{-(n-1)}(X^1)$}:\, X^1\in \X^1\}.\]
In  this way, when the initial cover $\X^1$ is sufficiently fine, one can show that $\{\X^n\}$  is a quasi-visual approximation of $(\Jul(g),\sigma)$ with the dynamical property ~\eqref{eq:dyngX}, but we  will not pursue the details of this argument here. 

An appropriate reverse implication in Theorem~\ref{thm:semi_hyp_admits_approx} is also  true. Actually, we will prove the following statement.  

\begin{theorem}\label{thm:Adconverse}
Let $g\colon \CDach \to \CDach$ be  a rational map with $\deg(g)\ge 2$. If there exists 
 a dynamical quasi-visual approximation $\{ \X^n\}$ of $(\Jul(g),\sigma)$, then 
$g$ is semi-hyperbolic.  
\end{theorem}

A proof of this theorem will be given in Section~\ref{subsec:ProofThmCon}.  

If we combine Theorems~\ref{thm:semi_hyp_admits_approx} and~\ref{thm:Adconverse}, then we obtain the following immediate consequence.

\begin{cor}\label{cor:charsemi} Let $g\colon \CDach \to \CDach$ be  a rational map with $\deg(g)\ge 2$. 
  Then $g$ is semi-hyperbolic if and only if $(\Jul(g),\sigma)$ admits a dynamical  quasi-visual approximation $\{ \X^n\}$.   
\end{cor}

 \subsection{Auxiliary lemmas}\label{subsec: aux_lemmas} In order to prove Theorem~\ref{thm:semi_hyp_admits_approx},
 we need some  auxiliary facts, in particular a lemma about the control of the distortion of sets under proper holomorphic maps of uniformly bounded degree (see Lemma~\ref{lem:dist1}). 
 The proof of this statement is rather technical, and we will 
  present the details in the appendix.

In order to state the lemma, we use the term  {\em distortion function}  for   an increasing  function 
$\eta\: (0,\infty) \ra (0, \infty)$ such that $\eta(t)\to 0$ as $t\to 0^+$. 
We extend the definition of such a function to $0$ and $\infty$ by setting 
$\eta(0)=0$ and $\eta(\infty)=\infty$. This will allow us to include some trivial cases in the formulation of estimates.

Now the following fact is true. 

\begin{lemma} \label{lem:dist1} For each  $\eps>0$ and $p\in \N$ there exists a distortion function 
$\eta=\eta_{\eps,p}$ with  the following property: Let $\widetilde W\sub \CDach$
 be a simply connected  region with 
$\diam_\sigma(\widetilde W)<\pi/2$, $W\coloneqq B_\sigma(w_0, r)$, where $w_0\in \CDach$ and $0<r<\pi/4$,  
and $f\: \widetilde W\ra W$ be a proper holomorphic map with $\deg(f)\le p$.
Suppose $z_0\in \widetilde{W}$ satisfies $w_0=f(z_0)$ 
and  $A,A'\sub \widetilde W$ are sets such that:
\begin{enumerate}
\smallskip
\item\label{item:lem_dist1}  $z_0\in A\cap A'$ and  
$f(A)\cup f(A')\sub B_\sigma(w_0,r/2)$,
\smallskip
\item\label{item:lem_dist2}  $A$ is connected and  $\diam_\sigma (f(A'))\ge \eps r$. 
\end{enumerate}
\smallskip\noindent
Then we have 
\[
\frac{\diam_\sigma(A)}{\diam_\sigma(A')}\le \eta(\diam_\sigma(f(A))/r).
\]
\end{lemma}

As we mentioned, we will give a proof of this lemma in the appendix, but we can point out  some immediate consequences that will be important for us later on.
In the setting of the lemma we have $\diam_\sigma(f(A))\le r$ and so 
\[
 \frac{\diam_\sigma(A)}{\diam_\sigma(A')}\le  \eta(1)\eqqcolon C(\eps, p).
\] 
If $A'$ is also connected and $A$ satisfies $\diam_\sigma(f(A))\ge 
\eps r$, then the roles of $A$ and $A'$ are symmetric and we conclude that 
\[
 \diam_\sigma(A) \asymp \diam_\sigma(A')
 \] 
 with 
 $C(\asymp)=C(\eps, p)$.

Before we turn to the proof of Theorem~\ref{thm:semi_hyp_admits_approx}, we need one  more auxiliary fact that summarizes the relevant properties of semi-hyperbolic rational maps in a form convenient for us.

\begin{lemma} \label{lem:pullbnice} Let $g\: \CDach \ra \CDach $ be a semi-hyperbolic rational map. Then there exist $r_0\in (0,\pi/4)$ and $p\in \N$ with the following properties:

\begin{enumerate}
\smallskip
\item If $n\in \N_0$, $r\in (0,r_0]$, and $w_0\in \Jul(g)$, then each component $U^n$ of $g^{-n}(B_\sigma(w_0,r))$ is a simply connected  region with $\diam_\sigma (U^n)< \pi/2$ and  $\deg(g^n|_{U^n})\le p$. 
\smallskip
\item For each  $\eps>0$ there exists  $n_0\in \N$ independent of $w_0$  and $r$ such that for the components $U^n$  above we have 
$\diam_\sigma(U^n)<\eps$ whenever
 $n\ge n_0$. 
 
 Moreover, if $r\in (0,r_0]$ is sufficiently small dependent on 
 $\eps$, but independent of $w_0$, then $\diam_\sigma(U^n)<\eps$ for all 
 $n\in \N$. 
\end{enumerate}
\end{lemma}

\begin{proof} The statements about the degrees and diameters easily follow
from Proposition~\ref{prop:semihyp}~\ref{item:sem2} and we will skip the details. To see that 
the components $U^n$ are simply connected, one may assume that 
$r_0>0$ is so small that for all $w_0\in \Jul(g)$ the ball $B_\sigma(w_0, r_0)$ and all of its 
preimage components
under any iterate of $g$ contain at most one critical value of $g$. 

It is now  a standard fact  that if $f\: U\ra V$ is a proper holomorphic map between subregions $U$ and $V$ of $\CDach$ such that $V$ is simply connected and contains at most one critical value of $f$, then
$U$ is simply connected as well (this can easily be derived from the Riemann--Hurwitz formula, for example; for this formula, see~\cite[Theorem~A3.4]{Hu06}).  One uses this together with an induction argument  on $n$ to show that indeed all components $U^n$ of $g^{-n}(B_\sigma(w_0,r))$ for $w_0\in \Jul(g)$ and $r\in (0,r_0]$ are simply connected. Again we skip the details. 
\end{proof}

\subsection{Proof of Theorem~\ref{thm:semi_hyp_admits_approx}} 
\label{subsec:thm:semiapprox}
We can now prove the first  main result in this section.

Suppose $r_0\in(0,\pi/4)$ and $p\in \N$  are constants as in Lem\-ma~\ref{lem:pullbnice}
for the given semi-hyperbolic rational map $g$. 
Then the following condition is true:
if $n\in \N_0$, $w_0\in \Jul(g)$, $W\coloneqq 
B_\sigma(w_0, r_0)$, and $\widetilde W$ is a component of $g^{-n}(W)$, then 
$\widetilde W$ is a simply connected region with $\diam_\sigma(\widetilde W)<\pi/2$. 

Moreover,  $g^n|_{\widetilde W}\:  \widetilde W\ra W$ is a proper holomorphic map with $1\le \deg(g^n|_{\widetilde W})\le p$.  Hence all the basic assumptions in Lemma~\ref{lem:dist1} are satisfied for the map $g^n|_{\widetilde W}\:  \widetilde W\ra W$; so we can always apply the distortion estimates provided by this lemma for suitable sets $A$ and $A'$.

By Lemma~\ref{lem:pullbnice} we can find $0<\delta_0\le r_0$ such that if 
$\V^1$ is a  finite admissible cover of $\Jul(g)$ by  subregions of $\CDach$  
such that 
$\diam_\sigma(V^1)<\delta_0$  for all $V^1\in \V^1$,
then
\begin{equation}\label{eq:smallcovg}
\diam_\sigma(V)<r_0/4
\text{ for all $V\in \V^n$, $n\in \N$.}
\end{equation} 
Here, $\V^n$ denotes the pull-back of $\V^1$ under $g^{n-1}$ (see \eqref{eq: pullback_Vn}).
Note that this implies  
\begin{align}\label{eq:smallcovg2}
V\sub & \, B_\sigma(w_0, r_0/2)\sub B_\sigma(w_0, r_0)
 \text{} \\ 
&\text{ for all $n\in \N$,  $V\in \V^n$, and $w_0\in V$. 
}\notag
\end{align} 
 In the following, $\V^1$ is such a cover of $\Jul(g)$.

Let $\{\X^n\}$ be the sequence of covers of $\Jul(g)$ generated by $g$ and $\V^1$ (see \eqref{eq: Xn_induced_by_Vn}). 
There are finitely many elements in the sets $\X^1$ and $\X^2$ and they all have positive diameter, since they are non-empty and $\Jul(g)$ does not have 
isolated points (see \cite[Corollary~4.14]{Mi06}).  Hence we can find $0<\eps_0<1/3$ such that 
\begin{equation} \label{eq:diamXlar}
 \diam_\sigma(X)\ge \eps_0 r_0 \text{ 
for all $X\in \X^1\cup \X^2$.}
\end{equation}
Moreover, by making $\eps_0$ smaller if necessary, we may also assume that 
$2\eps_0 r_0$ is smaller than the {\em Lebesgue number} 
of the cover $\V^1$ (see \cite[Theorem~1-32]{HY}). This implies that 
\begin{align}\label{eq:Lebno}
\text{for all } 
&w_0\in \Jul(g) 
\text{ there exists $V^1\in \V^1$ such }\text{that}\\ 
\text{}&
\text{ $B_\sigma(w_0,\eps_0 r_0)\sub V^1$.}\notag
\end{align}

\smallskip
{\em Claim~1}. We have $\diam_\sigma(X^n)\asymp 
\diam_\sigma(V^n)$ for all $n\in \N$, $X^n\in \X^n$, where $V^n\in \V^n$ is any ambient of $X^n$, that is, $X^n=V^n\cap \Jul(g)$. 

\smallskip 
In this statement and its ensuing proof, all comparability constants  only  depend on the  $p$ and $\eps_0$.  

\smallskip
To prove the claim, let $n\in \N$ and $X^n\in \X^n$ be arbitrary, and $V^n\in 
\V^n$ be an ambient of $X^n$.  Then $X^n=
V^n\cap \Jul(g)\sub V^n$ and so $\diam_\sigma(X^n)\le 
\diam_\sigma(V^n)$. 

To obtain an inequality in the other direction, we choose a point $z_0\in X^n\sub \Jul(g)$ and define $w_0\coloneqq g^{n-1}(z_0)$,  $X^1\coloneqq g^{n-1}(X^n)$ and 
$V^1\coloneqq g^{n-1}(V^n)$.  Then $V^1\in \V^1$, $X^1=V^1\cap \Jul(g)\in \X^1$, and $w_0\in X^1$.  
 
Now let $W\coloneqq B_\sigma(w_0, r_0)$ and 
 $\widetilde W$ be the component of $g^{-(n-1)}(W)$ that contains $z_0$. Recall from \eqref{eq:smallcovg2} that $V^1\subset
 W$. 
  Then $z_0\in X^n\sub V^n\sub \widetilde W$, because the set $V^n\sub g^{-(n-1)}(V^1)\sub g^{-(n-1)}(W)$ is connected and  contains $z_0$. 
   
As we discussed previously, we can apply Lemma~\ref{lem:dist1}  to the map
\[ 
 f\coloneqq g^{n-1}|_{\widetilde W}\: \widetilde W\ra W.
 \] 
 We want to   do this with the sets $A\coloneqq V^n\sub  \widetilde W$ and $A'\coloneqq X^n\sub  \widetilde W$. Note that  
 \[
 \diam_\sigma(g^{n-1}(X^n))=\diam(X^1)\ge \eps_0r_0
 \]
 by \eqref{eq:diamXlar} and 
 \[
 g^{n-1}(X^n)\cup g^{n-1}(V^n)=V^1\sub B_\sigma(w_0, r_0/2)
 \]
 by \eqref{eq:smallcovg2}; thus the assumptions \ref{item:lem_dist1} and \ref{item:lem_dist2}
in  Lemma~\ref{lem:dist1} are satisfied for the chosen sets $A$ and $A'$.   
This lemma and \eqref{eq:smallcovg} now show that 
\[ 
\frac{\diam_\sigma(V^n)}{\diam_\sigma(X^n)}\le \eta(\diam_\sigma(V^1)/r_0)\le  \eta(1/4). 
\]
Here $\eta$ only depends on $p$ and $\eps_0$, but not on $n$ and $V^n$.
This shows that $\diam_\sigma(V^n)\lesssim \diam_\sigma(X^n)$ with a uniform constant $C(\lesssim)$ and the claim follows.

\smallskip
We now check the conditions  in Definition~\ref{def:qv_approx} for $\{\X^n\}$ to be a quasi-visual approximation of width $w\coloneqq1$ for $(\Jul(g), \sigma)$.  Since $\X^0=\{\Jul(g)\}$ and $\X^1$ is a finite cover of $\Jul(g)$ consisting of sets of positive diameter, it is sufficient to verify conditions \ref{item:qv_approx1}--\ref{item:qv_approx4} only for  $n\in \N$ and we do not have to worry about $n=0$ (see Remark~\ref{rem:X^0auto}).

Similar to the proof of Claim~1, the ensuing arguments will rely on specific applications of 
Lemma~\ref{lem:dist1}. 
 In all cases, we will choose a point $z_0\in \Jul(g)$, and for  suitable $n\in \N$ define
\begin{align}\label{eq:standdef}
w_0&\coloneqq g^{n-1}(z_0),\   W\coloneqq B_\sigma(w_0, r_0),  \\
\widetilde W&\sub \CDach \text{  to be the component of $g^{-(n-1)}(W)$ containing $z_0$,}\notag \\
f&\coloneqq  g^{n-1}|_{\widetilde W}\: \widetilde W\ra W.  \notag
\end{align} 
As before, $\eta$ will be the distortion function provided by this lemma for $p$ and $\eps_0$, and all  implicit constants will depend only on these two parameters.

 \smallskip
 \ref{item:qv_approx1} Let  $n\in \N$, and $X^n,Y^n\in \X^n$ with $X^n\cap Y^n \neq \emptyset$ be given, and  let   $V^n$, $\widehat V^n\in \V^n$ be  ambients of $X^n$, $Y^n$, respectively. We can then 
 choose a point $z_0\in X^n\cap Y^n\sub \Jul(g)$, and  define
 $w_0$, $W$, $\widetilde W$ as in \eqref{eq:standdef}. 
 
 Then $z_0\in V^n$ and $w_0 \in g^{n-1}(V^n)\in \V^1$. 
 This  implies that 
 \[ g^{n-1}(V^n)\sub B_\sigma(w_0, r_0/2)\subset W\] by \eqref{eq:smallcovg2} and so $V^n\sub \widetilde W$, because $V^n$ is connected. Moreover, $g^{n-1}(X^n)\in \X^1$ 
 and so by \eqref{eq:diamXlar} we have 
 \[
 \diam_\sigma(g^{n-1}(V^n))\ge \diam_\sigma(g^{n-1}(X^n))\ge \eps_0r_0.
 \]
Similarly, $\widehat{V}^n \sub \widetilde W$, $g^{n-1}(\widehat{V}^n)\sub B_\sigma(w_0, r_0/2)$,  
and 
$\diam_\sigma(g^{n-1}(\widehat{V}^n))\ge \eps_0 r_0$. 
 
 Then Claim~1,  Lemma~\ref{lem:dist1} and its  subsequent remarks (applied to $f=g^{n-1}|_{\widetilde W}$, $A=V^n$ and $A'=\widehat{V}^n$) imply
 that 
 \[
\diam_\sigma(X^n) \asymp \diam_\sigma(V^n)\asymp  \diam_\sigma(\widehat{V}^n)  \asymp \diam_\sigma(Y^n)
 \] 
 with  uniform constants $C(\asymp)$. Condition~\ref{item:qv_approx1} follows.

\smallskip
\ref{item:qv_approx2} To show the condition with  width $w=1$,
suppose  $n\in \N$,  $X^n,Y^n\in \X^n$ and $U_w(X^n)\cap U_w(Y^n) = \emptyset$. Let $V^n\in \V^n$ be an ambient 
of $X^n$. We can find points $x\in X^n$ and $y\in Y^n$ such 
\[
\sigma(x,y) \le 2\dist_\sigma(X^n, Y^n). 
\]  
Let  $\ga\sub \CDach$ be a spherical geodesic  joining $x$ and $y$. Then 
\[
\diam_\sigma(\ga)= \sigma(x,y)\le 2\dist_\sigma(X^n, Y^n). 
\]  
Define $x'\coloneqq g^{n-1}(x)\in \Jul(g)$,  $y'\coloneqq g^{n-1}(y)\in \Jul(g)$, and  $\ga'\coloneqq g^{n-1}(\ga)$. Note that 
$\ga'$ is a path starting at $x'$ and ending at $y'$. 
 
\smallskip 
{\em Claim~2.} $\ga'$ is not contained in  the ball $B_\sigma(x', \eps_0 r_0)$.  

\smallskip 
To establish  this claim,  we argue by contradiction and assume  that 
$\ga' \sub B_\sigma(x', \eps_0 r_0)$. Then by \eqref{eq:Lebno} we can find 
$\widehat{V}^1\in \V^1$ such that $\ga'\sub \widehat{V}^1\in \V^1$.
Now let  $\widehat{V}^n\in \V^n$ be the connected component of 
$g^{-(n-1)}(\widehat{V}^1)$ that contains $x$. Note that there is such a component, because $g^{n-1}(x)=x'\in \widehat{V}^1$. 
Since $\ga$ contains $x$,  is connected,  and 
\[ g^{n-1}(\ga)=\ga'\sub B_\sigma(x', \eps_0 r_0)\sub \widehat{V}^1,
\]
we have $\ga\sub \widehat{V}^n$. So if $\widehat{X}^n\coloneqq
\widehat{V}^n\cap \Jul(g)\in \X^n$, then $x,y\in \widehat{X}^n$. 
In particular, $x\in X^n\cap \widehat{X}^n$ and $y\in Y^n\cap \widehat{X}^n$. 
Hence $  X^n\cap \widehat{X}^n\ne \emptyset$ and 
$  Y^n\cap \widehat{X}^n\ne \emptyset$. This is a contradiction to our assumption that $U_w(X^n)\cap U_w(Y^n) = \emptyset$ for $w=1$.
The claim follows. 

\smallskip
By  Claim~2 we can find a  first point $u\in \ga$  such that 
$u'\coloneqq g^{n-1}(u)\not \in B_\sigma(x', \eps_0 r_0)$ when traveling from $x$ to $y$ along $\ga$. 
Then necessarily 
$\sigma(u', x')= \eps_0 r_0$ and for the subarc $\alpha$ of $\ga$ between 
$x$ and $u$ we have $g^{n-1}(\alpha)\sub \overline{B}_\sigma(x', \eps_0 r_0) $. 

 In order to apply 
 Lemma~\ref{lem:dist1}, we let $z_0\coloneqq x\in X^n =V^n \cap \Jul(g)$, and  define $w_0$, $W$, $\widetilde W$ as in \eqref{eq:standdef}. 
 Then $w_0=g^{n-1}(z_0)=x'\in g^{n-1}(V^n)\in \V^1$; so 
 $\diam_\sigma(g^{n-1}(V^n))\le r_0/4$ by \eqref{eq:smallcovg} and  $g^{n-1}(V^n)\sub B_\sigma(w_0, r_0/2)\sub W$ by \eqref{eq:smallcovg2}. The last inclusion implies that $V^n\sub \widetilde W$.

 Moreover,  by the choices above we have 
 \[g^{n-1}(\{x,u\})\sub g^{n-1}(\alpha)\sub  \overline{B}_\sigma(w_0, \eps_0 r_0) \sub 
  B_\sigma(w_0, r_0/2) \subset W.\]
  Since $\alpha$ is connected and contains $z_0=x$, this implies that 
  $\alpha$ is contained in  $\widetilde W$. In particular, $u\in \widetilde W$. 
  Note also that 
  \[
  \diam_\sigma(g^{n-1}(\{x,u\}))=\sigma(g^{n-1}(x), g^{n-1}(u))=\sigma(x',u')= \eps_0 r_0.
  \]

Lemma~\ref{lem:dist1} (applied to $f=g^{n-1}|_{\widetilde W}$, $A=V^n$, and $A'=\{x,u\}$) 
now shows that 
\begin{align*}
\frac{\diam_\sigma(X^n)}{\dist_\sigma(X^n, Y^n)}&\le 2 \frac{\diam_\sigma(V^n)}{\diam_\sigma(\ga)} \le 2 \frac{\diam_\sigma(V^n)}{\sigma(x,u)}\\
&\le 2 \eta(\diam_\sigma(g^{n-1}(V^n))/r_0)\le 2 \eta(1/4).
\end{align*}
Condition~\ref{item:qv_approx2}  follows.

\smallskip
\ref{item:qv_approx3} The proof is analogous to \ref{item:qv_approx1}. Namely, let   $n\in \N$,  $X^n\in  \X^n$, and $Y^{n+1}\in  \X^{n+1}$ with $X^n\cap Y^{n+1}\neq \emptyset$ be given. Suppose that 
$V^n\in \V^n$ and $\widehat{V}^{n+1}\in \V^{n+1}$ are ambients of $X^n$ and $Y^{n+1}$, respectively. We  choose a point $z_0\in X^n\cap Y^{n+1}$ and define $w_0$, $W$, $\widetilde W$ as in \eqref{eq:standdef}.
Then  $w_0\in g^{n-1}(V^n) \in \V^1$ and $w_0\in g^{n-1}(\widehat{V}^{n+1})\in \V^2$ and so
  \[g^{n-1}(V^n)\cup g^{n-1}(\widehat{V}^{n+1})\sub B_\sigma(w_0, r_0/2)\subset W\] by 
  \eqref{eq:smallcovg2}. In particular, this implies that $V^n\cup \widehat{V}^{n+1}\sub \widetilde W$. At the same time, we have $g^{n-1}(X^n) \in \X^1$ and $g^{n-1}(Y^{n+1})\in \X^2$, and thus
 \[
 \diam_\sigma(g^{n-1}(V^n))\ge  \eps_0r_0 \text{ and }
\diam_\sigma(g^{n-1}(\widehat{V}^{n+1}))\ge \eps_0 r_0 
\]
by \eqref{eq:diamXlar}.
 Then Claim~1,  Lemma~\ref{lem:dist1} and its  subsequent remarks (applied to $f=g^{n-1}|_{\widetilde W}$, $A=V^n$, and $A'=\widehat{V}^{n+1}$) show that  
 \[
\diam_\sigma(X^n) \asymp \diam_\sigma(V^n)\asymp  \diam_\sigma(\widehat{V}^{n+1})  \asymp \diam_\sigma(Y^{n+1})
 \] 
 with  uniform constants $C(\asymp)$. Condition~\ref{item:qv_approx3} follows.

\smallskip
\ref{item:qv_approx4} We can find $\delta_1\in (0,1/4)$ such that 
$\eta(\delta_1)\le 1/2$. Lemma~\ref{lem:pullbnice} in turn shows that we can find $k_0\in \N$ such that 
\[
\diam_\sigma(V^k)\le \delta_1r_0<r_0/4\text{ whenever $k\in \N$, $V^k\in \V^k$, and $k\ge k_0$.}
\]
 We will show that \ref{item:qv_approx4} holds with this constant $k_0$ and $\lambda=1/2$. 

To see this,  let  $n\in \N$, $X^n\in \X^n$ and $Y^{n+k_0}\in \X^{n+k_0}$ 
     with $X^n\cap Y^{n+k_0}\neq \emptyset$ be arbitrary.  We  choose a point $z_0\in X^n\cap Y^{n+k_0}\sub \Jul(g)$ and define $w_0$, $W$, $\widetilde W$ as in \eqref{eq:standdef}. 
Let   
$V^n$ be an ambient of $X^n$ and $\widehat{V}^{n+k_0}$ be an ambient of $Y^{n+k_0}$. Then $z_0\in V^n\cap \widehat{V}^{n+k_0}$.

We also have $w_0=g^{n-1}(z_0)\in g^{n-1}(V^n)\in \V^1$ and so 
\[g^{n-1}(X^n)\subset g^{n-1}(V^n)\sub B_\sigma(w_0, r_0/2) \subset W\] by \eqref{eq:smallcovg2}. This implies that $X^n\sub V^n\sub \widetilde W$, since $V^n$ is a connected set containing $z_0$.  Similarly, $z_0\in \widehat{V}^{n+k_0}$ and  \[g^{n-1}(\widehat{V}^{n+k_0})\sub B_\sigma(w_0, r_0/2)\subset W\] by \eqref{eq:smallcovg2}, which implies  $\widehat{V}^{n+k_0}\sub \widetilde W$.  Moreover, $g^{n-1}(X^n)\in \X^1$  and so by \eqref{eq:diamXlar} we have
\[
 \diam_{\sigma}(g^{n-1}(X^n)) \ge \eps_0 r_0.
\]
 Finally, $g^{n-1}(\widehat{V}^{n+k_0})\in \V^{k_0+1}$ and so $\diam_\sigma(g^{n-1}(\widehat{V}^{n+k_0}))\leq \delta_1r_0$ by the choice of $k_0$.

 Lemma~\ref{lem:dist1} (applied to $f=g^{n-1}|_{\widetilde W}$, $A=\widehat{V}^{n+k_0}$, and $A'=X^n$) now shows  that 
\begin{align*}
\frac{\diam_\sigma(Y^{n+k_0})}{\diam_\sigma(X^n)}
&\le \frac{\diam_\sigma(\widehat{V}^{n+k_0})}{\diam_\sigma(X^n)}\\
&\le 
\eta(\diam(g^{n-1}(\widehat{V}^{n+k_0}))/r_0)\le \eta(\delta_1)\le 1/2.
\end{align*}
Condition~\ref{item:qv_approx4} follows. 

As the argument above shows, the respective implicit and explicit constants for the inequalities in conditions~\ref{item:qv_approx1}–\ref{item:qv_approx4} are independent of the level $n$ and the tiles $X,Y$. With this, the proof of Theorem~\ref{thm:semi_hyp_admits_approx} is complete.

\begin{remark} \label{re:ambientqva} Claim~1  and a small adjustment in  the argument for \ref{item:qv_approx2} in the proof above  show that the sequence $\{\V^n\}$ of ambients actually satisfies metric conditions analogous to those of a quasi-visual approximation of a space.
\end{remark}

\subsection{Proof of Theorem~\ref{thm:Adconverse}}\label{subsec:ProofThmCon} 
First, we will briefly review some facts  about  
{\em normal families} of meromorphic functions in the complex plane (for general background on this topic, see \cite{Sch93}). 
Let $\{f_n\}$ be  a sequence of meromorphic functions $f_n\: \C\ra \CDach$. We say that this sequence converges {\em locally uniformly on $\C$} to a limit function $f\: \C\ra \CDach$, written 
\[ f_n\ra f \text{ locally uniformly on $\C$,}
\]
if for all $z_0\in \C$ there exists a neighborhood $U$ of $z_0$ in 
$\C$ such that we have uniform convergence $f_n\ra f$ on $U$, 
where we use the spherical metric in the target; so more precisely, we require that for all $\eps>0$ there exists $N\in \N$ such that 
\[
\sigma(f_n(z), f(z))<\eps
\]
for all $n\in \N$ with $n\ge N$ and all $z\in U$. 

It is well known that locally uniform convergence $f_n\ra f$ on $\C$  is equivalent to uniform convergence 
$f_n\ra f$ on each compact subset of $\C$. Moreover, since the  functions $f_n$ are meromorphic, the limit function $f$ is a meromorphic function as well, possibly equal to a constant $c\in \CDach$.

Let $\mathcal{F}$ be a family of meromorphic functions on $\C$. We say that 
the family is {\em normal} if every sequence in $\mathcal{F}$ has a subsequence that converges locally uniformly on $\C$ to some limit function (not necessarily contained in 
$\mathcal{F}$). 

We will need the following normality criterion. 

\begin{lemma}\label{lem:Zalc}
Let $\mathcal{F}$ be a family of meromorphic functions on $\C$ and  
$K\sub \CDach$ be a set that contains at least four points. Suppose for all $R>0$ there exists a distortion function $\eta_R$ 
such that 
\[
\sigma(f(u), f(v))\le \eta_R(|u-v|)
\]
whenever 
$f\in \mathcal{F}$ and $u,v\in D(0,R)\cap f^{-1}(K)$.  Then  $\mathcal{F}$ is a normal family.
\end{lemma}
Here and below, we  denote by  
\begin{equation}\label{eq:defEucldisk}
D(a,r)\coloneqq\{z\in \C: |z-a|<r\}
\end{equation}
the open Euclidean disk of radius $r>0$ centered at $a\in \C$. We will give the proof of the lemma in Section~\ref{subsec:mormfam} of the appendix. 

In passing,  we remark that the requirement that $K$ contains 
at least four points is sharp. Indeed, the family $\mathcal{F}$ consisting of the functions $f_n(z)=e^{nz}$, $n\in \N$, (trivially) satisfies the condition in the lemma  with $K=\{0,1,\infty\}$ for any distortion function, but it is not normal. 

We also need the following facts about  normal families.

\begin{lemma} \label{lem:GBd} Let  $\mathcal{F}$ be a normal family of meromorphic functions on 
$\C$ such  that
each locally uniform limit of a sequence of functions in  $\mathcal{F}$ is non-constant.   Then the following statements are true:

 \begin{enumerate}
 \smallskip
 \item
    \label{item:norm1} There exists
a radius  $r_0>0$ with the following property: if $f\in \mathcal{F}$ and $V$ is the component of $f^{-1}(B_\sigma(f(0), r_0))$ containing $0$, then 
$V\sub D(0,1)$.
\smallskip
  \item
    \label{item:norm2} For each $R>0$ there exists $d=d(R)\in \N$ such that 
\begin{equation}\label{eq:degreebd}
    \sum_{z\in f^{-1}(w)\cap D(0,R)} \deg(f, z)\le d
\end{equation}
for all $w\in \CDach$ and  $f\in \mathcal{F}$.
  \end{enumerate}
\end{lemma}

Recall that  $\deg(f,z)$ denotes the local degree of $f$ at $z$. Again,  we will give the proof in Section~\ref{subsec:mormfam} of the appendix.

The biholomorphisms $T$ of $\CDach$ that preserve the spherical metric are precisely the M\"obius transformations  that can be written in the form 
\begin{equation}\label{eq:rotsph}
T(z)=\frac{az+b}{-\overline{b}z+\overline{a}}
\end{equation}
with $a,b\in \C$ and $|a|^2+|b|^2>0$. These maps correspond to rotations of $\CDach$ as identified with the unit sphere in $\R^3$ under stereographic projection. Now for each $z_0\in \CDach$ 
we choose such a map $T$ with  $T(0)=z_0$. More specifically, for $z_0\in \C$ we define 
\[
T_{z_0}(z)\coloneqq \frac{z+z_0}{1-\overline{z_0}z}, 
\]
and we set $T_\infty(z) \coloneqq 1/z$. Then for each $z_0\in \CDach$  the M\"obius transformation $T_{z_0}$  is an isometry of $(\CDach, \sigma)$ with 
$T_{z_0}(0)=z_0$.

Now let  $g\colon \CDach\to\CDach$  be  a rational map with $\deg(g)\ge 2$, and suppose  that $\{\X^n\}$ is  a quasi-visual approximation of $(\Jul(g), \sigma)$ that is dynamical and so satisfies  \eqref{eq:dyngX}. 
Note that we do not assume that the covers $\X^n$ 
of $\Jul(g)$ for $n\in\N$ are finite.

We now use the function $g$ and its iterates to define a normal family.

\begin{lemma}\label{lem:FamG}  Let $\mathcal{G}$ consist of all meromorphic functions $f\:\C\ra \CDach$ of the form 
\begin{equation}\label{eq:deffG}
 z\in \C\mapsto    f(z)=(g^n\circ T_{z_0})(\rho_0 z), 
\end{equation}
where $n\in \N$ and $z_0\in X^{n+1}\in \X^{n+1}$ are arbitrary,
and $\rho_0\coloneqq \diam_\sigma(X^{n+1})$. Then 
$\mathcal{G}$ is a normal family. 

Moreover, each locally uniform  limit on $\C$ 
of a sequence of functions in  $\mathcal{G}$ is non-constant. 
\end{lemma}

Of course, $n$, $z_0\in X^{n+1}\in \X^{n+1}$, and $\rho_0$  in \eqref{eq:deffG} are different  from function to function. In particular, Lemma~\ref{lem:GBd} applies for the family 
$\mathcal{G}$. This will be important later on. 

\begin{proof} To establish the first statement,  we want to apply Lemma~\ref{lem:Zalc} with $K=\Jul(g)$ (which is an infinite set). 

Fix $R>0$ and let $f\in\mathcal{G}$ and $u,v\in D(0,R)\cap f^{-1}(K) $ be arbitrary.  Then there exist $n\in \N$, $X^{n+1}\in \X^{n+1}$, and $z_0\in X^{n+1}$ such that  $f$ has a representation as in \eqref{eq:deffG}. Note that $\rho_0=\diam_\sigma(X^{n+1})
\le \diam_\sigma(\CDach)=\pi$. It follows that 
on $D(0, \rho_0R)$ the Euclidean metric and the spherical metric are comparable up to a multiplicative constant $C(R)>0$ only depending on $R$.

Let $x\coloneqq T_{z_0}(\rho_0u)$
and $y\coloneqq T_{z_0}(\rho_0v)$. Since $T_{z_0}$ is an isometry with respect to the spherical metric and $T_{z_0}(0)=z_0$, it follows that 
\[
\sigma(z_0, x)=\sigma(0, \rho_0 u) \asymp \rho_0 |u|< R\rho_0,
\]
where $C(\asymp)=C(R)$. Hence  there exists a constant $R'\coloneq R\cdot C(R)>0$ such that $x\in B_\sigma(z_0, R'\rho_0)$. Similarly, $y\in B_\sigma(z_0, R'\rho_0)$.

Lemma~\ref{lem:dgxy} then shows that there are uniform  constants $\nu>0$ and  $C'=C'(R')=C'(R)>0$
such that
\begin{equation}\label{eq:fuvdist}
 \sigma(f(u), f(v))=\sigma(g^n(x), g^n(y))\le C'\cdot \biggl(
 \frac{\sigma(x,y)}{\rho_0}\biggr)^\nu.  
\end{equation}

Moreover, using again that $T_{z_0}$ is an isometry of $(\CDach, \sigma)$, we have 
\[
\sigma(x,y)/\rho_0=\sigma(\rho_0u, \rho_0 v)/\rho_0\asymp |\rho_0u-\rho_0v|/\rho_0=|u-v|
\]
with $C(\asymp)=C(R)$. This and \eqref{eq:fuvdist} show that there exists a  constant $\widetilde{C}=\widetilde{C}(R)>0$  (independent of $f$, $u$, and $v$) such that 
\[
\sigma(f(u), f(v))\le \widetilde{C}\cdot|u-v|^{\nu}.
\]
Now Lemma~\ref{lem:Zalc} implies that $\mathcal{G}$ is a normal family.

\smallskip

To establish the second statement, first note that if $n\in \N$ and 
$X^{n+1}\in \X^{n+1}$, then $g^{n}(X^{n+1})\in \X^1$ by~\eqref{eq:dyngX}, and so 
$\diam_\sigma(g^{n}(X^{n+1}))\gtrsim \diam_\sigma(\CDach)$ by condition~\ref{item:qv_approx3} in Definition~\ref{def:qv_approx} since $\X^0=\{\CDach\}$. In particular, if $z_0\in X^{n+1}$,
then we can find a point $z_1\in X^{n+1}$ such that 
\begin{equation*}
\sigma(z_0, z_1)\le \diam_\sigma(X^{n+1}) \text{ and }
\sigma(g^{n}(z_0), g^n(z_1))\ge \frac 13 \diam_\sigma(g^{n}(X^{n+1}))\ge c_0,
\end{equation*}
where $c_0>0$ is a uniform constant. 

Consequently, for $f$ as in \eqref{eq:deffG} and $w\coloneq T_{z_0}^{-1}(z_1)/\diam_\sigma(X^{n+1})$ we have
\begin{equation}\label{eq:fklargeim}
\sigma(0, w\cdot \diam_\sigma(X^{n+1}))\le \diam_\sigma(X^{n+1}) \; \text{ and }\; 
\sigma(f(0), f(w))\ge c_0.
\end{equation}

Now suppose $\{f_k\}$ is a sequence in $\mathcal{G}$ and $f_k\ra f$ locally uniformly on $\C$, where  $f\:\C\ra \CDach$ is some function.
Each $f_k$ can be written in the form \eqref{eq:deffG}. 

Let us first assume 
that the associated levels $n=n_k$ of the functions $f_k$ stay bounded along some subsequence. Then along this subsequence the 
corresponding  factors $\rho_0=\rho_{0,k}$  stay  uniformly bounded away from $0$ as follows from
condition~\ref{item:qv_approx3} in Definition~\ref{def:qv_approx}. It is then easy to see that the limit function $f$ has the form $f(z)=g^m(T(rz))$ for $z\in \C$ with 
$m\in \N$, $r>0$, and $T$  a M\"obius transformation as in   
\eqref{eq:rotsph}.  Hence $f$ is   non-constant.

So we may assume that for the associated levels $n_k$ we have 
$n_k\to \infty$ as $k\to \infty$. 
 Then by Lemma~\ref{lem:sub_shrink} we  may also assume that all the tiles $X^{n+1}=X^{n_k+1}$ involved in the definition 
of the functions $f_k$ are so small that on 
$\overline{B}_\sigma(0, \diam_\sigma(X^{n+1}))$ the spherical metric and the Euclidean metric are comparable up to a fixed uniform factor, say $<3$.
Then $z/\diam_\sigma(X^{n+1})\in D(0,3)$ for each 
$z\in \overline{B}_\sigma(0,\diam_\sigma(X^{n+1}))$.

If we now translate \eqref{eq:fklargeim} to the functions $f_k$, this implies  that for each $k\in \N$  there exists $w_k\in \C$  with $|w_k|<3$   such that $\sigma(f_k(0), f_k(w_k))\ge c_0$, where 
 $c_0>0$  is 
independent of $k$. Passing to subsequences if necessary, we may assume that
as $k\to \infty$ we have $w_k\to w_\infty\in \C$. Since $f_k\to f$ locally uniformly on $\C$, this implies that $f_k(0)\to f(0)$ and $f_k(w_k)\to  f(w_\infty)$. Then we have 
\[\sigma(f(0), f(w_\infty) )=\lim_{k\to \infty}
\sigma(f_k(0), f_k(w_k))\ge c_0>0.
\]
This shows that the limit function $f$ is non-constant, finishing the proof of the lemma.
\end{proof}

We are now ready to prove the main statement of this subsection. 

\begin{proof}[Proof of Theorem~\ref{thm:Adconverse}] Suppose that $\{\X^n\}$ is a dynamical quasi-visual approximation of $(\Jul(g),\sigma)$ for a rational map $g\colon \CDach\to\CDach$ with $\deg(g)\geq 2$, and let $\mathcal{G}$ be the normal family of meromorphic functions on $\C$ as in Lemma~\ref{lem:FamG}. To show that $g$ is semi-hyperbolic, it suffices to verify  condition~\ref{item:sem2} in Proposition~\ref{prop:semihyp}.   For $r_0>0$ in this condition we choose the radius for the family $\mathcal{G}$ as provided by 
Lemma~\ref{lem:GBd}~\ref{item:norm1}.

Now 
let  $w_0\in \Jul(g)$ be arbitrary and $U$ be a component of 
$g^{-n}(B_\sigma(w_0, r_0))$ for some $n\in \N$.  Then 
$g^n(U)=B_\sigma(w_0, r_0)$ and so  we can find 
$z_0\in  U$ with  $g^n(z_0)=w_0$. Since $w_0\in \Jul(g)$, we also have $z_0\in \Jul(g)$ by \eqref{eq: Jul_inv},  and so there exists 
$X^{n+1}\in \X^{n+1}$ with $z_0\in  X^{n+1}$.

Set  $\rho_0\coloneqq\diam_\sigma(X^{n+1})$, and let $f\colon \C\to \CDach$ be defined as in \eqref{eq:deffG}. Then $f\in \mathcal{G}$ and so by the choice of 
$r_0$ we have $V\sub D(0,1)$ for the component $V$ of $f^{-1}(B_\sigma(f(0),r_0))$ 
containing $0$. Note that $f(0)=g^n(z_0)=w_0$ and so 
$B_\sigma(f(0),r_0)=B_\sigma(w_0, r_0)$. 

The definition of $f$ shows that 
$T_{z_0}(\{\rho_0 u: u\in V \})=U.
$
This in turn implies  that the function $g^n$ attains each value in $B_\sigma(w_0, r_0)$ on $U$---in particular, the value $w_0=f(0)$---as often as $f$ attains this value on $V\sub D(0,1)$ (counting multiplicities). 
It follows that 
\begin{align*}
\deg(g^n|_U)&= \sum_{z\in g^{-n}(w_0)\cap U}\deg(g^n, z) 
\\&=\sum_{u\in f^{-1}(w_0)\cap V}\deg(f, u)\le
\sum_{u\in  f^{-1}(w_0)\cap D(0,1)}\deg(f, u).
\end{align*}
Now by Lemma~\ref{lem:GBd}~\ref{item:norm2} the last quantity is uniformly bounded from above independently of $f$, say by $p\in \N$.
We see that  condition~\ref{item:sem2} in Proposition~\ref{prop:semihyp} is indeed satisfied, and the statement follows. 
\end{proof}

\begin{remark} \label{rem:relHP} The idea to use a normal family argument  
to prove that a rational map is semi-hyperbolic is due to
Ha\"{i}ssinsky and Pilgrim.  
Lemma~\ref{lem:Zalc} and Lemma~\ref{lem:GBd}~\ref{item:norm2}
were essentially  extracted from  their work (see \cite[pp.~96--97]{HP09},
where  similar statements were used without being  explicitly formulated).

The general assumption that a given bounded metric space with a self-map has a dynamical quasi-visual approximation is in fact weaker than the requirement that the map is \emph{metric
CXC} as in the  general framework of \cite{HP09}. More precisely, given a metric CXC system, there exists a dynamical quasi-visual approximation (of width 1) for the associated \emph{repellor}. We discuss this in the setting of rational maps, though the analogous argument applies in the general setting (when replacing the Julia set with the repellor of the corresponding CXC system).

Let $g\colon \CDach\to \CDach$ be a rational map, $\V^1$ be an admissible finite cover of the Julia set $\Jul(g)$ by subregions of $\CDach$, and $\{\X^n\}$ be the associated dynamical sequence of covers of $\Jul(g)$ generated by $g$ and $\V^1$. Suppose that the tile graph $\Gamma\coloneq \Gamma(\{\X^n\})$ is Gromov hyperbolic and that the natural identification $\Phi\colon (\Jul(g), \sigma)\to (\partial_\infty\Gamma, d_\infty )$ is a quasisymmetry for some visual metric $d_\infty$ on $\partial_\infty\Gamma$. This is true whenever $g$ is metric CXC on its Julia set by \cite[Theorem~3.5.1]{HP09}. 

Our Theorem~\ref{thm:Gromov_implies_qv} implies that the sequence of covers $\{\mathbf{K}^n\}$ formed by the neighborhood clusters of tiles in $\{\X^n\}$ for some radius $r\in \N_0$ is a quasi-visual approximation of width $w=1$ for $(\Jul(g), \sigma)$. One can also check that $\{\mathbf{K}^n\}$ has the dynamical property~\eqref{eq:dyngXgen_1} for sufficiently large $n$. Accordingly, our Theorem~\ref{thm:Adconverse}
is stronger than the corresponding statement 
established by Ha\"{i}ssinsky--Pilgrim \cite[Proposition~4.2.9]{HP09}.

Moreover, the class of self-maps on bounded metric spaces admitting dynamical quasi-visual approximations is strictly larger than the one arising from metric CXC systems, because in our setting ``folding'' is allowed. For example, every invariant Jordan curve $\mathcal{C}$ for an \emph{expanding Thurston map} $f$ as in \cite[Theorem~15.1]{BM} admits a dynamical quasi-visual approximation. The restriction of $f$ to $\mathcal{C}$ may have folding and, consequently, does not have to be open. On the other hand, CXC systems are assumed to be finite branched covering maps on the underlying repellors and are always open maps (see \cite[Section~2.2 and Lemma~2.1.2]{HP09}).

\end{remark}

 \section{Appendix}

The purpose of this appendix is to provide the details for the proofs  of some facts used in Section~\ref{sec:semihyp}. 

 \subsection{Distortion properties of proper holomorphic maps}\label{subsec:AppA} In this subsection we will establish   Lemma~\ref{lem:dist1}.
In this lemma  sets  are measured in the spherical metric $\sigma$ on $\CDach$, but it is 
 difficult to obtain the desired distortion estimates with this metric directly. 
 Accordingly, we switch to the hyperbolic metric on suitable regions. We first review
 some relevant terminology and facts. 
 
 We say that a simply connected region $W\sub \CDach$ is {\em hyperbolic}
 if its complement in   $\CDach$ consists of at least two points. Then 
by the Riemann mapping theorem there exists a biholomorphism 
\[
\varphi\: W\ra \D\coloneqq \{u\in \C: |u|<1\}.
\]
 The hyperbolic metric  on $\D$ is given by the length element 
\[
ds=\frac{2|du|}{1-|u|^2},
\]
where $u$ is a variable in $\D$. 
The hyperbolic metric on $W$ is now the pull-back of the hyperbolic 
metric on $\D$ by the biholomorphism $\varphi\: W\ra \D$.  It has length element 
\[
ds=\frac{2|\varphi'(z)|}{1-|\varphi(z)|^2}|dz|,
\]
where $z$ is a variable in $W$. It is a standard fact that this only depends on $W$ and not on the choice of the  biholomorphism $\varphi\: W\ra \D$ (see \cite[\S2]{Mi06}).

In the following, we denote the hyperbolic metric on a simply connected hyperbolic region $W \sub \CDach$  by $\rho_W$. The subscript $W$ in metric notations will indicate 
 that they are with respect to $\rho_W$. For example, $B_W(a,r)$ denotes 
     the hyperbolic ball in $W$ of radius $r>0$  centered at $a\in W$. For emphasis 
     we write $\hdiam_W(A)$ for the diameter of a set $A\sub W$ with respect to the hyperbolic metric $\rho_W$. 

In the formulation of the following lemmas and their proofs  we will use the notion of a  distortion function as introduced before the statement of   
Lemma~\ref{lem:dist1}.   Note that if $\alpha\: (0,\infty) \ra (0, \infty)$ is any function such that 
$\alpha$ is bounded on each interval $(0,t_0]$ for $t_0>0$ and  
$\alpha(t)\to 0$ as $t\to 0^+$, then there exists a distortion function $\eta$ such that $\alpha(t)\le \eta(t)$ for all $t>0$. Indeed, we can simply define 
$\eta(t) \coloneqq\sup_{t'\in (0,t]}\alpha(t')$ for $t>0$. This fact is useful in the construction of distortion functions and we will use it repeatedly in the following without explicitly referring to it.

 \begin{lemma}\label{lem:blaschke} Let $p\in \N$ and  $h\: \D\ra \D$ be a 
 finite Blaschke product  of degree
 $1\le \deg(h)\le p$.  Then there exists a distortion function
 $\eta=\eta_p$ only depending on $p$ with the following property:
 if $A\sub \D$ is a connected set with $0\in A$ and $h(A)\sub \overline{B}_\D(0,R)$ for some $R>0$, then $A\sub \overline{B}_\D(0,\eta(R))$.
 \end{lemma}

 Here and in the ensuing proof $\overline{B}_\D(a, R)$ denotes the closed hyperbolic ball of radius $R>0$ in $\D$ centered at $a\in \D$ in accordance with our earlier convention. Recall also that a \emph{finite Blaschke product} of degree $d\in \N$ is a holomorphic self-map of the unit disk $\D$ given by a product of $d$ M\"obius transformations of $\D$ (see, for example, \cite[Problem~7-b]{Mi06}).

 \begin{proof} According to the definition of a Blaschke product, the map $h$ has the form 
 \[
 h(z)=e^{i\theta}\prod_{k=1}^d\frac{z-a_k}{1-\overline{a_k}z}, \quad z\in \D,
 \]
 where $1\le d\le p$, $\theta\in [0,2\pi)$,  and $a_1,\dots, a_d\in \D$.
For $k\in \{1, \dots, d\}$ define  
\[
\varphi_k(z)\coloneqq \frac{z-a_k}{1-\overline{a_k}z},  \quad z\in \D. 
\]

Now suppose $z_0\in \D$ and $h(z_0)\in \overline{B}_\D(0,R).$ Then $|h(z_0)|\le r$
where $r=r(R)\in (0,1)$ and $r(R)\to 0$ as $R\to 0$.    
It follows that  $|\varphi_{k_0}(z_0)|\le r^{1/d}\le r^{1/p}\in (0,1)$ for some $k_0\in \{1, \dots, d\}$, or equivalently, $\varphi_{k_0}(z_0)\in \overline{B}_\D(0, R')$, where $R'=R'(r,p)=R'(R,p)>0$ and $R'(R,p)\to 0$ as $R\to 0$. 
We conclude that  
\[
A\sub h^{-1}(h(A))\sub h^{-1}\big( \overline{B}_\D(0,R)\big)\sub \bigcup_{k=1}^d \varphi_k^{-1}\big( \overline{B}_\D(0,R')\big).
\]
Now $ \overline{B}_\D(0, R')$ has hyperbolic diameter  $2R'$.  
Since each M\"obius transformation $\varphi_k$ is a hyperbolic
isometry (see \cite[Lemma~2.7]{Mi06}), the previous
inclusion shows that the connected set $A$ can be covered by
$d\le p$ sets of hyperbolic diameter $2R'$. This implies
 that $A$ has hyperbolic diameter $\le 2dR'\le 2pR'$. Since 
$0\in A$, we see that $A\subset  \overline{B}_\D(0, 2pR')$,  and, because $R'=R'(R,p)\to 0$ as $R\to 0$, the  statement easily follows. \end{proof}

We translate the previous statement to a distortion property for proper maps between  simply connected hyperbolic regions. 

\begin{lemma}\label{lem:dist_estimates}
Let $p\in \N$, and   $W,\, \widetilde W\sub \CDach$
 be  simply connected hyperbolic regions.   
Suppose $f\: \widetilde W\ra W$ is a proper holomorphic map with $\deg(f)\le p$.
Then the following statements are true:
\begin{enumerate}
\smallskip
\item  \label{item:dist1} $\hdiam_W(f(A))\le \hdiam_{\widetilde W}(A)$ for each set 
$A\sub 
\widetilde W$.  
\smallskip
\item  \label{item:dist2} There exists a distortion function $\eta=\eta_p$ only depending on $p$  such that 
\[
\hdiam_{\widetilde W}(A) \le \eta\big(\hdiam_W(f(A))\big)
\]
for each  connected set 
$A\sub 
\widetilde W$. 
\end{enumerate}
\end{lemma}

The proof will show that in  \ref{item:dist2} we can take the same distortion function $\eta=\eta_p$ as in Lemma~\ref{lem:blaschke}.

\begin{proof}
  Statement  \ref{item:dist1} is an  immediate consequence 
  of the Schwarz--Pick lemma (see \cite[Theorem~2.11]{Mi06})
   and is true for all holomorphic maps $f\: \widetilde W\ra W$. 

To prove   \ref{item:dist2}, we may assume that   
\[
R\coloneqq 
 \hdiam_W(f(A))\in (0,\infty).
 \] Now let $a\in A$ be arbitrary. We  can then  choose two biholomorphisms
 \[   \widetilde \varphi \colon \widetilde W \to \D \text{ and }\varphi \colon  W \to \D 
  \text{ with }  
 \widetilde{\varphi} (a) = \varphi (f(a))=0.\]
 The composition $h\coloneqq  \varphi \circ f\circ\widetilde
   \varphi ^{-1}$ is a proper holomorphic self-map of $\D$ of
   (finite) degree $1\le d\coloneqq \deg(h)\le p$, and is  thus  a Blaschke product of degree $d$ (see \cite[Problem~15-c]{Mi06}). 
    
    Define $A'\coloneqq \widetilde \varphi(A)\subset \D$. Then $0=\widetilde \varphi(a)\in A'$ and 
    $A'$ is connected.  
    Moreover, we have
    $h(A')= \varphi (f(A))\subset \D$ and so $0= \varphi (f(a))\in h(A')$.      
   Since $ \varphi $ is a hyperbolic isometry, it follows that 
   \[R=\hdiam_W(f(A))=\hdiam_\D(h(A')),\] and so
  $h(A')\sub \overline{B}_\D(0,R)$. We conclude 
   $A'\sub \overline{B}_\D(0, \eta(R))$ with the function $\eta=\eta_p$ 
 as in Lemma~\ref{lem:blaschke}. Since $\widetilde  \varphi $ is also a 
 hyperbolic isometry, this in turn implies that 
 \[
 A={\widetilde  \varphi }^{-1}(A')\sub {\widetilde  \varphi }^{-1}(\overline{B}_\D(0,  \eta (R)))=\overline{B}_{\widetilde W}(a, \eta(R)).
 \] 
Now $a\in A$ was arbitrary, and it follows that $\hdiam_{\widetilde W} (A)\le 
 \eta(R)$, as desired. 
\end{proof}

  The following, essentially  well-known lemma states that  the  spherical metric $\sigma$ and the hyperbolic metric $\rho_W$ on a simply connected hyperbolic region 
   $W\subset \CDach$ are comparable when restricted to a
   relatively compact  subset of $W$. In order to obtain uniform comparability constants in our estimates, we 
   have to assume that the underlying region is not too large or rather that its complement in $\CDach$ is not too small. 
   A suitable assumption in this respect is that $W$ is contained in a hemisphere, that is,  a spherical disk of radius $\pi/2$. This in turn  is guaranteed by the condition $\diam_\sigma(W)<\pi/2$.  This  explains the corresponding assumptions on $W$ and $\widetilde W$ in Lemma~\ref{lem:dist1} and in the following statement.

\begin{lemma}\label{lem:hyp_vs_sph} Let  $W\subset \CDach$ be a simply connected hyperbolic region with   $\diam_\sigma(W) < \pi/2$, and
 $B\coloneqq B_W(w_0, R_0)$, where $w_0\in W$ and $R_0>0$.  Then the following statements are true:

\begin{enumerate}  
\smallskip
\item \label{item:hyp_vs_sph1} 
$\displaystyle \frac{\diam_\sigma(A)}{\diam_\sigma(B)} \asymp \hdiam_W(A)$
for all sets $A\sub  B$ with a constant $C(\asymp)=C(R_0)$.

\smallskip
\item \label{item:hyp_vs_sph2} If  $A\sub W$ is a set with $w_0\in A$  and $\hdiam_W(A)\ge \eps$, then 
 $\diam_\sigma(A)\gtrsim \diam_\sigma(B)$ with a constant $C(\gtrsim)=C(\eps, R_0)$. 
\end{enumerate} 
\end{lemma}

\begin{proof} Essentially, the statements follow from a spherical version 
of the Koebe distortion theorem that was explicitly formulated in 
\cite[Theorem~A.1]{BM}. To apply this result in our setting, we choose a biholomorphism $\varphi\: \D\ra W$ with $\varphi(0)=w_0$.
The assumption on the image domain   in  \cite[Theorem~A.1]{BM} is  then satisfied, because for our image domain $W$ we have $\diam_\sigma(W) < \pi/2$ and so $W$ is contained in a hemisphere 
of $\CDach$. Moreover, \cite[Theorem~A.1]{BM} is stated for the chordal metric on $\CDach$, but this is immaterial for our purposes, since the chordal metric and the spherical metric (as used by us) are comparable with universal constants.

Corresponding to the hyperbolic radius $R_0$ there exists 
a spherical radius $r_0=r_0(R_0)>0$ such that 
\[
B_\sigma(0, r_0)=B_\D(0, R_0)\subset \D.
\]
On $B_\D(0, R_0)$ the spherical and hyperbolic metrics are comparable 
and so
\[
\sigma(u,v)\asymp \rho_\D(u,v)
\]
for all $u,v\in B_\D(0, R_0)$. Here and in the following,  all comparability constants $C(\asymp)$ depend only on $r_0$, or equivalently on $R_0$, unless otherwise indicated. 
Since  $\varphi$ is a hyperbolic isometry, we have 
\[
B = B_W(w_0,R_0)=\varphi(B_\D(0, R_0))=  \varphi(B_\sigma(0, r_0)).
\]

We denote by 
\[
 \varphi^\sharp(z)\coloneqq \frac{1+|z|^2}{1+| \varphi(z)|^2}| \varphi'(z)|
\]
the spherical derivative of $ \varphi$ at $z\in \D$ (interpreted as a suitable limit if $\varphi(z)=\infty$). 
Then it follows from  \cite[Theorem A.1]{BM} that for all $u,v\in B_\sigma(0,r_0)=B_\D(0, R_0)$  we have
\begin{align*}
\sigma(\varphi(u), \varphi(v))&\asymp  \varphi^\sharp(0) \sigma(u,v)\\
&\asymp 
 \varphi^\sharp(0) \rho_\D(u,v)= \varphi^\sharp(0) \rho_W(\varphi(u), \varphi(v)).
\end{align*}

Translated to points in $B=B_W(w_0,R_0)=\varphi(B_\D(0,R_0))$ the previous relation implies  that 
\[
\sigma(z, w)\asymp  \varphi^\sharp(0) \rho_W(z,w)
\]
for all $z,w\in B$. 
Since $ \hdiam_W(B)=2R_0$, it follows that 
\[
\diam_\sigma(B)\asymp \varphi^\sharp(0), 
\]
and, putting this all together, we see that
\begin{equation}\label{eq:compsphhyp}
\sigma(z, w)\asymp  \diam_\sigma(B) \rho_W(z,w)
\end{equation}
for all $z,w\in B$. In particular, if $A$ is any subset of $B$, then
\[
\diam_\sigma(A)
\asymp  \diam_\sigma(B) \hdiam_W(A).
\]
Statement \ref{item:hyp_vs_sph1} follows. 

\smallskip
To prove \ref{item:hyp_vs_sph2}, consider a set $A\sub W$ with $w_0\in A$ and  $\hdiam_W(A)\ge \eps$.   If $A\sub B$, then 
 \ref{item:hyp_vs_sph1} implies that 
\[
 \diam_\sigma(A)\gtrsim \eps \diam_\sigma(B),
  \]
 and so the statement is true in this case. 
 
 If $A$ is not contained in $B$, then we can find a point $w_1\in A\setminus B$. 
It follows from 
\eqref{eq:compsphhyp} that there exists a constant $c_0=c_0(R_0)>0$ 
such that 
\[
B'\coloneqq B_\sigma(w_0, c_0 \diam_\sigma(B) )\sub B.
\]
Now we travel on a spherical geodesic $\ga$ from  $w_0\in A\cap B$ to $w_1\in A \setminus B$.
Since $\ga$ leaves $B$ it also leaves  $B'$ and we conclude that 
\[
\diam_\sigma(A)\ge \sigma(w_0,w_1)=\diam_\sigma(\ga)\ge c_0 \diam_\sigma(B) \gtrsim  \diam_\sigma(B).
\]
It follows that the statement is also true in this case (even with a constant independent of $\eps$). 
\end{proof}

 We can now prove our distortion lemma.

 \begin{proof}[Proof of Lemma~\ref{lem:dist1}] The basic idea of the proof is that Lemma~\ref{lem:dist_estimates} gives us good distortion control for the hyperbolic metric in source and target of the map $f$  and we can relate this to the spherical metric by Lemma~\ref{lem:hyp_vs_sph}.

In the given setting,  there exists a universal constant $R_0>0$ such that 
\[
B_\sigma(w_0, r/2)\sub B\coloneqq B_W(w_0, R_0)\sub 
  W=B_\sigma(w_0, r).
  \]
Then $\diam_\sigma(B)\asymp r$ with $C(\asymp)=2$, and by assumption \ref{item:lem_dist1} we have
 \[ f(A) \cup f(A')\sub B_\sigma(w_0, r/2)\sub B=B_W(w_0, R_0) \sub 
  W=B_\sigma(w_0, r).\] 
Since $\diam_\sigma(W)=2r<\pi/2$, Lemma~\ref{lem:hyp_vs_sph}~\ref{item:hyp_vs_sph1} implies that
 \begin{align*}
 \hdiam_W(f(A))&\asymp \diam_\sigma(f(A))/r \text { and }\\
 \hdiam_W(f(A'))&\asymp \diam_\sigma(f(A'))/r
 \end{align*}
 with the same universal constants $C(\asymp)=C(R_0)\eqqcolon C_0\ge 1$. We then conclude from Lemma~\ref{lem:dist_estimates}~\ref{item:dist1} and assumption~\ref{item:lem_dist2}
 that 
 \begin{equation} \label{eq:hdimA'}
 \hdiam_{\widetilde W} (A')\ge \hdiam_W(f(A'))\ge c_0 \diam_\sigma(f(A'))/r \ge c_0\eps, 
  \end{equation}
  where $c_0 \coloneqq 1/C_0 $.
 On the other hand, since $A$ is connected, by Lemma~\ref{lem:dist_estimates}~\ref{item:dist2}  we have
  \begin{align*}
 \hdiam_{\widetilde W} (A)&\le  \eta( \hdiam_W(f(A)))\\
 &\le
  \eta( c_0\diam_\sigma(f(A))/r)\le  \eta(c_0)\eqqcolon R_1, 
 \end{align*}
where   $\eta$ is a distortion function only depending on $p$. In particular,  $R_1$ only depends on $p$, and we have 
 $A\sub \widetilde {B}\coloneqq \overline{B}_{\widetilde W}(z_0, R_1)$, because $z_0\in A$ by  assumption~\ref{item:lem_dist1}. 
 Then Lemma~\ref{lem:hyp_vs_sph}~\ref{item:hyp_vs_sph1} implies that 
 \[
 \hdiam_{\widetilde W} (A)\asymp \diam_\sigma(A)/ \diam_\sigma(\widetilde {B})
 \]
 with $C(\asymp)=C(R_1)=C(p)$. 
 
 Moreover, we have $z_0\in A'$ and 
 $\hdiam_{\widetilde W} (A')\ge c_0 \eps$ by \eqref{eq:hdimA'}. 
 According to Lemma~\ref{lem:hyp_vs_sph}~\ref{item:hyp_vs_sph2} this  implies that 
 \[
 \diam_{\sigma} (A')\gtrsim \diam_\sigma(\widetilde {B})
 \]
 with $C(\gtrsim)=C(c_0\eps, R_1)=C(\eps, p)$.

  Collecting the above relations we see   that 
 \[
 \frac{\diam_\sigma(A)}{\diam_\sigma(A')}\lesssim  \frac{\diam_{\sigma}(A)}{\diam_{\sigma}(\widetilde {B})}\asymp \hdiam_{\widetilde {W}}(A) \le \eta(c_0 \diam_\sigma(f(A))/r).
 \]
 Since here the implicit constants only depend on $\eps$ and $p$, the constant $c_0$ is universal  and  the distortion function $\eta$ only depends  on $p$,    the statement  follows. 
 \end{proof}

\subsection{Some facts about normal families}\label{subsec:mormfam}
The following statement is a version of a well-known 
theorem due to Hurwitz.

\begin{lemma}\label{lem:Hurwitz}
Let $f\: \C\ra \CDach$ and $f_n\:\C\ra \CDach$ for $n\in \N$ be meromorphic functions. Suppose that $f$ is non-constant and that $f_n\ra f$ locally uniformly on $\C$.  Then the following statements are true:
\begin{enumerate}
\smallskip
    \item \label{i:Hur1}  Let  $w_0\in \CDach$ and suppose $D\sub \C$ is an open Euclidean disk such that $f(z)\ne w_0$ for all $z\in \partial D$. Then there exist 
    $\delta  >0$ and $N\in \N$ such that  
    \[
    \sum_{z\in f_n^{-1}(w)\cap D}\deg(f_n,z)= \sum_{z\in f^{-1}(w_0)\cap D}\deg(f,z)<\infty
    \]
for all $w\in B_\sigma(w_0,\delta)$ and  $n\in \N$ with $n\ge N$.

\smallskip
  \item \label{i:Hur2} If $a_0\in \C$, then there exist $N\in \N$ and points 
   $a_n\in \C$  defined for $n\in \N$ with  $n\ge N$  such that $f_n(a_n)=f(a_0)$   and  $a_n\to a_0$ as $n\to \infty$. 
\end{enumerate}
\end{lemma}

Since we were unable to find the statement in precisely this form in the literature, we will sketch the proof. We recall that $D(a,r)$ denotes
the open Euclidean disk of radius $r>0$ centered at $a\in \C$.

\begin{proof} 

\smallskip
 \ref{i:Hur1} We may assume that $w_0=0$. Let $z_1, \dots, z_k$ be all the points in $D$ where  $f$ attains the value $w_0=0$ 
 (note that the number of such  points is finite, 
 because  otherwise $f$ would be a  constant function 
 by the Uniqueness Principle (see
 \cite[p.~156]{Ga}). 
 We can find small open  Euclidean disks  $D_1, \dots, D_k\sub D$ centered at $z_1, \dots, z_k$, respectively, that are pairwise disjoint and 
 such that $|f|$ is small, but nonzero on each set $D_j\setminus\{z_j\}$,
 $j=1,\dots, k$. We may also assume that $f$ has no poles in the set 
 $\overline{D_1}\cup\dots \cup\overline{D}_k$. Since $f(z)\ne 0$ for $z\in \partial D$ by our hypotheses, 
 we can choose $\delta>0$ so small that  $|f(z)|>10\delta$ for $z\in 
 E\coloneqq \overline{D}\setminus (D_1\cup \dots \cup D_k)$.

 Now for sufficiently large $n$, say $n\ge N$, we may assume that 
 $f_n$ is so close to 
 $f$ that $f_n$ has no poles in
 $\overline{D_1}\cup\dots \cup\overline{D}_k$ and we have 
 $|f_n(z)-f(z)|<\delta$ for $z\in \overline{D_1}\cup\dots \cup\overline{D}_k$ and   $|f_n(z)|>9\delta$ for $z\in E$.
 
In particular, then $f_n$ does not attain any value  $w\in 
B_\sigma(0, \delta)\sub D(0,2\delta) $ on $E$. On the other hand, if $w\in B_\sigma(0, \delta)\sub D(0,2\delta)$ and  $z\in \partial D_j$,  then we have 
\[
|f_n(z)-f(z)-w|\le|f_n(z)-f(z)|+|w|  \le 3\delta <|f(z)|. 
\]
By Rouch\'e's Theorem (see \cite[p.~229]{Ga}), the functions  $f$ and $f+(f_n-f-w)=f_n-w$ have the same number of zeros in $D_j$,   counting multiplicities. In other
words, $f$ attains the value $w_0=0$ as often on $D_j$ as $f_n$  attains the value $w$ there.

 It follows that for $n\ge N$, 
\begin{align*}
 \sum_{z\in f_n^{-1}(w)\cap D}\deg(f_n,z)&=
 \sum_{j=1}^k\sum_{z\in f_n^{-1}(w)\cap D_j}\deg(f_n,z)
 =\sum_{j=1}^k\sum_{z\in f^{-1}(w_0)\cap D_j}\deg(f,z)\\
 &=\sum_{j=1}^k\deg(f,z_k) =\sum_{z\in f^{-1}(w_0)\cap D}\deg(f,z). 
\end{align*}
The statement follows. 

\smallskip
\ref{i:Hur2} Let  $w_0\coloneqq f(a_0)$. Since $f$ is non-constant, 
we can find a sequence of radii $r_k>0$ with the following property:
 $r_k\ra 0$ as $k\to \infty$ and if  $D_k\coloneqq  D(a_0, r_k)$, then 
 $f(z)\ne w_0$ for $z\in \partial D_k$ and $k\in \N$. Now by \ref{i:Hur1}, 
 for each $k\in \N$ and for each sufficiently large $n\in \N$ the function $f_n$ attains the value $w_0=f(a_0)$ in $D_k$. 
The statement easily follows. 
 \end{proof}

 We are now ready to establish   Lemmas~\ref{lem:Zalc} and~\ref{lem:GBd} from Section~\ref{subsec:ProofThmCon}.

 \begin{proof}[Proof of Lemma~\ref{lem:Zalc}] We argue by contradiction and assume that $\mathcal{F}$ is not a normal family. Then by Zalcman's lemma
 (see \cite[pp.~100--101]{Sch93} and \cite[p.~320]{Ga}) for each 
$n\in \N$ there exists  a point $z_n\in \C$, a number $r_n>0$, and a function  $f_n\in\mathcal{F}$ such that as $n\to \infty$ we have 
 $z_n\to z_\infty\in \C$, $r_n\to 
0$,  and 
\[
h_n(z)\coloneqq f_n(z_n+r_nz)\to h(z)
\]
locally uniformly for $z\in \C$ (with respect to the spherical metric in the target of the maps). Here $h$ is a non-constant meromorphic function on $\C$. 

By Picard's Theorem (see \cite[p.~57]{Sch93}) the function $h$ can omit at most two values in $\CDach$. Since $K$ contains at least four points, $h$ must attain two distinct values $w, w'\in K$. It follows that there exist points $a,a'\in \C$ such that 
$h(a)=w$ and $h(a')=w'$.  Then by Lemma~\ref{lem:Hurwitz}~\ref{i:Hur2},  for each sufficiently large $n$ the function $h_n$ will also attain the values $w$ and $w'$
at points near $a$ and $a'$, respectively. So by passing to a subsequence of $\{h_n\}$ if necessary, we may assume that  there exist sequences
$\{a_n\}$ and $\{a_n'\}$ in $\C$ such that $a_n\ra a$ and $a_n'\ra a'$ as $n\to \infty$ and such that $h_n(a_n)=w$ and $h_n(a_n')=w'$ for all $n\in \N$.

Define $u_n\coloneqq z_n+r_n a_n$ and $v_n\coloneqq z_n+r_n a'_n$.
Then $f_n(u_n)=h_n(a_n)=w\in K$ and $f_n(v_n)=h_n(a'_n)=w'\in K$. Moreover, $u_n\ra z_\infty$ and 
$v_n\ra z_\infty$ as $n\to \infty$. So if we choose 
$R>0$ large enough, then 
$u_n, v_n\in D(0, R)\cap f_n^{-1}(K)$ for all $n\in \N$. 

Our hypotheses imply that there exists a distortion function $\eta_R$  such that 
\[
0<\sigma(w,w')=\sigma(f_n(u_n), f_n(v_n))\le \eta_R (|u_n-v_n|)
\]
for all $n\in \N$. Since $u_n\ra z_\infty$ and $v_n\ra z_\infty$, the
right-hand side of the last inequality tends to $0$ as $n\to \infty$. This is absurd and the statement follows. 
\end{proof}

\begin{proof}[Proof of Lemma~\ref{lem:GBd}] 
  \ref{item:norm1} We argue by contradiction and assume the statement is not true. Then for each $n\in \N$ there exist a function
$f_n\in \mathcal{F}$ and a point $z_n\in\C$ with $|z_n|\ge 1$ such that $z_n$ belongs to the component $V_n$ of $f_n^{-1}(B_\sigma(f_n(0), 1/n))$ containing $0$. 

 By passing to a subsequence we may assume that 
$f_n\ra f$ locally uniformly on $\C$, where $f$  is a non-constant meromorphic function on $\C$. To reach a contradiction it suffices to show that for each $r\in (0,1)$ there exists a point $u\in \C$ with $|u|=r$ and $f(u)=f(0)$, because then $f$ is constant by the Uniqueness Principle.

To find such a point, recall that each component $V_n$ is connected and contains both $0$ and $z_n$ with $|z_n|\geq 1$. Hence, for each fixed $r\in(0,1)$, each $V_n$ must also contain a point $u_n$
with $|u_n|=r$.  Passing to a subsequence, we may assume that $u_n\to u\in \C$ as $n\to\infty$. Then $|u|=r$.  

Since $f_n\to f$ locally uniformly on $\C$, we have  $f_n(u_n)\to f(u)$ as  $n\to \infty$. On the other hand, 
we also have $f_n(u_n)\in f_n(V_n)=B_\sigma(f_n(0), 1/n)$. Since $f_n(0)\ra f(0)$,  this implies that $f_n(u_n)\to f(0)$ as $n\to \infty$. This shows that $f(u)=\lim_{n\to \infty} f_n(u_n)=f(0)$,  and we obtain the desired contradiction. 

\smallskip
 \ref{item:norm2}
We again argue by contradiction and assume that there exists $R>0$ such that the quantity in \eqref{eq:degreebd} is not uniformly bounded. Then for each $n\in \N$  there exist a function $f_n\in \mathcal{F}$ and a value $w_n\in \CDach$ such that 
\[
 \sum_{z\in f_n^{-1}(w_n)\cap D(0,R)} \deg(f_n, z)\ge n.
\]
Our hypotheses imply that by passing to a suitable subsequence, we may assume that as $n\to \infty$ we have $w_n\to w_0$  and $f_n\to f$ locally uniformly on $\C$, where $f$ is a non-constant meromorphic function.
Then we easily obtain a contradiction from the statement in Lemma~\ref{lem:Hurwitz}~\ref{i:Hur1}. 
\end{proof}


\begin{thebibliography}{BM17}

\bibitem[BM17]{BM} M.~Bonk and D.~Meyer,
\emph{Expanding Thurston Maps}, Math.\ Surveys and Monographs, Vol.\ 225, 
  Amer.\ Math.\ Soc., Providence, RI, 2017. 

\bibitem[BM20]{BM20} M.~Bonk and D.~Meyer,
\emph{Quasiconformal and geodesic trees}, Fund.\ Math. 250 (2020), 253--299.

\bibitem[BM22]{BM22} M.~Bonk and D.~Meyer,
  \emph{Uniformly branching trees}, Trans.\ Amer.\ Math.\ Soc. 375 (2022), 3841--3897.  

\bibitem[BS00]{BS00} M.~Bonk and O.~Schramm,   
{\em Embeddings of Gromov hyperbolic spaces}, 
{Geom.\ Funct.\ Anal.}
{10} (2000),  266--306. 

\bibitem[BP03]{BP03}
M.~Bourdon and H.~Pajot, \emph{Cohomologie $\ell_p$ et espaces de Besov}, J.\ Reine Angew.\ Math.\ 558 (2003), 85--108.

\bibitem[BH99]{BH99} M.R.~Bridson and A.~Haefliger,
\emph{Metric Spaces of Non-positive Curvature}, 
Grund.~math.~Wiss.~319. Springer, Berlin, 1999.

\bibitem[BBI01]{BBI01}
  D.~Burago, Dmitri, Y.~ Burago, and S.~Ivanov,
  \emph{A course in metric geometry},
  Grad. Stud. Math., 33,
  Amer.\ Math.\ Soc., Providence, RI, 2001. 

\bibitem[BS07]{BS}
S.~Buyalo and V.~Schroeder, \emph{Elements of Asymptotic Geometry}, Europ.\ Math.\ Soc., Z\"{u}rich, 2007.

\bibitem[Ca13]{Ca13} M.~Carrasco Piaggio, \emph{On the conformal gauge of a 
compact metric space},  Ann.\ Sci.\ \'Ec.\ Norm.\ Sup\'er.\ 46 (2013), 
495--548.

\bibitem[CJY94]{CJY94} L.~Carleson, P.W.~Jones, J.-C.~Yoccoz,
\emph{Julia and John},
Bol.\ Soc.\ Brasil.\ Mat.\ (N.S.) 25 (1994), 1--30. 

\bibitem[CGZ22]{CGZ22}
G.~Cui, Y.~Gao, and J.~Zeng,
\emph{Invariant graphs of rational maps}, Adv. Math. 404 (2022), Paper No. 108454.

\bibitem[Fo91]{Fo} O.~Forster, \emph{Lectures on Riemann
    Surfaces}, Grad. Texts in Math. 81, Springer, New York, 1991.

\bibitem[Ga01]{Ga} T.W.~Gamelin, \emph{Complex Analysis}, Springer, 2001.
  
\bibitem[GH90]{GH} E.~Ghys and P.~de la Harpe, Eds., \emph{Sur les groupes hyperboliques d'apr\`{e}s
Mikhael Gromov}, Progress in Math. 83, Birkh\"{a}user, Boston, 1990.  

\bibitem[Gr87]{Gr} M.~Gromov, \emph{Hyperbolic Groups}, in \emph{Essays in Group Theory (S. Gersten,
Ed.)}, MSRI Publ.\ 8, Springer, New York, 1987, 75--265.

\bibitem[HP09]{HP09} P.~Ha\"\i ssinsky and K.M.~Pilgrim, \emph{Coarse expanding 
conformal dynamics}, Ast\'erisque 325, 2009. 

\bibitem[He01]{He} J.~Heinonen, \emph{Lectures on Analysis on Metric Spaces},
  Springer, New York, 2001.

\bibitem[Hi93]{Hi} A.~Hinkkanen, \emph{Julia sets of rational
    functions are uniformly perfect}, Math. Proc. Cambridge Philos. Soc. 113 (1993), no. 3, 543-559.

\bibitem[HY98]{HY} J.G~Hocking and G.S.~ Young, 
\emph{Topology}, 2nd ed.,
Dover Publications, Inc., New York, 1988.

\bibitem[Hu06]{Hu06} J.H.~Hubbard, \emph{Teichm\"uller Theory and Applications to Geometry, Topology, 
and Dynamics.  Vol.~1.~Teichm\"uller Theory}, Matrix Editions, Ithaca, NY, 2006.

\bibitem[Ki20]{Ki20} J.~Kigami, \emph{Geometry and analysis of metric spaces via weighted partitions},  Lecture Notes in Math.\ 2265,  Springer, Cham, 2020.  
    

\bibitem[Ma08]{Ma08}  
  J.M.~Mackay,
  \emph{Existence of quasi-arcs},
  Proc.\ Amer.\ Math.\ Soc.\ 136 (2008), 3975--3981.
  
\bibitem[MT10]{MT10}
  J.M.~Mackay and J.T.~ Tyson, 
  \emph{Conformal dimension. Theory and application.} 
  University Lecture Series 54. Amer.\ Math.\ Soc.,
  Providence, RI, 2010. 

\bibitem[Ma93]{Ma93}  R.~Ma\~{n}e, 
\emph{On a theorem of Fatou}, 
Bol.\ Soc.\ Brasil.\ Mat.\ (N.S.) 24 (1993),  1--11. 

\bibitem[MdR92]{MdR} R.~Ma\~{n}e and L.F.~da Rocha, \emph{Julia
    sets are uniformly perfect}, Proc. AMS, 1992,  251--257.
  
\bibitem[Me02]{Me02} 
  D.~Meyer,
  \emph{Quasisymmetric embedding of self similar surfaces and origami with rational maps}, Ann.\ Acad.\ Sci.\ Fenn.\ Math.~27 (2002),
 461--484. 
 
\bibitem[Mi11]{Mi11} N.~Mihalache, 
\emph{Julia and John revisited}, 
Fund.\ Math.\ 215 (2011),  67--86. 

\bibitem[Mi06]{Mi06} J.~Milnor, \emph{Dynamics in One Complex
    Variable}, 3rd ed., Princeton Univ. Press, Princeton, NJ, 2006. 
 
\bibitem[Sch93]{Sch93} J.L.~Schiff, \emph{Normal Families},
Universitext, Springer, New York, 1993. 

\bibitem[Sch06]{Sch06} V.~Schroeder,
  \emph{Quasi-metric and metric spaces},
 Conform. Geom. Dyn. 10 (2006), 355--360. 

\end{thebibliography}
\end{document}